\documentclass[11pt,a4paper]{article}
\usepackage[utf8]{inputenc}
\DeclareUnicodeCharacter{2212}{-}
\usepackage[T1]{fontenc}
\usepackage[english]{babel}
\usepackage{csquotes}  
\usepackage{hyperref}
\usepackage{graphicx}
\usepackage{lmodern}
\usepackage{geometry}
\usepackage{amsthm}
\usepackage{mathtools}
\usepackage{indentfirst}
\usepackage{tabto}
\usepackage{color}
\usepackage{algorithm}
\usepackage{algorithmic}
\definecolor{bf}{rgb}{0,0,0.6} 
\usepackage{amssymb}
\usepackage{amsmath}
\usepackage{stmaryrd}
\usepackage{fullpage}
\usepackage{eso-pic}
\usepackage{float}
\usepackage{titlesec}
\usepackage{bigints}
\usepackage{fancyhdr}
\usepackage{mathrsfs}
\usepackage[shortlabels]{enumitem}
\usepackage{scalerel,stackengine}
\usepackage{yhmath}
\usepackage{verbatim}
\usepackage{authblk}
\usepackage{needspace}
\usepackage{graphicx}
\usepackage{subcaption}
\usepackage{bbold}
\usepackage[section]{placeins}
\usepackage[disable]{todonotes}
\setuptodonotes{line, color=blue!30}
\numberwithin{equation}{section}

\newcommand{\todoJules}[1]{\todo[color=green!30,line]{#1}} 

\usepackage[maxnames=99, maxcitenames=3,backend=bibtex]{biblatex}
\bibliography{bibliography}
\AtEveryBibitem{\clearfield{url} \clearfield{month} \clearfield{issn} \clearfield{doi}}
\newcommand{\jules}[1]{\textcolor{red}{#1}}

\definecolor{navy}{rgb}{0.2,0.2,0.698}

\newcommand{\dd}{\mathrm{d}} 
\newenvironment{proofNoQED}[1][Proof]{%
	\par\noindent\textit{#1.} 
}{%
	\par
}

\hypersetup{
	colorlinks=true,
	citecolor=green,
	linkcolor=blue,
}

\newtheorem{dft}{Definition}[section]
\newtheorem{te}[dft]{Theorem}
\newtheorem{rem}[dft]{Remark}
\newtheorem{prop}[dft]{Proposition}
\newtheorem{cor}[dft]{Corollary}
\newtheorem{lemm}[dft]{Lemma}
\newtheorem{ex}[dft]{Example}

\title{An inverse problem in cell dynamics: Recovering an initial distribution of telomere lengths from measurements of senescence~times}

\author{Jules Olayé$^{\dagger}$}
\date{}
\begin{document}
	
\maketitle	\let\thefootnote\relax\footnotetext{$^{\dagger}$ Institut de Mathématiques de Toulouse, CNRS, Université de Toulouse, 31062 Toulouse, France.
}
\let\thefootnote\relax\footnotetext{\hspace{2.525mm}Mail: \href{mailto:jules.olaye@math.univ-toulouse.fr}{jules.olaye@math.univ-toulouse.fr}}
	\begin{abstract}
		Telomeres are repetitive sequences situated at both ends of the chromosomes of eukaryotic cells. At each cell division, they are eroded until they reach a critical length that triggers a state in which the cell stops to divide: the senescent state. In this work, we are interested in the link between the initial distribution of telomere lengths and the distribution of senescence times. We propose a method to retrieve the initial distribution of telomere lengths, using only measurements of senescence times. Our approach relies on approximating our models with transport equations, which provide natural estimators for the initial telomere lengths distribution. We investigate this method from a theoretical point of view by providing bounds on the errors of our estimators, pointwise and in all Lebesgue spaces. We also illustrate it with estimations on simulations, and discuss its limitations related to the curse of dimensionality. 
	\end{abstract}
	\section{Introduction}
	\paragraph{Biological motivation.} Linear chromosomes of eukaryotic cells have repeated sequences of nucleotides called telomeres at each of their two ends. These regions  are non-coding, and prevent fusion between chromosomes. At each cell division, for each chromosome, one of its two telomeres is slightly shortened because the enzyme responsible for the DNA replication is unable to copy the last nucleotides of the~DNA. This phenomenon is called the \textit{end-replication problem}~\hbox{\cite{wynford-thomas_end-replication_1997}}. As telomeres progressively shorten, a cell may eventually reach a critical threshold in which telomeres are not long enough to protect the coding regions of DNA from degradation. To prevent this issue, when the shortest telomere of a cell attains a certain length, which is around $27$ base pairs for the yeast~\cite{rat_mathematical_2025}, the cell enters a state called \textit{senescence}~\hbox{\cite{abdallah_2009,bourgeron_2015,Martin2021,Xu2013}}. This state is characterised by the fact that the cell stops dividing, thus avoiding the loss of coding~DNA. Understanding this phenomenon may help to understand how cancer cells can emerge. The main reason is that cancer cells have mutated to be able to lengthen their telomeres, preventing them from senescence and allowing them to proliferate~\cite{robinson_telomerase_2022}. 
	
	\paragraph{Presentation of the problem.} Since the criterion for a cell to become senescent is given by the lengths of its telomeres, initial lengths distribution and fate of cell lineages/population are deeply connected. In~\cite{rat_mathematical_2025}, this link was studied with the following approach: how do the initial telomere lengths influence the fate of cell lineages/populations? Specifically, the authors investigated how the parameters of the initial distribution affect the population half-life time. In the current paper, we aim at understanding this link with the opposite approach. To do so, we propose to solve the following inverse problem: recovering the initial distribution of telomere lengths when only the senescence times distribution of cell lineages is observed.

	\paragraph{Informal description of the models.} To address the issue presented above, we use two deterministic telomere shortening models, whose dynamics are inspired from the stochastic models developed in recent years~\hbox{\cite{arino_mathematical_1995,benetos_stochastic_2025,benetos_branching_2025,bourgeron_2015,Martin2021,olaye_long-time_2026,rat_mathematical_2025}}. These models represent the mean value of telomere lengths concentration in cell lineages. They can be derived from stochastic telomere shortening models by proceeding as in the proof of~\cite[Theorem~$2.2.7$]{olaye_thesis_2025}, in which a deterministic telomere shortening model is derived from a stochastic population model by computing the mean value of telomere lengths concentration. The first model corresponds to a toy model for mathematical investigations, representing the evolution of cell lineages with only one telomere, see~\eqref{eq:PDE_model_telomeres_one_telo}. In this model, cells are structured according to the length of their telomere~$x\in\mathbb{R}_+$. The second model is a more biologically relevant one, and represents cell lineages with several telomeres, see~\eqref{eq:PDE_model_telomeres_several_telos}. Each chromosome in a cell has two telomeres, one at each end. Thus, denoting by $k\in\mathbb{N}^*$ the number of chromosomes of the species we study, cells are structured by the lengths of their telomeres~\hbox{$x\in\mathbb{R}_+^{2k}$}. In each model, the quantities modelled by our equations are the telomere lengths density over time, and the senescence times distribution. We assume that cells divide at a constant rate, and that at each division, the telomere (in the model on~$\mathbb{R}_+$) or one telomere by chromosome~(in the model on~$\mathbb{R}_+^{2k}$) is shortened. We model senescence by a cemetery state, and assume that cells enter this state when one of their telomeres has a length below a threshold. This threshold is set to $0$ for simplicity, as any other threshold value is equivalent to~$0$ by translation of the trait space. These dynamics give us a system of two integro-differential equations, including the deterministic equation of a jump process as a first equation, and the rate at which one of its coordinates reaches $0$ as a second~equation.

	\paragraph{Difficulty.} The main difficulty in solving this problem is that we work with integro-differential equations, which are non-local equations. This non-locality implies that we lose information about the telomere lengths distribution when only the senescence times are observed. Indeed, the models involve an integral over possible telomere lengths at the entry into senescence, see the second lines of~\eqref{eq:PDE_model_telomeres_one_telo} and~\eqref{eq:PDE_model_telomeres_several_telos}. Thus, the value of the telomere length that triggers senescence is not observable in both models. This loss of information implies that the inverse of the operator linking the initial telomere lengths distribution and the senescence times distribution, if it exists, cannot be continuous, see Remark~\ref{rem:ill_posed_inverse_problem_one_dimension}. The inverse problem is thus ill-posed in the sense of Hadamard or Tikhonov, see~\cite[Definitions~$2.1.1$ and~$2.1.3$]{Kabanikhin2011}. This also implies that it is difficult to prove that the operator mentioned above is invertible.

	\paragraph{Estimation strategy.} The estimation strategy we use consists in adapting the approach developed in~\hbox{\cite{armiento_estimation_2016,doumic_asymptotic_2025}} to our setting. This approach allows us to retrieve a well-posed inverse problem, at the cost of an approximation. First, we assume that the telomere shortening values are small and that the cell division rate is high. This is biologically relevant because the average telomere shortening value is much smaller than the initial telomere length, see Section~\ref{subsect:discussion_models_assumptions}. Then, we obtain approximations of our models by transport equations with an absorption into a cemetery state at the boundary of $\mathbb{R}_+$ or~$\mathbb{R}_+^{2k}$. Finally, we use the method of characteristics to construct estimators of the initial distribution on these approximated models. 

	\paragraph{Review of the literature.} Inverse problems in cell biology have been widely studied over the last few years~\cite{armiento_estimation_2016,bourgeron_estimating_2014,doumic_estimating_2013,doumic_asymptotic_2025,olaye_estimation_2025}. Among all these works, our article is in line with the following works in protein depolymerisation~\cite{armiento_estimation_2016,doumic_asymptotic_2025}. Precisely, as in~these articles, we have a trait that progressively degrades over time, and our aim is to retrieve the initial distribution of the trait. In~\cite{doumic_asymptotic_2025}, the approximation used is a transport-diffusion equation (second-order approximation) instead of a transport equation (first-order approximation). The present work is thus more in line with~\cite{armiento_estimation_2016} in which a first-order approximation is done, justified later in~\cite{doumic_asymptotic_2025}. In our case, a second-order approximation  can also be done to construct an estimator for the model in $\mathbb{R}_+$. This study has been conducted by the author, and will be presented in an upcoming article. 
	
	Integro-differential equations such as those described by our models, see~\eqref{eq:PDE_model_telomeres_one_telo} and \eqref{eq:PDE_model_telomeres_several_telos}, have been well-studied in the literature~\cite{barles_concentration_2009,ellefsen_theory_2021,jeddi_convergence_2024,makroglou_integral_2003}, and are often referred to as~\textit{Volterra} or \textit{Lotka-Volterra equations}. However, the theoretical study of mathematical models representing telomere shortening is very recent, and still quite unexplored. The main articles we can cite are~\cite{benetos_stochastic_2025,benetos_branching_2025} or \cite{olaye_long-time_2026}. In these works, the question of the existence of a stationary profile was addressed. Previously, studies were conducted from a biomathematical or numerical perspective~\hbox{\cite{bourgeron_2015,eugene_effects_2017,Martin2021,portillo_influence_2023,rat_mathematical_2025,wattis_mathematical_2020,Xu2013}}. In some of these works, model approximations have been studied and justified from a modelling perspective~\hbox{\cite{portillo_influence_2023,wattis_mathematical_2020}}. However, to the best of our knowledge, theoretical results on the approximation of telomere shortening models have not yet been~obtained. 

	\paragraph{Main contributions.} 
	Our first contribution is to provide qualitative bounds for the approximation errors of our models and the errors of our estimators, pointwise and in all Lebesgue spaces. Specifically, we obtain bounds on the approximation and estimation errors that decrease exponentially fast after a certain length and a certain time. This completes the approximation result obtained in~\cite{doumic_asymptotic_2025}, in a discrete model similar to ours. In their work, a result in $L^2$-norm for the approximation error was obtained with a bound increasing linearly with the time. Our main idea to obtain these new bounds is to rewrite the approximation errors of the telomere lengths densities as the integral of several sub-errors, see~\eqref{eq:step2_keylemma}. This integral representation allows us to highlight for all $0\leq s \leq t$ how the error generated at time~$s$ still influences the error at time~$t$. Then, thanks to a maximum principle, we show that this error dissipates exponentially~fast. 

	Our second contribution concerns the extension of the estimation to a multidimensional setting, i.e., to the case where cells have several telomeres. From a theoretical point of view, we show that the estimator is less straightforward to obtain, as it requires a careful control over the rate at which cells remain outside the cemetery. From a practical point of view, we highlight the difficulties related to the curse of dimensionality and extreme value theory. 



	\paragraph{Organisation of the paper.} First, in Section~\ref{sect:models_and_results}, we present the models and the estimators we use, and state the main result of the article. Then, we prove the main result for the model with one telomere and the model with several telomeres in Sections~\ref{sect:model_one_telomere} and~\ref{sect:model_several_telomeres} respectively. Thereafter, we illustrate in Section~\ref{sect:estimation_on_simulations} the quality of our estimators on simulations and data. Finally, we discuss the limits and the prospects of our work in Section~\ref{sect:discussion_estimation}. Auxiliary statements are presented in the Appendix, as well as the proof of certain results given during the paper. 
	
	\section{Presentation of the models and the main result}\label{sect:models_and_results}
	
	This section is devoted to the presentation of the notations and the results of the paper. First, in Section~\ref{subsect:models}, we define the two models that we study. Then, in Section~\ref{subsect:main_results}, we introduce our estimators, and give the assumptions and the main result of this work. Finally, in Section~\ref{subsect:discussion_models_assumptions}, we explain our choice of studying a model with only one telomere, and discuss our~assumptions.

	\subsection{The models}\label{subsect:models}
	
	Our goal is to have models representing the microfluidic experiments presented in~\cite{xu_two_2015}, as our data are derived from them, see Section~\ref{subsect:estimation_real_datas}. In these experiments, cell lineages are tracked over time in microcavities until their last cell becomes senescent. Biologists then observe the senescence times of each lineage. We therefore need models which describe the telomere lengths distribution over time in lineages, and the time at which their last cell becomes senescent. 
	
	In both of our models, we assume that cell division times are distributed according to an exponential law with parameter $\tilde{b} > 0$. We consider $\tilde{\delta} >0$ the maximum shortening value, and \hbox{$\tilde{g} : [0,\tilde{\delta}]  \rightarrow \mathbb{R}_+$} a probability density function representing the distribution of telomere shortening values at each division. We assume that $\tilde{g}$ has finite first and second moments, and introduce the following constants and functions 
	$$
	\forall i\in\{1,2\}:\hspace{1.5mm} \tilde{m}_i := \int_0^{\tilde{\delta}}u^i \tilde{g}(u) \dd u, \hspace{4mm} \text{ and } \hspace{4mm} \forall x\in\mathbb{R}_+:\hspace{1.5mm}  \tilde{G}(x) := \int_0^{\min(x,\tilde{\delta})} \tilde{g}(s) \dd s.
	$$
	We also consider a non-negative function $n_0\in L^{1}\left(\mathbb{R}_+\right)$ verifying \hbox{$||n_0||_{L^1\left(\mathbb{R}_+\right)} = 1$}. This function represents the initial telomere length distribution in the model with one telomere, and the initial length distribution for each individual telomere in the model with several telomeres. The goal of this paper is, thus, to infer this initial distribution knowing the distribution of senescence~times. 

	\paragraph{Model with one telomere.} In the first model, we assume that each cell has only one telomere. This is not biologically realistic because a cell has always several telomeres, since each chromosome of a cell has two of them. However, this model helps us to gain mathematical intuition, and it is discussed in Section~\ref{subsect:discussion_models_assumptions} how this model can be used in practice.
	
	Cells are structured according to their telomere length that belongs to $\mathbb{R}_+$. We consider the density of telomere with length $x \geq 0$ at time $t \geq 0$ when several lineages are tracked, denoted~$n^{(1)}(t,x)$. We also consider the rate at which cell lineages enters senescence at time~$t\geq0$, denoted~$n_{\partial}^{(1)}(t)$. In fact, $n_{\partial}^{(1)}$ can also be seen as the senescence times distribution by Remark~\ref{rem:conservation_individual_one_telo} below. In this model, at each cell division, the telomere of the dividing cell is shortened by a random value distributed according to~$\tilde{g}$. If after this shortening, its telomere length is still greater than~$0$, then the cell remains in the dynamics. Otherwise, the cell goes to the cemetery, i.e. becomes senescent. Thus, the probability that the length~$y \geq 0$ of the telomere of the dividing cell becomes negative after shortening is~$\int_{\min(y,\tilde{\delta})}^{\tilde{\delta}} \tilde{g}(v) \dd v = 1 - \tilde{G}(y)$. Recalling that $n_0$ is the initial telomere length distribution, we then consider the following system of integro-differential~equations
	\begin{equation}\label{eq:PDE_model_telomeres_one_telo}
		\begin{cases}
			\partial_t n^{(1)}(t,x) = \tilde{b}\int_0^{\tilde{\delta}}n^{(1)}\left(t,x + v\right)\tilde{g}(v) \mathrm{d}v - \tilde{b}.n^{(1)}(t,x), & \forall t\geq0,\,x\geq0,\\
			n^{(1)}_{\partial}(t) = \tilde{b}\int_0^{\tilde{\delta}} n^{(1)}(t,y)(1-\tilde{G}(y))\mathrm{d}y ,& \forall t\geq0, \\ 
			n^{(1)}(0,x) = n_0(x), & \forall x\geq 0.
		\end{cases}
	\end{equation}	
	This system corresponds to the above description, and is our first model for telomere shortening. When we study this model, our aim is to estimate~$n_0$ when $n^{(1)}_{\partial}$ is observed.
	\begin{rem}\label{rem:conservation_individual_one_telo}
	By integrating both lines of~\eqref{eq:PDE_model_telomeres_one_telo}, we have conservation of the number of individuals, i.e. that for all $t\geq0$ 
	$$
	\begin{aligned}
	\int_{0}^{+\infty} n^{(1)}(t,x) \mathrm{d}x + \int_0^{t} n_{\partial}^{(1)}(s) \dd s= \int_0^{\infty} n^{(1)}(0,x) \dd x = 1.
	\end{aligned}
	$$
	The above yields that $\int_0^{\infty} n_{\partial}^{(1)}(s) \dd s = 1$ by letting $t$ tend to infinity, and that $n_{\partial}^{(1)}(t) \leq \tilde{b}$ for all~$t\geq0$ by the second line of~\eqref{eq:PDE_model_telomeres_one_telo}. The fact that $n_{\partial}^{(1)}$ is bounded by $\tilde{b}$ comes from the fact that the senescence rate cannot be higher than the division rate.
	\end{rem}
	\begin{rem}\label{rem:ill_posed_inverse_problem_one_dimension}
	 Let us denote the following sets, 	endowed respectively with the $L^1$-norm, and the sum of the $L^1$-norm and the $L^\infty$-norm,
	 	$$
	 	\begin{aligned}
	 	\mathcal{F} &= \left\{f\in L^1\left(\mathbb{R}_+\right)\,\big|\,f\geq0,\,||f||_{L^1\left(\mathbb{R}_+\right)} = 1\right\}, \\
	 	\mathcal{F}_{\partial} &= \left\{f_{\partial}\in L^1\left(\mathbb{R}_+\right)\cap L^\infty\left(\mathbb{R}_+\right)\,\big|\,f_{\partial}\geq0,\,||f_{\partial}||_{L^1\left(\mathbb{R}_+\right)} = 1,\,||f_{\partial}||_{L^\infty\left(\mathbb{R}_+\right)} \leq \tilde{b}\right\}.
	 	\end{aligned}
	 	$$
	  	We also consider $\Psi_{\partial} : \mathcal{F} \rightarrow \mathcal{F}_{\partial}$ the operator such that for all \hbox{$f\in L^{1}\left(\mathbb{R}_+\right)$}, $\Psi_{\partial}(f)$ corresponds to the cemetery $n_{\partial}^{(1)}$ of~\eqref{eq:PDE_model_telomeres_one_telo} when $n_0 = f$. This operator is well-posed by Proposition~\ref{prop:well_definition_general_model} and Remark~\ref{rem:conservation_individual_one_telo}. In view of Proposition~\ref{prop:non_continuity_inverse_cemetery} with $\tilde{\xi} = \tilde{b} g(y)\dd y$, if $\left(\Psi_{\partial}\right)^{-1}$ exists, then it is not continuous (even if its domain is restricted on $\Psi_{\partial}\left(\mathcal{F}\right)$). As a consequence, the inverse problem of estimating $n_0$ from~$n_{\partial}^{(1)}$ is ill-posed in the sense of Hadamard and Tikhonov, see~\cite[Definitions~$2.1.1$ and $2.1.3$]{Kabanikhin2011}. This difficulty is what motivates us to obtain a model approximation.
	\end{rem}
	\paragraph{Model with $2k$ telomeres.} We now create a model which takes into account the fact that a cell has several telomeres. We denote by $k\in\mathbb{N}^*$ the number of chromosomes of the species we study. Each chromosome has~\hbox{$2$ ends}, so each cell has $2k$ telomeres. Hence, we use the space $\mathbb{R}_+^{2k}$ to represent telomere lengths in our second model. For all $i\in \llbracket 1, k \rrbracket$, the $i$-th and the \hbox{$(i+k)$-th} coordinates of a vector in $\mathbb{R}_+^{2k}$ represent telomeres on the same chromosome (the chromosome~$i$). \needspace{\baselineskip}\noindent At each cell division, for all $i\in\llbracket 1,k \rrbracket$, we have the following biological constraint:
	
	\begin{itemize}[label = \fontsize{6}{12}\selectfont\textbullet,leftmargin=0.8cm]
		\item Either the telomere linked to the coordinate $i$ is shortened and the telomere linked to the coordinate $i+k$ is unchanged, with probability $1/2$,
		
		\item Or the telomere linked to the coordinate $i+k$ is shortened and the telomere linked to the coordinate $i$ is unchanged, with probability $1/2$.
	\end{itemize}
	To take into account this, we introduce the following set
	\begin{equation}\label{eq:set_shortening}
		\mathcal{I}_k = \left\{I\in \mathcal{P}\left(\llbracket1,2k\rrbracket\right)  \left.\big|\right. \# I = k,\,\forall (i,j)\in I^2:\,i\neq j\Longrightarrow i \,\text{mod}\, k \neq j\,\text{mod}\, k\right\},
	\end{equation}
	where $\mathcal{P}\left(\llbracket1,2k\rrbracket\right)$ is the power set of  $\llbracket1,2k\rrbracket$, i.e. the set that contains all subsets of~$\llbracket1,2k\rrbracket$.
	The set $\mathcal{I}_k$ contains all the possible combinations of telomeres that can be shortened at each division. In particular, the condition in its definition represents the fact that only one end of the chromosome is shortened at each division. When a cell divides, we draw one set $I\in\mathcal{I}_k$ uniformly to know the indices where there is a shortening for the daughter cell we follow. 
	
	\begin{ex}\label{ex:shortening_set}
		When $k = 2$, we have
		$$
		\mathcal{I}_2 = \{\{1,2\},\{1,4\},\{2,3\},\{3,4\}\}.
		$$
		If at a cell division, we draw $\{1,4\}\in\mathcal{I}_2$ (probability $1/4$) as a set where indices are shortened, then there is a shortening in the coordinates $1$ and $4$, and the coordinates~$2$ and~$3$ stay unchanged.
	\end{ex}
	We also introduce the following measure for all $I \in \mathcal{I}_k$
	\begin{equation}\label{eq:measure_shortening_by_set}
		\mathrm{d}\tilde{\mu}^{(I)}(v) := \left[\prod_{i\in I}\tilde{g}\left(v_{i}\right)\mathrm{d}v_{i}\right]\left[\prod_{i\in\llbracket1,2k\rrbracket\backslash I}\delta_0(\mathrm{d}v_{i})\right].
	\end{equation}
	This measure represents the shortening distribution at each division, knowing the fact that the shortened telomeres are those in $I$. Dirac measures represent the fact that the coordinates in $\llbracket1,2k\rrbracket\backslash I$ are unchanged, and the measures $\left(\tilde{g}(v_i) \mathrm{d}v_i\right)_{i\in I}$ that the coordinates in~$I$ are shortened.

	Finally, we consider
	\begin{equation}\label{eq:measure_shortening}
		\mathrm{d}\tilde{\mu}(v) := \frac{1}{2^k}\sum_{I \in \mathcal{I}_k} \mathrm{d}\tilde{\mu}^{(I)}(v).
	\end{equation}
	This measure represents how telomeres are shortened at each division, taking into account all the possible combinations of telomeres that can be shortened at a division. The term $2^k$ comes from the fact that~$\#(\mathcal{I}_k) = 2^k$, see Lemma~\ref{lemm:zero_singleton_cardinal}. 
	
	We denote by $n^{(2k)}(t,x)$ the density of telomere with lengths $x\in\mathbb{R}_+^{2k}$ at time $t\geq 0$. We also denote by $n_{\partial}^{(2k)}(t)$ the density of senescence times at time $t\geq0$. We assume that at each division, the telomeres of the dividing cell are shortened by a random value distributed according to $\tilde{\mu}$, and that it becomes senescent if one of its telomere has a length below~$0$. Then, the probability that a cell with length $y\in\mathbb{R}_+^{2k}$ becomes senescent after division is $\tilde{\mu}\left(\left\{v\in\mathbb{R}_+^{2k}\,|\,y-v \notin \mathbb{R}_+^{2k}\right\}\right)$. We also assume that the initial lengths distribution is $n^{(2k)}(0,x) = \prod_{i = 1}^{2k}n_0(x_i)$, and refer to Remark~\ref{rem:initial_lengths_assumptions} for more details about this choice. From the above description, we have that our second model is the following system
	
	\begin{equation}\label{eq:PDE_model_telomeres_several_telos}
		\begin{cases}
			\partial_t n^{(2k)}(t,x) = \tilde{b}\int_{v\in\mathbb{R}_+^{2k}} n^{(2k)}\left(t,x + v\right)  \mathrm{d}\tilde{\mu}(v) - \tilde{b}.n^{(2k)}(t,x), & \forall t\geq0,\,x\in\mathbb{R}_+^{2k},\\
			n_{\partial}^{(2k)}(t) = \tilde{b}\int_{y\in\mathbb{R}_+^{2k}} n^{(2k)}(t,y)\tilde{\mu}\left(\left\{v\in\mathbb{R}_+^{2k}\,|\,y-v \notin \mathbb{R}_+^{2k}\right\}\right)\mathrm{d}y,  & \forall t\geq0, \\ 
			n^{(2k)}(0,x) = \prod_{i = 1}^{2k}n_0(x_i), &\forall x\in\mathbb{R}_+^{2k}.
		\end{cases}
	\end{equation}
	\noindent Again, our aim is to estimate $n_0$ when~$n^{(2k)}_{\partial}$ is observed. 
	\begin{rem}\label{rem:initial_lengths_assumptions}
		Our assumption that $n^{(2k)}(0,x) = \prod_{i = 1}^{2k}n_0(x_i)$ for all $x\in\mathbb{R}_+^{2k}$ is the same as in~\cite{Martin2021,rat_mathematical_2025,Xu2013}, and allows us to simplify computations. However, there are currently no theoretical or simulation-based results  to support that this assumption is biologically relevant. Further studies will be conducted in this~regard. 
	\end{rem} 
	\begin{rem}\label{rem:conservation_individual_several_telos}
	By adapting the computations done in Remark~\ref{rem:conservation_individual_one_telo}, we also have conservation of the number of individuals for this model, and that $n_{\partial}^{(2k)}$ is bounded by $\tilde{b}$. In particular, the following equalities hold, for all~$t\geq0$,
		\begin{equation}\label{eq:link_density_lengths_cemetery}
		\begin{aligned}
		\int_{x\in\mathbb{R}_+^{2k}} n^{(2k)}(t,x) \mathrm{d}x + \int_0^{t} n_{\partial}^{(2k)}(s) \dd s= 1, \hspace{2mm}\text{ and }\hspace{2mm}\int_{x\in\mathbb{R}_+^{2k}} n^{(2k)}(t,x) \mathrm{d}x = 	\int_t^{+\infty}n_{\partial}^{(2k)}(s) \dd s.
		\end{aligned}
		\end{equation}
\end{rem}
\begin{rem}\label{rem:ill_posed_inverse_problem_several_dimension}
For the same reasons as those presented in Remark~\ref{rem:ill_posed_inverse_problem_one_dimension}, retrieving $n_0$ from $n_{\partial}^{(2k)}$ is an ill-posed inverse problem in the sense of Hadamard and Tikhonov, see~\cite[Definitions~$2.1.1$ and $2.1.3$]{Kabanikhin2011}.
\end{rem}

\subsection{Assumptions and main result}\label{subsect:main_results}
The main result of this work is the construction of estimators for the initial distribution of telomere lengths in both models, with their respective error of estimation. We begin by presenting the main hypothesis of this work and some notations. 
\paragraph{Assumptions.}\hypertarget{assumption:assumption}{} 
Our main result assumptions are the following. They are verified for example when $n_0$ is the density of an Erlang distribution. We study this case in detail in  Section~\ref{sect:estimation_on_simulations}.

\begin{enumerate}[start=1,label={$(H\textsubscript{\arabic*}):$},leftmargin=1.5cm]
	
	\item \hypertarget{assumption:H1}{} There exist $N >0$ (large), $\delta > 0$, $g : [0,\delta] \rightarrow \mathbb{R}_+$ a probability density function and $b >0$ such that
	$$
	\tilde{\delta} = \frac{\delta}{N}, \hspace{8mm}\forall x \in [0,\tilde{\delta}]:\,\tilde{g}(x) = Ng(Nx),\hspace{4mm}\text{ and }\hspace{4mm}\tilde{b} = bN. 
	$$

	\item \hypertarget{assumption:H2}{} It holds $n_0\in W^{2,1}\left(\mathbb{R}_+\right)$, and there exist $\lambda > 0$, $C_{\lambda}> 0$ and $C'_{\lambda}> 0$ such that for all~$x \in \mathbb{R}_+$
	$$
	\left|n''_0\left(x\right)\right| \leq C_{\lambda}\exp\left(-\lambda x\right),\hspace{2.5mm}\text{ and }\hspace{2.5mm}\left|n'_0\left(x\right)\right| \leq C'_{\lambda}\exp\left(-\lambda x\right).
	$$
	
	\item \hypertarget{assumption:H3}{} There exists $D_{\lambda} \geq 1$ such that for all $x \in \mathbb{R}_+$
	$$
	n_0\left(x\right) \leq D_{\lambda}\lambda \exp\left(-\lambda x\right).
	$$
	
	\item \hypertarget{assumption:H4}{} There exist $\omega \geq \lambda$, $f_{\omega} : \mathbb{R}_+ \rightarrow \mathbb{R}_+$ non-decreasing, and $D_{\omega} \in\left(0,1\right]$ such that for all~$x\geq 0$ 
	$$
	n_0(x) \geq \frac{D_{\omega}f_{\omega}(x)\exp\left(-\omega x\right)}{\int_{0}^{+\infty}f_{\omega}(y)\exp\left(-\omega y\right) \dd y}.
	$$
	
\end{enumerate}
\nopagebreak 
The above assumptions are further discussed in Section~\ref{subsect:discussion_models_assumptions}. Let us only emphasise one crucial point: in~\hyperlink{assumption:H1}{$(H_1)$}, $N$ is assumed to be large. Then, when a result is stated, one needs to imagine that $N\rightarrow+\infty$ to understand it.
\paragraph{Notations.} First, we denote the equivalents of~$\left(\tilde{m}_i\right)_{i\in\{1,2\}}$ and~$\tilde{G}$ for $g$: 
$$
\forall i\in\{1,2\}:\hspace{1.5mm} m_i := \int_0^{\delta} u^ig(u)\dd u, \hspace{4mm} \text{ and } \hspace{4mm} \forall x\in\mathbb{R}_+:\hspace{1.5mm}  G(x) := \int_0^{\min(x,\delta)} g(s) \dd s.
$$
Then, we introduce the equivalents of $\left(\tilde{\mu}^{(I)}\right)_{I\in\mathcal{I}_k}$ and $\tilde{\mu}$ for $g$:
\begin{equation}\label{eq:scaled_measures_shortening_and_by_set}
	\begin{aligned}
		\forall I \in \mathcal{I}_k:\hspace{1.5mm}\mathrm{d}\mu^{(I)}(v) := \left[\prod_{i\in I}g\left(v_{i}\right)\mathrm{d}v_{i}\right]\left[\prod_{i\in \llbracket1,2k\rrbracket\backslash I}\delta_0(\mathrm{d}v_{i})\right], \hspace{3.75mm}\text{ and } \hspace{3.75mm}\mathrm{d}\mu(v) := \frac{1}{2^k}\sum_{I \in \mathcal{I}_k} \mathrm{d}\mu^{(I)}(v).
	\end{aligned}
\end{equation}
Thereafter, for all $d\in\mathbb{N}^*$ and $\xi$ finite measure on $\mathbb{R}_+^d$, we consider $\mathcal{L}(\xi)$ the Laplace transform of~$\xi$, defined for all $p\in\mathbb{C}$ such that $\text{Re}(p) >0$ as:
\begin{equation}\label{eq:definition_laplace_transform}
\mathcal{L}(\xi)(p) := \int_{u\in\mathbb{R}_+^d} \exp\left(-p\sum_{j = 1}^d u_{j}\right) \xi(\dd u) .
\end{equation}
Finally, we define the following three constants:
\begin{equation}\label{eq:approximation_eigenvalues}
\begin{aligned}
	\lambda_N := \frac{N}{m_1}\left[1-\mathcal{L}(g)\left(\frac{\lambda}{N}\right)\right]&, \hspace{10mm}\lambda'_N := \frac{N}{km_1}\left[1-\left(\mathcal{L}(g)\left(\frac{\lambda}{N}\right)\right)^k\right], \\
	&\hspace{-12.5mm}\omega'_N := \frac{N}{km_1}\left[1-\left(\mathcal{L}(g)\left(\frac{\omega}{N}\right)\right)^k\right].
\end{aligned}
\end{equation}
In view of the fact that $\mathcal{L}(g)(0) = \int_0^{\delta} g(u) \dd u = 1$ and the definition of the derivative, we have that $\lim_{N \rightarrow +\infty} \lambda_N = -\frac{\lambda\left(\mathcal{L}(g)\right)'(0)}{m_1} = \lambda$. By the formula of the derivative of the power of a function, we also have that $\lim_{N \rightarrow +\infty} \lambda'_N = \lambda$ and $\lim_{N \rightarrow +\infty} \omega'_N = \omega$. Then, as $N$ is supposed large, $\lambda_N$ and $\lambda'_N$ correspond to approximations of $\lambda$, and $\omega'_N$ to an approximation of $\omega$. 

\paragraph{Estimators and main result.}  To estimate $n_0$ from the observations of $n_{\partial}^{(1)}$ or $n_{\partial}^{(2k)}$, we use the estimators $\widehat{n}_0^{(1)}$ and $\widehat{n}_0^{(2k)}$ respectively, defined for all $x \geq 0$ as
\begin{equation}\label{eq:definitions_estimators}
	\widehat{n}_0^{(1)}(x) := \frac{1}{\tilde{b}\tilde{m}_1}n_{\partial}^{(1)}\left(\frac{x}{\tilde{b}\tilde{m}_1}\right),\hspace{8mm} \widehat{n}_0^{(2k)}(x) := \frac{1}{k\tilde{b}\tilde{m}_1}\frac{n_{\partial}^{(2k)}\left(\frac{2x}{\tilde{b}\tilde{m}_1}\right)}{\left(\int_{\frac{2x}{\tilde{b}\tilde{m}_1}}^{\infty} n_{\partial}^{(2k)}\left(s\right) \dd s\right)^{1 - \frac{1}{2k}}}.
\end{equation}
The estimator on the left comes from the approximation of the pair $\left(n^{(1)},n_{\partial}^{(1)}\right)$ by a transport equation with drift~$-bm_1$ and an absorbing state at $x = 0$, see Section~\ref{sect:model_one_telomere}. Thanks to this approximation, we have by using the method of characteristics that $n_{\partial}^{(1)}(t)$ is close to $bm_1 n_0\left(bm_1 t\right)$ for all~$t \geq 0$ (see~\eqref{eq:alternative_representation_transport_model_1_telomere}). Therefore, by first taking $t = \frac{x}{bm_1}$, then dividing both terms by $bm_1$, and finally using the following equality that comes from the change of variable~$u' = \frac{u}{N}$ and~\hyperlink{assumption:H1}{$(H_1)$}:
\begin{equation}\label{eq:equality_transport_terms}
bm_1 = b\int_0^{\delta} u g(u) \dd u = bN^2\int_0^{\frac{\delta}{N}} u' g\left(Nu'\right) \dd u' = \tilde{b}\tilde{m}_1,
\end{equation}
we obtain that $\frac{1}{\tilde{b}\tilde{m}1} n_{\partial}^{(1)}\left(\frac{x}{\tilde{b}\tilde{m}_1}\right)$ is close to $n_0\left(x\right)$ for all $x\geq0$. 

The estimator on the right of~\eqref{eq:definitions_estimators} also comes from an approximation of $\left(n^{(2k)},n_{\partial}^{(2k)}\right)$, see Section~\ref{sect:model_several_telomeres}. This time, the approximated model is a transport equation with drift~$-\frac{bm_1}{2}$ in each coordinate, and with an absorbing state at the boundary of $\mathbb{R}_+^{2k}$. Here are the steps to obtain the estimator. First, by~\eqref{eq:link_density_lengths_cemetery}, our model approximation, and the method of characteristics, we have that $\int_{t}^{+\infty} n_{\partial}^{(2k)}(s) \dd s$ is close to~$\left[\int_{\frac{bm_1 t}{2}}^{+\infty} n_{0}(s) \dd s\right]^{2k}$  for all~$t\geq0$ (see Remark~\ref{rem:mass_conservation_transport}). Then, by taking both terms to the power~$\frac{1}{2k}$ and differentiating, we have that $\frac{1}{2k} n_{\partial}^{(2k)}(t) \left[\int_{t}^{+\infty} n_{\partial}^{(2k)}(s) \dd s \right]^{\frac{1}{2k} - 1}$ is approximatively equal to $\frac{bm_1}{2} n_{0}\left(\frac{bm_1 t}{2}\right)$ for all $t\geq0$. Finally, by taking \hbox{$t = \frac{2x}{bm_1}$}, then multiplying by $\frac{2}{bm_1}$, we have that $\frac{1}{k b m_1} n_{\partial}^{(2k)}\left(\frac{2x}{b m_1}\right) \left[\int_{\frac{2x}{b m_1}}^{+\infty} n_{\partial}^{(2k)}(s) \dd s \right]^{\frac{1}{2k} - 1}$ is close to $n_{0}\left(x\right)$ for all $x\geq0$. This last result justifies the expression of the estimator, because~$bm_1$ and $\tilde{b}\tilde{m}_1$ are equals, see~\eqref{eq:equality_transport_terms}. 


We detail more rigorously these steps in Sections~\ref{sect:model_one_telomere} and~\ref{sect:model_several_telomeres}. This leads to the following result, which provides upper bounds on the errors between the estimators defined in~\eqref{eq:definitions_estimators} and~$n_0$.
\begin{te}[Main result]\label{te:main_result}
	We recall the constants $\lambda_N$, $\lambda'_N$ and $\omega'_N$ defined in~\eqref{eq:approximation_eigenvalues}. We also consider the constant $\beta'_N := \left(\lambda + 2k\lambda'_N\right) - \left(\omega + (2k-1)\omega'_N\right)$. The following statements hold. 
	\begin{enumerate}[$(a)$]
		\item Assume \hyperlink{assumption:H1}{$(H_1)-(H_2)$}. Then, there exists $c_1>0$ such that for all $x\geq 0$
		$$
		\left|\widehat{n}_0^{(1)}(x) - n_0(x)\right|  \leq  \frac{c_1}{N}(x +1)\exp\left(-\lambda_N x\right),
		$$
		where $c_1$ depends only on $g$, $C_{\lambda}$, $C'_{\lambda}$ and $D_{\lambda}$. 
		\item Assume \hyperlink{assumption:H1}{$(H_1)-(H_4)$}. Then, there exists $d_1 > 0$ such that for all~$x\geq0$ 
		$$
		\left|\widehat{n}_0^{(2k)}(x) - n_0(x)\right| \leq \frac{d_1}{N}\left(\frac{D_{\lambda}}{D_{\omega}}\right)^{2k}\left(k^2x + k + 1\right)\exp\left[-\beta'_Nx\right],
		$$
		where $d_1$ depends only on $g$, $\delta$, $\lambda$, $C_{\lambda}$, $C'_{\lambda}$, $D_{\lambda}$, and $D_{\omega}$.
	\end{enumerate}
\end{te}
\begin{rem}
\noindent From the above, we have that the pointwise errors tend to zero when $N \rightarrow +\infty$.
\end{rem}

Let us comment on the dependence on $x$ of the bounds in Theorem~\ref{te:main_result}. First, assume that $\lambda + 2k\lambda'_N > \omega + (2k-1)\omega'_N$. Then, for each of the estimators, the bound contains a term corresponding to a linear growth in $x$, and a term corresponding to an exponential decay in~$x$. The linear growth is related to the accumulation of errors during the model approximation. The exponential decay is related to the fact that the influence of past errors dissipates over time at an exponential rate. In the case where~$\lambda + 2k\lambda'_N \leq \omega + (2k-1)\omega'_N$, there is no exponential decay for $\widehat{n}_0^{(2k)}$. This is because the term in the denominator of $\widehat{n}_0^{(2k)}$, which tends to~$0$ exponentially fast, compensates for the dissipation and leads to an exponential growth of the error. The fact that there is a condition for having an exponential decay is consistent with the fact that the hazard rate function associated to a given density, which is similar to $\widehat{n}_0^{(2k)}$, may explode when $x\rightarrow+\infty$, see~\cite[Sections~7.VI \& 7.IX]{kleinbaum_survival_2012}. However, the condition~$\lambda + 2k\lambda'_N > \omega + (2k-1)\omega'_N$ is not optimal, as discussed in the penultimate paragraph of Section~\ref{sect:discussion_estimation}. This opens new perspectives, notably to improve our bounds.



Let us now comment on the dependence on $k$ of the bound obtained in Theorem~\ref{te:main_result}-$(b)$. We observe that two terms contribute to the $k$-dependence of the error. The first one is the term~$\left(\frac{D_{\lambda}}{D_{\omega}}\right)^{2k}$. It implies that the error bound grows exponentially when the number of chromosomes increase. In fact, this growth is mainly related to the fact that our assumptions are not optimal, which leads to a non-optimal bound. This non-optimality is due to the fact that comparing~$n_0$ with the densities presented in~\hyperlink{assumption:H3}{$(H_3)$} and~\hyperlink{assumption:H4}{$(H_4)$} results in a loss of information about~$n_0$, and that this loss is multiplied when the dimension increases. We would like to manage this loss of information in a future work in order to obtain a bound without the term~$\left(\frac{D_{\lambda}}{D_{\omega}}\right)^{2k}$. The second term that contributes to the $k$-dependence of the bound in Theorem~\ref{te:main_result}-$(b)$ is the term~$k^2x + k + 1$. It is related to the fact that the size of the space~$\mathbb{R}_+^{2k}$ becomes larger when~$k$ increases, making the model more difficult to approximate. In fact, the dependence on $k$ of the error poses problems when one is interested in species for which $2k$ is large, such as yeast cells~($2k = 32$) or human cells~($2k = 92$). This issue is further studied in Section~\ref{subsect:estimation_results_severaltelos}. 




We now need to ensure that the accumulation of the pointwise errors does not lead to an explosion of the global error. To verify this, we obtain from Theorem~\ref{te:main_result} and the fact that $\int_0^{+\infty} x^p e^{-\beta x} \, \dd x = \frac{\Gamma(p+1)}{\beta^{p+1}}$ for all $\beta \geq 0$ and~$p > 0$ the following corollary. It provides bounds on the errors of our estimators in all Lebesgue spaces. \todoJules{the ?}
\begin{cor}[Estimation errors in Lebesgue spaces]\label{cor:error_lebesgue_spaces}
	We recall the constants $\lambda_N$, $\lambda'_N$ and $\omega'_N$ defined in~\eqref{eq:approximation_eigenvalues}. We also consider the constant $\beta'_N := \left(\lambda + 2k\lambda'_N\right) - \left(\omega + (2k-1)\omega'_N\right)$. The following statements hold.
	
	\begin{enumerate}[$(a)$]
		\item Assume \hyperlink{assumption:H1}{$(H_1)-(H_2)$}. Then, for all~$p > 0$ we have
		$$
		\left|\left|\widehat{n}_0^{(1)} - n_0\right|\right|_{L^{p}(\mathbb{R}_+)}  \leq  \frac{c_1}{N}\left(\frac{\Gamma(p+1)^{\frac{1}{p}}}{(\lambda_N p)^{1+\frac{1}{p}}} + \frac{1}{\left(\lambda_N p\right)^{\frac{1}{p}}}\right),
		$$
		where $c_1$ is the same constant as in Theorem~\ref{te:main_result}-$(a)$.
		\item Assume \hyperlink{assumption:H1}{$(H_1)-(H_4)$}, and that $\beta'_N > 0$. Then, for all $p > 0$, we have
		$$
		\left|\left|\widehat{n}_0^{(2k)} - n_0\right|\right|_{L^{p}(\mathbb{R}_+)}  \leq  \frac{d_1}{N}\left(\frac{D_{\lambda}}{D_{\omega}}\right)^{2k}\left(\frac{k^2\Gamma(p+1)^{\frac{1}{p}}}{\left(\beta'_Np\right)^{1+\frac{1}{p}}} + \frac{k+1}{\left(\beta'_Np\right)^{\frac{1}{p}}}\right),
		$$
		where $d_1$ is the same constant as in Theorem~\ref{te:main_result}-$(a)$.
	\end{enumerate}
\end{cor}	
\noindent From the above, we have that if~$\lambda + 2k\lambda'_N \leq \omega + (2k-1)\omega'_N$, then the accumulation of the pointwise errors for~$\widehat{n}_0^{(2k)}$ is too large to get a result in norm, due to the exponential growth of the pointwise errors (linear when $\lambda + 2k\lambda'_N = \omega + (2k-1)\omega'_N$). 

\subsection{Discussion about the single-telomere model and the assumptions}\label{subsect:discussion_models_assumptions}

We conclude this section by explaining why we work  with a model with one telomere, and discussing our assumptions. 



\paragraph{Relevance of using a model with one telomere.} As said in the presentation of the first model, up to our knowledge, there is no species with only one telomere. However, there are at least two good reasons to work with such a model.

\begin{itemize}[label = \fontsize{6}{12}\selectfont\textbullet,leftmargin=0.5cm]
	
	\item It is possible with experimental methods to place ourselves in a setting very similar to the study of cells with one telomere, see~\cite[p.~$112$]{bechara2023}. 
	
	\item It has been deduced numerically in~\cite[Figure~5.a]{bourgeron_2015}, which study a discrete probabilistic model with the same dynamics as those presented in~\eqref{eq:PDE_model_telomeres_several_telos}, that $60\%$ of senescence times are signalled by the telomere that was the shortest at the beginning of the dynamics, so that this telomere quite often signals senescence. Therefore, even if the approximation is rough, we can assume that the shortest telomere at the beginning always signals senescence to gain first insight. Under this approximation, it is sufficient to use a single-telomere model.
\end{itemize}

\paragraph{Discussion about the assumptions.}\label{paragraph:discussion_models} 

Let us present the consequences of our assumptions, how they are useful in our proofs, and why we made them.

\begin{itemize}[label = \fontsize{6}{12}\selectfont\textbullet,leftmargin=0.5cm]
	\item\hyperlink{assumption:H1}{$(H_1):$}\label{paragraph:discussion_H1}\hypertarget{paragraph:discussion_H1}{} This is the key assumption of this paper, as it allows us to justify that we can approximate our models by transport equations, see~Sections~\ref{subsect:onetelo_rewritingmodel}~and~\ref{subsect:severaltelos_rewritingmodel}. Since $N$ is assumed to be large, this assumption means that the shortening values are small compared to the scale where telomere lengths are initially distributed. This assumption that the shortening values are small is supported by the following biological~reality:
	\begin{itemize}%
		\item The average telomere length of the budding yeast is of the order of $300$ base pairs~\hbox{\cite{pfeiffer_replication_2013}}, and the average shortening value is of the order of $7.5$ base pairs~\cite{eugene_effects_2017}. The ratio between these is $\frac{7.5}{300} = \frac{1}{40}= 2.5\%$, which is small. 
		\item The average telomere length of the human is of the order of $12.5$ kilobase pairs~\cite{pfeiffer_replication_2013}, and the average shortening value is of the order of $0.125$ kilobase pairs~\cite{hwang_telomeric_2014}. The ratio between these is $\frac{0.125}{12.5} = \frac{1}{100}= 1\%$, which is small.
	\end{itemize}
	By the above explanation, we have that $N = 40$ when we study the budding yeast, and $N = 100$ when we study the human.
	
	The assumption on $\tilde{b}$ means that we work on a time scale where division times occur very frequently. This allows us to compensate the fact that the shortening values are small, and to avoid to have senescence times that tend to infinity when $N\rightarrow+\infty$. This assumption also implies that the constraints 
	\begin{equation}\label{eq:constraints_cemetery}
	\left|\left|n_{\partial}^{(1)}\right|\right|_{L^\infty\left(\mathbb{R}_+\right)} \leq \tilde{b} = bN \hspace{4mm}\text{and}\hspace{4mm} \left|\left|n_{\partial}^{(2k)}\right|\right|_{L^\infty\left(\mathbb{R}_+\right)} \leq \tilde{b}= bN
	\end{equation}
	presented in Remarks~\ref{rem:conservation_individual_one_telo} and~\ref{rem:conservation_individual_several_telos} become by letting $N$ go to infinity: 
	$$
	\left|\left|n_{\partial}^{(1)}\right|\right|_{L^\infty\left(\mathbb{R}_+\right)} < +\infty \hspace{4mm}\text{and}\hspace{4mm} \left|\left|n_{\partial}^{(2k)}\right|\right|_{L^\infty\left(\mathbb{R}_+\right)} < +\infty, 
	$$
	which is no more restrictive. Hence, in the rest of the article, when it is required to approximate or estimate~$n_{\partial}$, we do not try to satisfy~\eqref{eq:constraints_cemetery}.
	
	\smallskip
	
	\item\hyperlink{assumption:H2}{$(H_2):$} Due to a second-order Taylor expansion, the errors between the approximated models and the original models are mainly given by the second derivative of $n_0$, see Section~\ref{subsect:proof_key_lemma_lengths}. The inequality on the left-hand side of \hyperlink{assumption:H2}{$(H_2)$} allows us to have a control on it. 
	
	The inequality on the right-hand side, for its part, allows us to have a control on the variation of telomere length density close to $0$. Controlling this is important because cells susceptible to become senescent have telomere lengths close to $0$. For more information, we refer to Sections~\ref{subsect:proof_key_lemma_cemetery},~\ref{subsect:proof_approximation_one_telomere_model} and~\ref{subsect:proof_approximation_cemetery_several}. 
	\smallskip

	\item\hyperlink{assumption:H3}{$(H_3):$} The approximation error for the model with several telomeres does not only depend on the derivatives of $n_0$, but also on $n_0$ itself. This assumption allows us to control it.
	\smallskip

	\item\hyperlink{assumption:H4}{$(H_4):$} This assumption allows us to obtain a lower bound for $\int_{t}^{+\infty} n_{\partial}^{(2k)}(s) \dd s$, for all $t\geq 0$. It is important to have such lower bound because the inverse of $\int_{t}^{+\infty} n_{\partial}^{(2k)}(s) \dd s$ appears in the expression of $\widehat{n}_0^{(2k)}$, see~\eqref{eq:definitions_estimators}, and tends to $0$ when $t\rightarrow+\infty$. Therefore, if the decay of the function $t\mapsto\int_{t}^{+\infty} n_{\partial}^{(2k)}(s) \dd s$ is too fast, then $\widehat{n}_0^{(2k)}(x)$ explodes when~$x\rightarrow+\infty$. 
	
\end{itemize}

\section{The single-telomere model}\label{sect:model_one_telomere}

We begin by bounding the error done by~$\widehat{n}^{(1)}_0$, i.e., we prove Theorem~\ref{te:main_result}-$(a)$. This statement follows almost directly from the model approximation. First, in Section~\ref{subsect:onetelo_rewritingmodel}, we successively rewrite~\eqref{eq:PDE_model_telomeres_one_telo} using the parameters introduced in~\hyperlink{assumption:H1}{$(H_1)$}, explain how an approximation can be obtained from this rewriting, and prove Theorem~\ref{te:main_result}-$(a)$ assuming the approximation is true. Then, in Section~\ref{subsect:plan_mainproof_onetelo}, we present the auxiliary results necessary to obtain this approximation. Thereafter,  in Sections~\ref{subsect:proof_key_lemma_lengths} and~\ref{subsect:proof_key_lemma_cemetery}, we prove these auxiliary statements. Finally, in Section~\ref{subsect:proof_approximation_one_telomere_model}, we prove the model approximation. Throughout this section, we assume that \hyperlink{assumption:H1}{$(H_1)$} holds.
\subsection{Model approximation and proof of Theorem~\ref{te:main_result}-\texorpdfstring{$(a)$}{(a)}}\label{subsect:onetelo_rewritingmodel}

The approximation of~\eqref{eq:PDE_model_telomeres_one_telo} is obtained by letting the scaling parameter~$N$, introduced in~\hyperlink{assumption:H1}{$(H_1)$}, tend to infinity. Thus, we need to rewrite~\eqref{eq:PDE_model_telomeres_one_telo} to make $N$ appears. Let us start with the first line of~\eqref{eq:PDE_model_telomeres_one_telo}. In this equation, we replace $\tilde{b}$ with~$bN$, then $\tilde{g}(v)$ with~$Ng(Nv)$ for all $v\geq0$, and finally $\tilde{\delta}$ with~$\frac{\delta}{N}$. Thereafter, we do the change of variable $v' = Nv$, and place $n^{(1)}(t,x)$ inside the integral by using that $\int_0^{\delta}g(v) \dd v = 1$. We obtain that for all~$t\geq0$,~$x\geq0$,
\begin{equation}\label{eq:rewriting_firstmodel_firstline}
	\begin{aligned}
		\partial_t n^{(1)}(t,x) &= bN^2\int_0^{\frac{\delta}{N}}n^{(1)}\left(t,x + v\right)g\left(Nv\right) \mathrm{d}v - bN.n^{(1)}(t,x) \\
		& = bN\int_0^{\delta}\left[n^{(1)}\left(t,x + \frac{v'}{N}\right)-n^{(1)}(t,x)\right]g(v') \mathrm{d}v'.
	\end{aligned}
\end{equation}
Now, we rewrite the second line of~\eqref{eq:PDE_model_telomeres_one_telo}. First observe that by the change of variable $w' = \frac{w}{N}$, we have for all $v\in\left[0,\frac{\delta}{N}\right]$
\begin{equation}\label{eq:relation_between_Gtilde_and_G}
	G(Nv) = \int_0^{Nv} g(w)\,\mathrm{d}w= N\int_0^{v} g(Nw')\,\mathrm{d}w' = \int_0^{v} \tilde{g}(w')\,\mathrm{d}w' = \tilde{G}(v). 
\end{equation}
In view of the above equality, in the second line of~\eqref{eq:PDE_model_telomeres_one_telo}, we successively replace $\tilde{b}$ with $bN$, $\tilde{G}(v)$ with $G(Nv)$  for all $v\geq0$, and finally $\tilde{\delta}$ with $\frac{\delta}{N}$. Then, we do the change of variable $v' = Nv$. It comes the following rewriting, for all~$t\geq0$,
\begin{equation}\label{eq:rewriting_firstmodel_secondline}
	\begin{aligned}
		n_{\partial}^{(1)}(t) &=  bN\int_{0}^{\frac{\delta}{N}}n^{(1)}(t,v)(1-G(Nv))\,\mathrm{d}v = b\int_{0}^{\delta} n^{(1)}\left(t,\frac{v'}{N}\right)(1-G(v'))\,\mathrm{d}v'. 
	\end{aligned}
\end{equation}
By combining~\eqref{eq:rewriting_firstmodel_firstline} and~\eqref{eq:rewriting_firstmodel_secondline}, we now have below a new expression for~\eqref{eq:PDE_model_telomeres_one_telo} in which $N$ appears
\begin{equation}\label{eq:rescaled_PDE_model_telomeres_one_telo}
	\begin{cases}
		\partial_t n^{(1)}(t,x) = bN\int_0^{\delta}\left[n^{(1)}\left(t,x + \frac{v}{N}\right)- n^{(1)}(t,x)\right]g(v)\,\mathrm{d}v, & \forall t\geq0,\,x\geq0,\\
		n^{(1)}_{\partial}(t) = b\int_0^{\delta} n^{(1)}\left(t,\frac{v}{N}\right)(1-G(v))\,\mathrm{d}v ,& \forall t\geq0, \\ 
		n^{(1)}(0,x) = n_0(x), & \forall x\geq 0.
	\end{cases}
\end{equation}
We aim to derive a system corresponding to the limit version of~\eqref{eq:rescaled_PDE_model_telomeres_one_telo} as $N\rightarrow +\infty$. In fact, this system can be intuitively obtained. To do so, one has to observe that by the definition of the derivative and the equality $\int_0^{\delta}(1-G(v))\,\mathrm{d}v = \int_0^{\delta}vg(v) \,\mathrm{d}v = m_1$ (integration by part), we have the following two results, for all $t\geq0$,
$$
\begin{aligned}
	N\int_0^{\delta}\left[n^{(1)}\left(t,x + \frac{v}{N}\right)- n^{(1)}(t,x)\right]g(v) \,\mathrm{d}v &\underset{N\rightarrow+\infty}{\approx} \int_0^{\delta}\left[v\partial_x n^{(1)}(t,x) \right]g(v) \,\mathrm{d}v = m_1\partial_x n^{(1)}(t,x), \\
	\int_0^{\delta} n^{(1)}\left(t,\frac{v}{N}\right)(1-G(v))\,\mathrm{d}v &\underset{N\rightarrow+\infty}{\approx} m_1n^{(1)}\left(t,0\right).
\end{aligned}
$$
Then, by plugging the above in~\eqref{eq:rescaled_PDE_model_telomeres_one_telo}, we can conjecture that the following approximates \eqref{eq:rescaled_PDE_model_telomeres_one_telo}
\begin{equation}\label{eq:approximation_transport_model_1_telomere}
	\begin{cases}
		\partial_t u^{(1)}(t,x) = bm_1\partial_{x} u^{(1)}(t,x), & \forall t\geq0,\,x\geq0,\\
		u_{\partial}^{(1)}(t) = bm_1u^{(1)}(t,0), & \forall t\geq 0,\\ 
		u^{(1)}(0,x) = n_0(x), & \forall x\geq0.
	\end{cases}
\end{equation}
\begin{rem}
	By using the method of characteristics to solve the first line of~\eqref{eq:approximation_transport_model_1_telomere}, and then plugging the solution in the second line, we have for all $t\geq 0$,~$x\geq 0$,
	\begin{equation}\label{eq:alternative_representation_transport_model_1_telomere}
		\begin{aligned}
			u^{(1)}(t,x) = n_0(bm_1t + x), \quad \text{and} \quad u_{\partial}^{(1)}(t) = bm_1n_0(bm_1t).
		\end{aligned}
	\end{equation}
\end{rem} 

In fact, our conjecture can be rigorously proven. Specifically, the following result, proved in Section~\ref{subsect:proof_approximation_one_telomere_model}, provides bounds on the pointwise errors between~\eqref{eq:rescaled_PDE_model_telomeres_one_telo} and~\eqref{eq:approximation_transport_model_1_telomere}.
\begin{prop}[Pointwise approximation errors, one telomere]\label{prop:approximation_PDE_one_telomere}
	We recall the constant $\lambda_N$ defined in~\eqref{eq:approximation_eigenvalues}. The following statements hold. 
	\begin{enumerate}[$(a)$]
		\item Assume \hyperlink{assumption:H1}{$(H_1)-(H_2)$}. Then, there exists $c'_0 > 0$ such that for all $t\geq 0$, $x\geq 0$, we have
		$$
		\begin{aligned}
			\left|n^{(1)}(t,x) - u^{(1)}(t,x)\right| \leq \frac{c'_0bt}{N}\exp\left(-bm_1\lambda_N t\right)\exp\left(-\lambda x\right),
		\end{aligned}
		$$
		where $c'_0$ depends only on $g$ and $C_{\lambda}$.
		\item Assume \hyperlink{assumption:H1}{$(H_1)-(H_2)$}. Then, there exists $c'_1 > 0$ such that for all $t\geq 0$, we have
		$$
		\left|n_{\partial}^{(1)}(t) - u_{\partial}^{(1)}(t)\right| \leq \frac{c'_1b\left(bm_1t+1\right)}{N}\exp\left(-bm_1\lambda_N t\right),
		$$
		where $c'_1$ depends only on $g$, $C_{\lambda}$ and $C'_{\lambda}$. 
	\end{enumerate}
\end{prop}
\begin{rem}\label{rem:no_corollary_explanation}
As done in Corollary~\ref{cor:error_lebesgue_spaces}, the errors  in all Lebesgue spaces between $n^{(1)}$ and $u^{(1)}$, and $n_{\partial}^{(1)}$ and~$u_{\partial}^{(1)}$ can be bounded thanks to the above inequalities. We do not state these bounds as they are not necessary to prove our main theorem.
\end{rem}
\noindent The main interest of Proposition~\ref{prop:approximation_PDE_one_telomere} is that it allows us to get a bound on the error done by $\widehat{n}_0^{(1)}$. To obtain it, first observe that by the right-hand side of~\eqref{eq:alternative_representation_transport_model_1_telomere}, it holds~\hbox{$n_0(x) = \frac{1}{bm_1}u_{\partial}^{(1)}\left(\frac{x}{bm_1}\right)$} for all $x\geq0$. Then, in view of the definition of~$\widehat{n}_0^{(1)}$ given in~\eqref{eq:definitions_estimators} and Eq.~\eqref{eq:equality_transport_terms}, we have by applying Proposition~\ref{prop:approximation_PDE_one_telomere}-$(b)$ the following corollary, which directly implies Theorem~\hbox{\ref{te:main_result}-$(a)$}.
\begin{cor}[Pointwise estimation error, one telomere]\label{cor:quality_estimator_model_one_telo}
	We recall the constant $\lambda_N$ defined in~\eqref{eq:approximation_eigenvalues}. Assume that \hyperlink{assumption:H1}{$(H_1)-(H_2)$} hold. Then, for all $x\geq0$, we have 
	$$
	\left|\widehat{n}_0^{(1)}(x) - n_0(x)\right| = \frac{1}{bm_1}\left|n_{\partial}^{(1)}\left(\frac{x}{bm_1}\right)-u_{\partial}^{(1)}\left(\frac{x}{bm_1}\right)\right|  \leq \frac{c'_1\left(x+1\right)}{m_1N}\exp\left(-\lambda_N x\right),
	$$
	where $c'_1$ is the same constant as in Proposition~\ref{prop:approximation_PDE_one_telomere}-$(b)$.
\end{cor}
\noindent From the above, if we manage to justify that Proposition~\ref{prop:approximation_PDE_one_telomere} is true, then \hbox{Theorem~\ref{te:main_result}-$(a)$} will be proved. Explaining why this proposition holds is thus what we do in the next subsections. In particular, we now present the main arguments and the auxiliary statements used to obtain~it.


\subsection{Plan of the proof of Proposition~\ref{prop:approximation_PDE_one_telomere}}\label{subsect:plan_mainproof_onetelo}


The proof of the first statement of Proposition~\ref{prop:approximation_PDE_one_telomere} consists in controlling the absolute value of $\overline{u}^{(1)} := n^{(1)} - u^{(1)}$. To do so, we first obtain an equation verified by~$\overline{u}^{(1)}$. By taking the difference between the first lines of~\eqref{eq:rescaled_PDE_model_telomeres_one_telo} and~\eqref{eq:approximation_transport_model_1_telomere}, then decomposing $n^{(1)}$ with the equality $n^{(1)}= \overline{u}^{(1)}  +  u^{(1)}$, and finally using that $bm_1= bN\int_0^{\delta}\frac{v}{N}g(v) \dd v$, we have for all $t\geq0$, $x\geq0$,
\begin{equation}\label{eq:equation_difference_onetelomodel_approximant_lengths} 
	\begin{aligned}
		\partial_t \overline{u}^{(1)}(t,x) &= bN\int_0^{\delta}\left[n^{(1)}\left(t,x + \frac{v}{N}\right)- n^{(1)}(t,x)\right]g(v)\,\mathrm{d}v - bm_1 \partial_xu^{(1)}(t,x) \\
		&= bN\int_0^{\delta}\left[\overline{u}^{(1)}\left(t,x + \frac{v}{N}\right)- \overline{u}^{(1)}(t,x)\right]g(v)\,\mathrm{d}v  \\
		&+bN\int_0^{\delta}\left[u^{(1)}\left(t,x + \frac{v}{N}\right)- u^{(1)}(t,x) - \frac{v}{N}\partial_xu^{(1)}(t,x)\right]g(v)\,\mathrm{d}v.
	\end{aligned}
\end{equation}
Therefore, if we are able to bound a solution of the above equation, then Proposition~\ref{prop:approximation_PDE_one_telomere}-$(a)$ will be proved. The following lemma, proved in Section~\ref{subsect:proof_key_lemma_lengths}, allows us to do this. This lemma is stated for equations in a multidimensional trait space instead of only~$\mathbb{R}_+$, as we use it later in Section~\ref{subsect:proof_approximation_lengths_several} for the approximation of the model with several telomeres.
\begin{lemm}[Key lemma for approximating lengths densities]\label{lemm:key_lemma_lengths}
	Let $d\in\mathbb{N}^*$, $\xi$ a probability measure on $\mathbb{R}^d$ with finite first and second moments, and $F : \mathbb{R}_+\times\mathbb{R}_+^d \rightarrow \mathbb{R}$. We also consider~$u_{\xi}$ the solution of the following integro-differential equation, for all $(t,x)\in\mathbb{R}_+\times\mathbb{R}_+^d$,
	
	\begin{equation}\label{eq:PDE_to_develop_general}
		\begin{aligned}
			\partial_tu_{\xi}(t,x) &= bN\int_{v\in\mathbb{R}_+^d}\left[u_{\xi}\left(t,x + \frac{v}{N}\right)- u_{\xi}(t,x)\right] \xi\left(\mathrm{d}v\right)  \\ 
			&+ bN\int_{v\in \mathbb{R}_+^d}\left[F\left(t,x + \frac{v}{N}\right)- F\left(t,x\right) - \sum_{i = 1}^{d}\frac{v_i}{N}\partial_{x_i} F\left(t,x\right)\right]\xi\left(\mathrm{d}v\right),
		\end{aligned}
	\end{equation}
	with initial condition $u_{\xi}(0,.) \equiv 0$. Assume that there exist  $\alpha,\,\beta > 0$ and $C > 0$ such that
	\begin{align}				\forall (t,x)\in\mathbb{R}_+\times\mathbb{R}_+^d:\hspace{1mm}\sup_{(\ell,\ell')\in\llbracket1,d\rrbracket^2}\left(\left|\partial_{x_\ell x_{\ell'}}F(t,x)\right|\right) &\leq C\exp\left(-\alpha t - \beta\sum_{i = 1}^d x_i\right), \label{eq:condition_second_derivative_general_lemma} \\ 
		bN\left(1 - \mathcal{L}(\xi)\left(\frac{\beta}{N}\right)\right)&\leq \alpha. \label{eq:inequalities_alpha_and_beta}
	\end{align}
	Then, denoting the constant $\sigma_{\xi} = \sum_{1\leq \ell,\ell'\leq d}\int_{v\in\mathbb{R}_+^d} v_\ell v_{\ell'} \xi(\dd v)$, we have for all $(t,x)\in\mathbb{R}_+\times\mathbb{R}_+^d$
	\begin{equation}\label{eq:bound_key_lemma}
		\big|u_{\xi}(t,x)\big| \leq \frac{btC\sigma_{\xi}}{2N}\exp\left[-bN\left(1-\mathcal{L}(\xi)\left(\frac{\beta}{N}\right)\right)t\right]\exp\left(-\beta \sum_{i =1}^dx_i\right).
	\end{equation}
\end{lemm}
\noindent One can easily see that~\eqref{eq:equation_difference_onetelomodel_approximant_lengths} is an equation of the form given in~\eqref{eq:PDE_to_develop_general} with $d = 1$, $u_{\xi} = \overline{u}^{(1)}$, \hbox{$\xi = g$} and~$F = u^{(1)}$. The proof of the first statement of Proposition~\ref{prop:approximation_PDE_one_telomere} is thus to check the assumptions of Lemma~\ref{lemm:key_lemma_lengths}, and then to apply it. We do this in Section~\ref{subsect:proof_approximation_one_telomere_model}. 

To prove Proposition~\ref{prop:approximation_PDE_one_telomere}-$(b)$, we need this time to control \hbox{$\big|\overline{u}_{\partial}^{(1)}\big| := \big|n_{\partial}^{(1)} - u_{\partial}^{(1)}\big|$}. Again, we do this by obtaining the equation verified by $\overline{u}_{\partial}^{(1)}$, and then applying a general lemma. By taking the difference between the second lines of~\eqref{eq:rescaled_PDE_model_telomeres_one_telo} and~\eqref{eq:approximation_transport_model_1_telomere},  then decomposing $n^{(1)}$ with the equality $n^{(1)} = \overline{u}^{(1)} + u^{(1)}$, and finally using that $\int_0^{\delta}(1-G(v)) \dd v = \int_0^{\delta}v g(v) \dd v = m_1$ (integration by part), we have for all~$t\geq0$ 
\begin{equation}\label{eq:equation_difference_model_approximant_cemeteries}
\begin{aligned}
\overline{u}_{\partial}^{(1)}(t) &= b\int_0^{\delta} n^{(1)}\left(t,\frac{v}{N}\right)(1-G(v)) \mathrm{d}v - bm_1u^{(1)}(t,0) \\ 
&= b\int_0^{\delta} \overline{u}^{(1)}\left(t,\frac{v}{N}\right)(1-G(v)) \mathrm{d}v + b\int_0^{\delta} \left[u^{(1)}\left(t,\frac{v}{N}\right) - u^{(1)}\left(t,0\right)\right](1-G(v)) \mathrm{d}v.
\end{aligned}
\end{equation}
We thus need to control an equation with the same form as~\eqref{eq:equation_difference_model_approximant_cemeteries}, and the second statement of Proposition~\ref{prop:approximation_PDE_one_telomere} will be proved. This control is done by using the following lemma, that is proved in Section~\ref{subsect:proof_key_lemma_cemetery}.



\begin{lemm}[Key lemma for approximating cemeteries]\label{lemm:key_lemma_cemetery}
	Assume that the assumptions of Lemma~\ref{lemm:key_lemma_lengths} hold. We consider a set of functions~$\left(h_i\right)_{i\in\llbracket1,d\rrbracket}$ from $\mathbb{R}_+\times\mathbb{R}^d$ to $\mathbb{R}$ and $C' >0$ such that for all $t\geq0$, $x\in\mathbb{R}_+^d$ and~$i\in\llbracket1,d\rrbracket$, it holds
	\begin{equation}\label{eq:condition_approx_cemetery_first}
		\left|h_i(t,x)\right| \leq \frac{C'x_i}{N}\exp\left[-bN\left(1-\mathcal{L}(\xi)\left(\frac{\beta}{N}\right)\right)t -\beta \sum_{\substack{j = 1,\,j\neq i}}^d x_j\right].
	\end{equation}
	We denote the function $v_{\xi} : \mathbb{R}_+ \rightarrow \mathbb{R}$, defined for all $t\geq 0$ as
	\begin{equation}\label{eq:condition_approx_cemetery_second}
		v_{\xi}(t) := bN\int_{y\in\mathbb{R}_+^{d}} u_{\xi}(t,y) \xi\big(\big\{v\in\mathbb{R}_+^{d}\,|\,Ny - v\notin\mathbb{R}_+^d\big\}\big)\mathrm{d}y + b\sum_{i = 1}^d\int_{y\in\mathbb{R}_+^{d}} h_i(t,y)(1-G(y_i))1_{\{y_i \leq \delta\}}\mathrm{d}y.
	\end{equation}
	Then, recalling the constant $\sigma_{\xi}$ introduced in Lemma~\ref{lemm:key_lemma_lengths}, we have for all $t\geq 0$
	\begin{equation}\label{eq:key_lemma_cemetery}
		\big|v_{\xi}(t)\big| \leq \left(\frac{btC\sigma_{\xi}\alpha}{2\beta^dN} + \frac{bC'm_2d}{2\beta^{d-1}N}\right)\exp\left[-bN\left(1-\mathcal{L}\left(\xi\right)\left(\frac{\beta}{N}\right)\right)t\right].
	\end{equation}
\end{lemm}
\noindent This is less evident to see it, but Eq.~\eqref{eq:equation_difference_model_approximant_cemeteries} and a change of variable imply that $\overline{u}_{\partial}^{(1)}$ is of the form presented in~\eqref{eq:condition_approx_cemetery_second} with $\xi = g$. More details about this are given in Section~\ref{subsect:proof_approximation_one_telomere_model}, as well as the proof of \hbox{Proposition~\ref{prop:approximation_PDE_one_telomere}-$(b)$} from Lemma~\ref{lemm:key_lemma_cemetery}. 

We now prove the auxiliary statements given in this section, and then obtain Proposition~\ref{prop:approximation_PDE_one_telomere} from them.
\subsection{Proof of Lemma~\ref{lemm:key_lemma_lengths}}\label{subsect:proof_key_lemma_lengths}

We begin with preliminaries. To simplify notations, we denote \hbox{$H_N: \mathbb{R}_+\times\mathbb{R}_+^d \rightarrow \mathbb{R}$,}  the function defined for all $(t,x)\in\mathbb{R}_+\times\mathbb{R}_+^d$ as
\begin{equation}\label{eq:definition_function_HN}
	\begin{aligned}
		H_N(t,x) &:= bN\int_{v\in \mathbb{R}_+^d}\left[F\left(t,x + \frac{v}{N}\right)- F\left(t,x\right) - \sum_{i = 1}^{d}\frac{v_i}{N}\partial_{x_i} F\left(t,x\right)\right]\xi\left(\mathrm{d}v\right) \\ 
		&=  \frac{b}{N}\sum_{1 \leq \ell,\ell' \leq d}\int_{v\in\mathbb{R}_+^{d}}\left[\int_0^1 (1-w)\left( \partial_{x_\ell x_{\ell'}}F\left(t,x + w\frac{v}{N}\right)v_\ell v_{\ell'} \right)\dd w\right]\xi(\mathrm{d}v) .
	\end{aligned} 
\end{equation}
The above last equality comes from expanding $F\left(t,x + \frac{v}{N}\right)- F\left(t,x\right) - \sum_{i = 1}^{d}\frac{v_i}{N}\partial_{x_i} F\left(t,x\right)$ using a Taylor expansion with remainder in integral form. We also introduce the linear operator~\hbox{$\Phi : L^1(\mathbb{R}^d) \rightarrow C\left(\mathbb{R}_+,L^1(\mathbb{R}^d)\right)$}, defined such that for all $f_0\in L^1(\mathbb{R}^d)$,  $m = \Phi(f_0)$ is the solution in $C\left(\mathbb{R}_+,L^1(\mathbb{R}^d)\right)$ of the following integro-differential equation 
\begin{equation}\label{eq:PDE_jump_process_operator_Phi}
	\begin{cases}
		\partial_t m(t,x) = bN\int_{v\in\mathbb{R}_+^d} m\left(t,x+\frac{v}{N}\right) \xi(\dd v) - bNm(t,x), & \forall t\geq0,\,x\in\mathbb{R}_+^d, \\
		m(0,x) = f_0(x),& \forall x\in\mathbb{R}_+^d.
	\end{cases}
\end{equation}
This operator is well-posed by Proposition~\ref{prop:well_definition_general_model}.

The proof of Lemma~\ref{lemm:key_lemma_lengths} is done in two steps. First, in Step~\hyperlink{paragraph:step1_proof_keylemma}{$1$}, we prove that for all $(t,s,x)\in \mathbb{R}_+\times\mathbb{R}_+\times\mathbb{R}_+^d$ such that $t\geq s$, it holds
\begin{equation}\label{eq:step1_keylemma}
	\left|\Phi\left(H_N(s,.)\right)(t-s,x)\right| \leq \frac{bC\sigma_{\xi}}{2N}\exp\left[-bN\left(1 - \mathcal{L}(\xi)\left(\frac{\beta}{N}\right)\right)t\right]  \exp\left(- \beta\sum_{i = 1}^d x_i\right).
\end{equation}
Then, in Step~\hyperlink{paragraph:step2_proof_keylemma}{$2$}, we prove that for all $(t,x)\in\mathbb{R}_+\times\mathbb{R}_+^{d}$, we have 
\begin{equation}\label{eq:step2_keylemma}
	u_{\xi}(t,x) =\int_0^t \Phi\left(H_N(s,.)\right)(t-s,x) \dd s.
\end{equation}
In view of the fact that for all $s\geq0$, $H_N(s,.)$ is the source term in~\eqref{eq:PDE_to_develop_general} at time $s$, Eq.~\eqref{eq:step2_keylemma} means that we have rewritten $u_{\xi}$ as the sum of the evolutions of the source terms over time. This equation and Eq.~\eqref{eq:step1_keylemma} imply that the lemma is proved, as we only have to plug~\eqref{eq:step1_keylemma} in~\eqref{eq:step2_keylemma} and then integrate to obtain~\eqref{eq:bound_key_lemma}.  
\paragraph{Step $1$:}\hypertarget{paragraph:step1_proof_keylemma}{} Assume first that for all $(s,x)\in \mathbb{R}_+\times\mathbb{R}_+^d$ it holds 
\begin{equation}\label{eq:inequality_on_H_N}
	\left|H_N(s,x)\right| \leq \frac{bC\sigma_{\xi}}{2N}\exp\left(-\alpha s\right)f_{\beta}(x),
\end{equation}
where $f_{\beta}(x) := \exp\left(-\beta\sum_{i = 1}^d x_i\right)$. By Corollary~\ref{cor:maximum_principle} and the definition of $\Phi$, we have for all \hbox{$(h,\varphi)\in \left(L^1(\mathbb{R}_+^d)\right)^2$} verifying $\left|h\right| \leq \varphi$ that $\left|\Phi(h)\right| \leq \Phi(\varphi)$.
To bound from above the left-hand side of~\eqref{eq:step1_keylemma}, we first apply this inequality for $h=H_N(s,.)$ and \hbox{$\varphi = \frac{bC\sigma_{\xi}}{2N}\exp\left(-\alpha s\right)f_{\beta}$}. Then, we use that it holds $\Phi\left(\varphi\right) = \frac{bC\sigma_{\xi}}{2N}\exp\left(-\alpha s\right)\Phi\left(f_{\beta}\right)$, as $\Phi$ is a linear operator. Finally, we apply Eq.~\eqref{eq:explicit_solution_cos_expo} with $\omega = 0$ to compute $\Phi\left(f_{\beta}\right)$. We obtain that for all $(t,s,x)\in\mathbb{R}_+\times\mathbb{R}_+^d$ such that~$s\leq t$
$$
\left|\Phi\left(H_N(s,.)\right)(t-s,x)\right| \leq \frac{bC\sigma_{\xi}}{2N}\exp\left(-\alpha s\right)\exp\left[-bN\left(1 - \mathcal{L}(\xi)\left(\frac{\beta}{N}\right)\right)(t-s)\right]  \exp\left[- \beta\sum_{i = 1}^d x_i\right].
$$
Then, by using~\eqref{eq:inequalities_alpha_and_beta} to bound from above the first exponential, we get that~\eqref{eq:step1_keylemma} is true assuming that~\eqref{eq:inequality_on_H_N} holds. 

It thus remains to prove~\eqref{eq:inequality_on_H_N}, and Step~\hyperlink{paragraph:step1_proof_keylemma}{$1$} will be done. To do so, we first plug~\eqref{eq:condition_second_derivative_general_lemma} in the second line of~\eqref{eq:definition_function_HN}. Then, we bound from above $\exp\left(-\alpha s-\sum_{i = 1}^d \beta \left(x_i + w\frac{v_i}{N}\right)\right)$ by~$\exp\left(-\alpha s\right)f_{\beta}(x)$, and use the equality~$\int_0^1 (1 - w)\dd w = \frac{1}{2}$. It comes the following, which proves~\eqref{eq:inequality_on_H_N} and concludes the proof of the first step, for all $(s,x)\in\mathbb{R}_+\times\mathbb{R}_+^d$,
$$
\begin{aligned}
	\left|H_N(s,x)\right| &\leq \frac{bC}{N}\sum_{1 \leq \ell,\ell' \leq d}\int_{v\in\mathbb{R}_+^{d}} \left[\int_0^{1} (1-w) \exp\left(-\alpha s-\sum_{i = 1}^d \beta \left(x_i + w\frac{v_i}{N}\right)\right)v_\ell v_{\ell'}\dd w\right]  \xi(\mathrm{d}v) \\
	&\leq \frac{bC}{2N}\left[\sum_{1 \leq \ell,\ell' \leq d}\int_{v\in\mathbb{R}_+^{d}}v_\ell v_{\ell'} \xi(\mathrm{d}v)\right]\exp\left(-\alpha s\right)f_{\beta}(x) = \frac{bC\sigma_{\xi}}{2N}\exp\left(-\alpha s\right)f_{\beta}(x).
\end{aligned}
$$

\paragraph{Step $2$:}\hypertarget{paragraph:step2_proof_keylemma}{} To obtain~\eqref{eq:step2_keylemma}, we prove that the function $\tilde{u}_{\xi}$, defined for all $(t,x)\in\mathbb{R}_+\times\mathbb{R}_+^d$ as 
\begin{equation}\label{eq:PDE_alternative_taylor}
	\tilde{u}_{\xi}(t,x) =\int_0^t \Phi\left(H_N(s,.)\right)(t-s,x) \dd s
\end{equation}
is a solution of~\eqref{eq:PDE_to_develop_general}, the equation verified by $u_{\xi}$. Then, by the uniqueness result stated in Proposition~\ref{prop:well_definition_general_model} and the fact that $u_{\xi}$ and $\tilde{u}_{\xi}$ have the same initial conditions, we will have that~$u_{\xi} = \tilde{u}_{\xi}$, so that~\eqref{eq:step2_keylemma} is~true. 



To do the above, we first compute $\partial_t\tilde{u}_{\xi}(t,x)$, by using that for all $f\in W^{1,1}\left(\mathbb{R}_+,L^1\left(\mathbb{R}_+\right)\right)$ such that $s\mapsto f(s,s)\in L^1\left(\mathbb{R}_+\right)$, it holds $\frac{\dd }{\dd t}\left(\int_0^tf(t,s) \dd s\right) =  \int_0^t \partial_tf(t,s) \dd s + f(t,t)$. Thereafter, we simplify the terms \hbox{$\partial_t \Phi\left(H_N(s,.)\right)(t-s,x)$} and $\Phi\left(H_N(t,.)\right)(0,x)$ that appear after the previous computation, in view of the fact that $\Phi\left(H_N(s,.)\right)$ is a solution of~\eqref{eq:PDE_jump_process_operator_Phi} with initial condition $f_0 = H_N(s,.)$. Finally, we switch the integrals, and use~\eqref{eq:PDE_alternative_taylor} to replace the integrals in $\dd s$ with the function~$\tilde{u}_{\xi}$. We obtain that for all $(t,x)\in\mathbb{R}_+\times\mathbb{R}_+^d$ 
$$
\begin{aligned}
	\partial_t \tilde{u}_{\xi}(t,x) =  &\int_0^t \partial_t\Phi\left(H_N(s,.)\right)(t-s,x) \dd s + \Phi\left(H_N(t,.)\right)(0,x)\\ 
	=  bN&\int_0^t\bigg[\int_{v\in\mathbb{R}_+^d} \Phi\left(H_N(s,.)\right)\Big(t-s,x+\frac{v}{N}\Big) \xi(\dd v) - \Phi\left(H_N(s,.)\right)(t-s,x)\bigg] \dd s + H_N(t,x)\\
	= bN&\left[\int_{v\in\mathbb{R}_+^d} \tilde{u}_{\xi}\left(t,x+\frac{v}{N}\right) \xi(\dd v) - \tilde{u}_{\xi}(t,x)\right] + H_N(t,x).
\end{aligned}
$$
The above equation is in fact the same as~\eqref{eq:PDE_to_develop_general} in view of the definition of $H_N$, see~\eqref{eq:definition_function_HN}. Then, we have that $u_{\xi} = \tilde{u}_{\xi}$, which concludes the proof of the second step, and thus of the lemma. \qed

\subsection{Proof of Lemma~\ref{lemm:key_lemma_cemetery}}\label{subsect:proof_key_lemma_cemetery}

To simplify notations, we denote  for all $y\in\mathbb{R}_+^d$  the set $A_{Ny} := \big\{v\in\mathbb{R}_+^{d}\,|\,Ny - v\notin\mathbb{R}_+^d\big\}$. We consider $v_{\xi,1}$ and $v_{\xi,2}$ the functions defined for all $t\geq0$ as
$$
\begin{aligned}
	v_{\xi,1}(t) = bN\int_{y\in\mathbb{R}_+^{d}} u_{\xi}(t,y) \xi\left(A_{Ny}\right)\mathrm{d}y,\hspace{1.5mm} \text{ and }\hspace{1.5mm}v_{\xi,2}(t)=b\sum_{i = 1}^d\int_{y\in\mathbb{R}_+^{d}} h_i(t,y)(1-G(y_i))1_{\{y_i \leq \delta\}}\mathrm{d}y.
\end{aligned}
$$
By~\eqref{eq:condition_approx_cemetery_second}, we have that $v_{\xi}= v_{\xi,1} + v_{\xi,2}$. Then, to prove this lemma, we bound $v_{\xi,1}$ and~$v_{\xi,2}$, and conclude by summing their bounds, in view of the triangle inequality.

Let $t\geq0$. We begin by bounding $v_{\xi,2}(t)$. First, for all $i\in\llbracket1,d\rrbracket$, we apply~\eqref{eq:condition_approx_cemetery_first} to bound  the function $h_i$. Then, we compute the integrals with respect to $(dy_j)_{j\in\llbracket1,d\rrbracket,\,j\neq i}$, by using that 
$$
\prod_{j\in\llbracket1,d\rrbracket,\,j\neq i}\left(\int_{y_j\in\mathbb{R}_+}\exp(-\beta y_j)dy_j\right) = \left(\int_{s\in\mathbb{R}_+}\exp(-\beta s)\dd s\right)^{d-1} = \frac{1}{\beta^{d-1}}.
$$
Finally, we use the equality~\hbox{$\int_{y_i\in[0,\delta]} y_i (1-G(y_i))dy_i = \int_{y_i\in[0,\delta]} \frac{(y_i)^2}{2} g(y_i) dy_i = \frac{m_2}{2}$} (integration by part) to compute the integral that remains. It comes

\begin{equation}\label{eq:first_approx_cemetery_model_general_intermediate_first}
	\begin{aligned}
		\big|v_{\xi,2}(t)\big| &\leq  b\sum_{i = 1}^d\frac{C'}{N}\frac{1}{\beta^{d-1}}\frac{m_2}{2}\exp\left[-bN\left(1-\mathcal{L}(\xi)\left(\frac{\beta}{N}\right)\right)t\right] \\
		&= \frac{bC'm_2d}{2\beta^{d-1}N}\exp\left[-bN\left(1-\mathcal{L}(\xi)\left(\frac{\beta}{N}\right)\right)t\right].
	\end{aligned}
\end{equation}

Now, we focus on bounding $v_{\xi,1}(t)$. To do so, we first apply~\eqref{eq:bound_key_lemma} to bound from above the term $u_{\xi}(t,y)$ in $\big|v_{\xi,1}(t)\big|$. Thereafter, we use the fact that for all $y\in\mathbb{R}^d$ it holds \hbox{$\xi\left(A_{Ny}\right) = 1 - \xi\left(\left(A_{Ny}\right)^c\right)$} to develop the bound in two different integrals. Finally, we compute the value of the first integral, which is  $\int_{y\in\mathbb{R}_+^d}\exp\left(-\beta \sum_{i =1}^dy_i\right) \dd y = \frac{1}{\beta^d}$. We get the following 
\begin{equation}\label{eq:first_approx_cemetery_model_general_intermediate_second}
	\begin{aligned}
		\big|v_{\xi,1}(t)\big| &\leq \frac{b^2tC\sigma_{\xi}}{2}\exp\left[-bN\left(1-\mathcal{L}(\xi)\left(\frac{\beta}{N}\right)\right)t\right]\int_{y\in\mathbb{R}_+^d}\exp\left(-\beta \sum_{i =1}^dy_i\right) \xi\left(A_{Ny}\right)\mathrm{d}y\\ 
		&= \frac{b^2tC\sigma_{\xi}}{2}\exp\left[-bN\left(1-\mathcal{L}(\xi)\left(\frac{\beta}{N}\right)\right)t\right] \left[\frac{1}{\beta^d}-\int_{y\in\mathbb{R}_+^d}\exp\left(-\beta \sum_{i =1}^dy_i\right) \xi\left(\left(A_{Ny}\right)^c\right)\mathrm{d}y\right].
	\end{aligned}
\end{equation}
To continue our computations, we need to find a better expression for the last term of~\eqref{eq:first_approx_cemetery_model_general_intermediate_second}. We notice that by the definition of $A_{Ny}$, it holds $\xi\left(\left(A_{Ny}\right)^c\right) = \int_{v\in\mathbb{R}_+^d} 1_{\left\{\forall i\in \llbracket1,d\rrbracket:\, y_i \geq \frac{v_i}{N}\right\}}\xi(\mathrm{d}v)$, for all $y\geq0$. Using this equality, then Fubini's theorem to switch the integrals, and finally the fact that $\int_{y_i\in\left[\frac{v_i}{N},+\infty\right)} \exp\left(-\beta y_i\right)\dd y_i = \frac{\exp\left(-\frac{\beta v_i}{N}\right)}{\beta}$ for all $i\in\llbracket1,d\rrbracket$, we~have 
$$
\begin{aligned}
	\int_{y\in\mathbb{R}_+^d}\hspace{-0.01mm}\exp\left[-\beta \sum_{i =1}^dy_i\right] \xi\left(\left(A_{Ny}\right)^c\right)\mathrm{d}y &= \int_{v\in\mathbb{R}_+^d} \left[\int_{y\in\mathbb{R}_+^d} 1_{\left\{\forall i\in \llbracket1,d\rrbracket:\, y_i \geq \frac{v_i}{N}\right\}}\exp\left(-\beta \sum_{i =1}^dy_i\right)  \mathrm{d}y\right]\xi(\mathrm{d}v) \\
	&=  \frac{1}{\beta^d}\mathcal{L}\left(\xi\right)\left(\frac{\beta}{N}\right).
\end{aligned}
$$
The above expression is what we need to continue the computations interrupted at~\eqref{eq:first_approx_cemetery_model_general_intermediate_second}. First, we plug the above equation in~\eqref{eq:first_approx_cemetery_model_general_intermediate_second}, and put the terms $\frac{1}{\beta^d}$ in factors. Then, we use~\eqref{eq:inequalities_alpha_and_beta} to bound the term $1-\mathcal{L}\left(\xi\right)\left(\frac{\beta}{N}\right)$ that appears from the previous step by $\frac{\alpha}{bN}$. We obtain
\begin{equation}\label{eq:first_approx_cemetery_model_general_intermediate_third}
\begin{aligned}
\big|v_{\xi,1}(t)\big| &\leq  \frac{b^2tC\sigma_{\xi}}{2\beta^d}\exp\left[-bN\left(1-\mathcal{L}(\xi)\left(\frac{\beta}{N}\right)\right)t\right]\left(1-\mathcal{L}\left(\xi\right)\left(\frac{\beta}{N}\right)\right)\\
&\leq \frac{btC\sigma_{\xi}\alpha}{2\beta^dN}\exp\left[-bN\left(1-\mathcal{L}(\xi)\left(\frac{\beta}{N}\right)\right)t\right].
\end{aligned}
\end{equation}
The lemma is thus proved by summing~\eqref{eq:first_approx_cemetery_model_general_intermediate_first} and~\eqref{eq:first_approx_cemetery_model_general_intermediate_third}. \qed

\subsection{Proof of Proposition~\ref{prop:approximation_PDE_one_telomere}}\label{subsect:proof_approximation_one_telomere_model}
We prove this proposition statement by statement. We first deal with Proposition~\hbox{\ref{prop:approximation_PDE_one_telomere}-$(a)$}. We need to verify the assumptions of Lemma~\ref{lemm:key_lemma_lengths} for $u_{\xi} = \overline{u}^{(1)}$ to obtain it. First, recall that Eq.~\eqref{eq:equation_difference_onetelomodel_approximant_lengths}, the equation verified by~$\overline{u}^{(1)}$, corresponds to~\eqref{eq:PDE_to_develop_general} with $F = u^{(1)}$ and $\xi = g$. Then, notice that by the left-hand side of~\eqref{eq:alternative_representation_transport_model_1_telomere} and~\hyperlink{assumption:H2}{$(H_2)$}, we have for all $(t,x)\in\mathbb{R}_+\times\mathbb{R}_+$

$$
\left|\partial_{xx}u^{(1)}(t,x)\right| = \left|n''_0(bm_1t+x)\right| \leq C_{\lambda}\exp\left(- \lambda bm_1 t-\lambda x \right),
$$
so that Eq.~\eqref{eq:condition_second_derivative_general_lemma} holds with $\alpha = \lambda bm_1$, $\beta = \lambda$ and $C = C_{\lambda}$. Finally, by the inequality $1 - e^{-x} \leq x$ for all $x\in \mathbb{R}$, one can easily obtain that
\begin{equation}\label{eq:inequality_approximation_eigenvalue}
	bN\left(1 - \mathcal{L}(g)\left(\frac{\lambda}{N}\right)\right) = bN\int_0^{+\infty} \left(1 - e^{-\frac{\lambda}{N}u}\right)g(u) \dd u \leq \lambda b\int_0^{+\infty} ug(u) \dd u = \lambda bm_1.
\end{equation}
The above is exactly Eq.~\eqref{eq:inequalities_alpha_and_beta} with the same $\alpha$ and $\beta$ as before. Then, we have that all the assumptions of Lemma~\ref{lemm:key_lemma_lengths} are verified. We therefore apply this lemma to bound $u_{\xi} = \overline{u}^{(1)}$, in view of the fact that $\sigma_{\xi} = \int_0^{\delta}u^2g(u)\dd u = m_2$. We obtain that for all $(t,x)\in\mathbb{R}_+\times\mathbb{R}_+$ 
$$
\left|\overline{u}^{(1)}(t,x)\right| \leq \frac{btC_{\lambda}m_2}{2N}\exp\left[-bN\left(1-\mathcal{L}(g)\left(\frac{\lambda}{N}\right)\right)t\right]\exp\left(-\lambda x\right).
$$
As it holds $bN \left( 1 - \mathcal{L}(g)\left(\frac{\lambda}{N}\right) \right) = bm_1\lambda_N$ by Eq.~\eqref{eq:approximation_eigenvalues}, the above inequality yields that Proposition~\ref{prop:approximation_PDE_one_telomere}-$(a)$ is true with $c'_0= C_{\lambda}\frac{ m_2}{2}$.

Now, we deal with Proposition~\ref{prop:approximation_PDE_one_telomere}-$(b)$. This time, we check the assumptions of Lemma~\ref{lemm:key_lemma_cemetery} for the function $v_{\xi} = \overline{u}_{\partial}^{(1)}$. We denote $h_1$ the function defined for all $(t,x)\in\mathbb{R}_+\times\mathbb{R}_+$ as 
\begin{equation}\label{eq:function_to_apply_lemma_cemetery}
	h_1(t,x) = u^{(1)}\left(t,\frac{x}{N}\right) - u^{(1)}(t,0).
\end{equation}
In the first term of the last line of~\eqref{eq:equation_difference_model_approximant_cemeteries}, we do the change of variable $y = \frac{v}{N}$, and replace~$\overline{u}^{(1)}$ with $u_{\xi}$ (same $u_{\xi}$ as before). In the second term, we plug Eq.~\eqref{eq:function_to_apply_lemma_cemetery}. We obtain that for all~$t\geq0$
$$
\overline{u}_{\partial}^{(1)}(t) = bN\int_0^{\frac{\delta}{N}} u_{\xi}\left(t,y\right)\left(1-G(Ny)\right) \mathrm{d}y + b\int_0^{\delta} h_1(t,v)(1-G(v)) \dd v,
$$
so that~\eqref{eq:condition_approx_cemetery_second} holds. It remains to prove~\eqref{eq:condition_approx_cemetery_first}. To do so, in view of~\eqref{eq:function_to_apply_lemma_cemetery}, we first write $h_1$ as an integral of $\partial_x u ^{(1)}$. Then, we use~\eqref{eq:alternative_representation_transport_model_1_telomere} to write $u^{(1)}$ in terms of~$n_0$. Finally, we apply~\hyperlink{assumption:H2}{$(H_2)$} and~\eqref{eq:inequality_approximation_eigenvalue} to bound $n_0'$ and integrate. We obtain that for all $(t,x)\in\mathbb{R}_+\times\mathbb{R}_+$
$$
\left|h_1(t,x)\right| = \left|\int_{0}^{\frac{x}{N}} \partial_x u^{(1)}(t,v) \dd v\right| = \left|\int_{0}^{\frac{x}{N}}  n'_0(bm_1t+v) \dd v\right| \leq \frac{C'_{\lambda}x}{N}\exp\left[-bN\left(1 - \mathcal{L}(g)\left(\frac{\lambda}{N}\right)\right) t\right].
$$
This is exactly~\eqref{eq:condition_approx_cemetery_first} with $C' = C'_{\lambda}$ .
Since all the assumptions of Lemma~\ref{lemm:key_lemma_cemetery} are verified, we apply this lemma to bound $v_{\xi} = \overline{u}_{\partial}^{(1)}$ (the constants $C$, $\sigma_{\xi}$, $\alpha$ and $\beta$ in~\eqref{eq:key_lemma_cemetery} are the same as for the first statement). We then use that $bN \left( 1 - \mathcal{L}(g)\left(\frac{\lambda}{N}\right) \right) = bm_1\lambda_N$ to simplify the term~$bN \left( 1 - \mathcal{L}(g)\left(\frac{\lambda}{N}\right) \right)$ in the bound, see~\eqref{eq:approximation_eigenvalues}. We obtain Proposition~\ref{prop:approximation_PDE_one_telomere}-$(b)$.  \qed

\section{The model with several telomeres}\label{sect:model_several_telomeres}

Having solved the single-telomere case, let us now turn to the case with $2k$ telomeres. The difference with the previous section is that here, we have to control the term in the denominator of~$\widehat{n}_0^{(2k)}$, see~\eqref{eq:definitions_estimators}. In addition, we need to prove an additional statement to Lemmas~\ref{lemm:key_lemma_lengths} and~\ref{lemm:key_lemma_cemetery} to obtain the model approximation. First, in Section~\ref{subsect:severaltelos_rewritingmodel}, we rewrite~\eqref{eq:PDE_model_telomeres_several_telos}, and explain how a model approximation can be conjectured from this rewriting. Then, in Section~\ref{subsect:plan_proof_prop_theo_severaltelos}, we present the main arguments and the auxiliary statements required to approximate our model and obtain Theorem~\ref{te:main_result}-$(b)$. Thereafter, in Section~\ref{subsect:proof_auxiliary_statements_several}, we prove these auxiliary statements. Finally, we prove the approximation of $n^{(2k)}$ in Section~\ref{subsect:proof_approximation_lengths_several}, the approximation of $n_{\partial}^{(2k)}$ in Section~\ref{subsect:proof_approximation_cemetery_several}, and Theorem~\ref{te:main_result}-$(b)$ in Section~\ref{subsect:proof_quality_estimator_several_telomeres}. In all this section, we assume that~\hyperlink{assumption:H1}{$(H_1)$} holds.

\subsection{Rewriting of the integro-differential equation and model approximation}\label{subsect:severaltelos_rewritingmodel}

As for the single-telomere model, we obtain the model approximation by letting the scaling parameter $N$ introduced in~\hyperlink{assumption:H1}{$(H_1)$} tend to infinity. We split this subsection into two parts: first, we rewrite our model in order to make $N$ appears. Then, we intuit and state a model approximation based on this rewriting.

\paragraph{Rewriting of the model.}\label{subsubsect:rewriting_model_several_telos}

Our first goal is to substitute the term $\int_{v\in\mathbb{R}_+^{2k}} n^{(2k)}(t,x+v) \tilde{\mu}\left(\dd v\right)$ in the first line of~\eqref{eq:PDE_model_telomeres_several_telos} for the term~$\int_{v\in\mathbb{R}_+^{2k}} n^{(2k)}\left(t,x+\frac{v}{N}\right)\mu(\dd v)$. To do so, we first denote for all $i\in\llbracket1,2k\rrbracket$ the $i$-th vector of the canonical basis $e_i\in\mathbb{R}_+^{2k}$. Then, we use both equalities in~\eqref{eq:scaled_measures_shortening_and_by_set} and integrate. We obtain that for all $t\geq0$, $x\in\mathbb{R}_+^{2k}$, 
\begin{equation}\label{eq:first_argument_rewriting_several}
\int_{v\in\mathbb{R}_+^{2k}} n^{(2k)}\left(t,x+\frac{v}{N}\right)\mu(\dd v) = \frac{1}{2^k}\sum_{I \in \mathcal{I}_k} \int_{(v_i,i\in I) \in \mathbb{R}_+^k} n^{(2k)}\left(t,x + \sum_{i\in I} \frac{v_i}{N}e_i\right)\left[\prod_{i\in I}g\left(v_{i}\right)\mathrm{d}v_{i}\right].
\end{equation}
We therefore need to prove that~$\int_{v\in\mathbb{R}_+^{2k}} n^{(2k)}(t,x+v) \tilde{\mu}\left(\dd v\right)$ is equal to the right-hand side of~\eqref{eq:first_argument_rewriting_several}, and the substitution will be possible. To do this, we first develop~$\tilde{\mu}$ by using~\eqref{eq:measure_shortening} and~\eqref{eq:measure_shortening_by_set}. Then, for all $I\in\mathcal{I}_k$, $i\in I$, we replace $\tilde{g}(v_i)$ with $Ng(Nv_i)$ in view of~\hyperlink{assumption:H1}{$(H_1)$}, and we integrate. Finally, we do the change of variable $v_i' = Nv_i$ in each coordinate. It comes for all $t\geq 0$,~$x\in\mathbb{R}_+^{2k}$,
$$
\begin{aligned}
	\int_{v\in\mathbb{R}_+^{2k}} n^{(2k)}(t,x+v) \tilde{\mu}\left(\dd v\right) &= \frac{1}{2^{k}}\sum_{I \in \mathcal{I}_k}\int_{(v_i,i\in I) \in \mathbb{R}_+^k} n^{(2k)}\left(t,x + \sum_{i\in I}v_i e_i\right) \left[\prod_{i\in I}Ng\left(Nv_{i}\right)\mathrm{d}v_{i}\right] \\
	&= \frac{1}{2^{k}}\sum_{I \in \mathcal{I}_k}\int_{\left(v'_i,i\in I\right) \in \mathbb{R}_+^k} n^{(2k)}\left(t,x + \sum_{i\in I}\frac{v'_i}{N} e_i\right) \left[\prod_{i\in I}g\left(v'_{i}\right)\mathrm{d}v'_{i}\right].
\end{aligned}
$$
Therefore, it holds $\int_{v\in\mathbb{R}_+^{2k}} n^{(2k)}(t,x+v) \tilde{\mu}\left(\dd v\right) = \int_{v\in\mathbb{R}_+^{2k}} n^{(2k)}\left(t,x+\frac{v}{N}\right) \mu\left(\dd v\right)$. We now rewrite the first line of~\eqref{eq:PDE_model_telomeres_several_telos} by using this last equality. First, we plug it in~\eqref{eq:PDE_model_telomeres_several_telos}. Then, we apply~\hyperlink{assumption:H1}{$(H_1)$} to replace $\tilde{b}$ with~$bN$. We obtain the following rewriting, for all $t\geq0$, $x\in\mathbb{R}_+^{2k}$,
\begin{equation}\label{eq:rewriting_several_firstline}
\partial_t n^{(2k)}(t,x) = bN\int_{v\in\mathbb{R}_+^{2k}}\left[n^{(2k)}\left(t,x + \frac{v}{N}\right)  - n^{(2k)}(t,x)\right]\mathrm{d}\mu(v).
\end{equation}

Now, we aim at rewriting the second line of~\eqref{eq:PDE_model_telomeres_several_telos}. In Section~\ref{subsect:onetelo_rewritingmodel}, the rewriting of the second line of~\eqref{eq:PDE_model_telomeres_one_telo} was obtained by proving an equality involving $\tilde{G}$ and $G$, given in~\eqref{eq:relation_between_Gtilde_and_G}. Here, we need to obtain the equivalent equality, involving $\tilde{\mu}$ and $\mu$, which is the following, for all $y\in\mathbb{R}_+^{2k}$,
\begin{equation}\label{eq:relation_between_mutilde_and_mu}
	\tilde{\mu}\left(\left\{v\in\mathbb{R}_+^{2k}\,|\, y - v \in\mathbb{R}_+^{2k}\right\}\right) = \mu\left(\left\{v\in\mathbb{R}_+^{2k}\,|\, Ny - v \in\mathbb{R}_+^{2k}\right\}\right).
\end{equation}
By first using~\eqref{eq:scaled_measures_shortening_and_by_set} to develop $\mu$, then integrating, and finally applying~\eqref{eq:relation_between_Gtilde_and_G}, we have
\begin{equation}\label{eq:secondstep_relation_between_mutilde_and_mu}
\begin{aligned}
\mu\left(\left\{v\in\mathbb{R}_+^{2k}\,|\, Ny - v \in\mathbb{R}_+^{2k}\right\}\right) &= \frac{1}{2^{k}}\sum_{I \in \mathcal{I}_k} \int_{v\in\mathbb{R}_+^{2k}} 1_{\left\{\forall i\in\llbracket1,2k\rrbracket:\, Ny_i \geq v_i\right\}}\left[\prod_{i\in I}g\left(v_{i}\right)\mathrm{d}v_{i}\right]\left[\prod_{i\notin I}\delta_0(\mathrm{d}v_{i})\right]\\ 
&= \frac{1}{2^{k}}\sum_{I \in \mathcal{I}_k}  \left(\prod_{i\in I} G\left(Ny_{i}\right)\right) = \frac{1}{2^{k}}\sum_{I \in \mathcal{I}_k} \left(\prod_{i\in I}\tilde{G}\left(y_{i}\right)\right).
\end{aligned}
\end{equation}
In addition, by doing the first two steps presented above, replacing $g$ with $\tilde{g}$, we have that the left-hand side of~\eqref{eq:relation_between_mutilde_and_mu} is equal to the last term in~\eqref{eq:secondstep_relation_between_mutilde_and_mu}. Hence, from this last result, we have that Eq.~\eqref{eq:relation_between_mutilde_and_mu} is true. We now use this last equation to rewrite the second line of~\eqref{eq:PDE_model_telomeres_several_telos}. First, we take the complement of both sets in~\eqref{eq:relation_between_mutilde_and_mu}. Then, we plug the equality obtained in the second line of~\eqref{eq:PDE_model_telomeres_several_telos}. Finally, we replace $\tilde{b}$ with $bN$, in view of~\hyperlink{assumption:H1}{$(H_1)$}. We obtain that for all $t\geq0$ 

\begin{equation}\label{eq:rewriting_several_secondline}
	n_{\partial}^{(2k)}(t) =  bN\int_{y\in\mathbb{R}_+^{2k}} n^{(2k)}\left(t,y\right)\mu\left(\left\{v\in\mathbb{R}_+^{2k}\,|\,Ny-v \notin \mathbb{R}_+^{2k}\right\}\right)\mathrm{d}y.
\end{equation}

We conclude by presenting the full rewriting of the model. By combining~\eqref{eq:rewriting_several_firstline} and~\eqref{eq:rewriting_several_secondline}, we have the following new version of~\eqref{eq:PDE_model_telomeres_several_telos}
\begin{equation}\label{eq:rescaled_PDE_model_telomeres_several_telos}
	\begin{cases}
		\partial_t n^{(2k)}(t,x) = bN\int_{v\in\mathbb{R}_+^{2k}}\left[n^{(2k)}\left(t,x + \frac{v}{N}\right)  - n^{(2k)}(t,x)\right]\mathrm{d}\mu(v), & \forall t\geq0,\,x\in\mathbb{R}_+^{2k},\\
		n^{(2k)}_{\partial}(t) = bN\int_{y\in\mathbb{R}_+^{2k}} n^{(2k)}\left(t,y\right)\mu\left(\left\{v\in\mathbb{R}_+^{2k}\,|\,Ny-v \notin \mathbb{R}_+^{2k}\right\}\right)\mathrm{d}y ,& \forall t\geq0, \\ 
		n^{(2k)}(0,x) = \prod_{i = 1}^{2k}n_0(x_i), & \forall x\in\mathbb{R}_+^{2k}.
	\end{cases}
\end{equation}
\noindent The above is the system of integro-differential equations that we study in the rest of the section.
\paragraph{Model approximation.}\label{subsubsect:approximation_model_several_telos} As for the single-telomere model, we start by giving the intuition allowing us to obtain the approximant of~\eqref{eq:rescaled_PDE_model_telomeres_several_telos}. We begin by conjecturing the approximation of the first line of~\eqref{eq:rescaled_PDE_model_telomeres_several_telos}. In view of the Taylor expansion, we have that 
$$
bN\int_{v\in\mathbb{R}_+^{2k}}\left[n^{(2k)}\left(t,x + \frac{v}{N}\right)  - n^{(2k)}(t,x)\right]\mathrm{d}\mu(v) \underset{N\rightarrow+\infty}{\approx} b\sum_{i = 1}^{2k} \int_{v\in\mathbb{R}_+^{2k}}v_i\mathrm{d}\mu(v)\partial_{x_i}n^{(2k)}\left(t,x\right).
$$
In addition, the explicit value of $b \int_{v\in\mathbb{R}_+^{2k}}v_i\mathrm{d}\mu(v)$ can be computed for all $i\in\llbracket1,2k\rrbracket$ thanks to the following statement.
\begin{lemm}[First moment of $\mu$]\label{lemm:moments_measure_mu}
	Assume that \hyperlink{assumption:H1}{$(H_1)$} holds. Then, for all $i\in\llbracket1,2k\rrbracket$, we have
	$$
	\int_{v\in\mathbb{R}_+^{2k}}v_i\mathrm{d}\mu(v) = \frac{m_1}{2}.
	$$ 
\end{lemm}
\begin{proof}
	Let $i\in\llbracket1,2k\rrbracket$. To obtain the above expression, we first develop $\mu$ by using the right-hand side of~\eqref{eq:scaled_measures_shortening_and_by_set}. Then, we integrate both measures, in view of the fact that for all $I\in\mathcal{I}_k$, it holds by the left-hand side of~\eqref{eq:scaled_measures_shortening_and_by_set}: $\int_{v\in\mathbb{R}_+^{2k}} v_i \mu^{(I)}(\dd v) = m_11_{\{i\in I\}}$. We obtain 
	$$
	\int_{v\in\mathbb{R}_+^{2k}}v_i\mathrm{d}\mu(v)  = \frac{1}{2^{k}}\sum_{I \in \mathcal{I}_k} m_11_{\{i\in I\}} = \frac{m_1\#\left(\left\{I\in\mathcal{I}_k\,|\,i\in I\right\}\right)}{2^k}.
	$$
	The lemma is then proved by using Lemma~\ref{lemm:zero_singleton_cardinal} to compute the ratio.
\end{proof}
\noindent Therefore, by denoting the vector $\mathbb{1}_{2k} := \left(1,\hdots,1\right) \in \mathbb{R}^{2k}$, our conjecture is that~\eqref{eq:rescaled_PDE_model_telomeres_several_telos} can be approximated by the  function~\hbox{$u^{(2k)}\in C\left(\mathbb{R}_+,W^{2,1}\left(\mathbb{R}_+^{2k}\right)\right)$}, solution of the following equation 
\begin{equation}\label{eq:approx_first_line_several}
	\partial_t u^{(2k)}(t,x) = \frac{bm_1}{2}\nabla.\left(u^{(2k)}\mathbb{1}_{2k}\right)(t,x).
\end{equation}

We now intuit how the second line of~\eqref{eq:rescaled_PDE_model_telomeres_several_telos} can be approximated. In comparison with what is done in Section~\ref{subsect:onetelo_rewritingmodel}, it is necessary to decompose $n_{\partial}^{(2k)}$ into two terms. Notice that by the second equality in~\eqref{eq:secondstep_relation_between_mutilde_and_mu}, the fact that $G(x) = 1$ for all $x>\delta$, and~Lemma~\ref{lemm:zero_singleton_cardinal},~we have
\begin{equation}\label{eq:decomposition_cemetery_first_reason}
	\forall y\in\left(\frac{\delta}{N},+\infty\right)^{2k}:\hspace{2.5mm}\mu\left(\left\{v\in\mathbb{R}_+^{2k}\,|\,Ny-v \notin \mathbb{R}_+^{2k}\right\}\right) = 1 - \frac{1}{2^{k}}\sum_{I \in \mathcal{I}_k} \left(\prod_{i\in I}G\left(Ny_{i}\right)\right) = 0.
\end{equation}
Then, combining~the above with the second line of~\eqref{eq:rescaled_PDE_model_telomeres_several_telos} yields that our decomposition of~$n_{\partial}^{(2k)}$ is the following, for all $t\geq0$,  
\begin{equation}\label{eq:decomposition_cemetery_several}
	\begin{aligned}
		n^{(2k)}_{\partial}(t) &= bN\int_{y\in\mathbb{R}_+^{2k}} n^{(2k)}\left(t,y\right)\mu\left(\left\{v\in\mathbb{R}_+^{2k}\,|\,Ny-v \notin \mathbb{R}_+^{2k}\right\}\right)1_{\left\{\#\left(i\in\llbracket1,2k\rrbracket\,|\,y_i \leq \frac{\delta}{N}\right) = 1\right\}}\mathrm{d}y \\ 
		&+bN\int_{y\in\mathbb{R}_+^{2k}} n^{(2k)}\left(t,y\right)\mu\left(\left\{v\in\mathbb{R}_+^{2k}\,|\,Ny-v \notin \mathbb{R}_+^{2k}\right\}\right)1_{\left\{\#\left(\left\{i\in\llbracket1,2k\rrbracket\,|\,y_i \leq \frac{\delta}{N}\right\}\right) \geq 2\right\}}\mathrm{d}y \\
		&=: n^{(2k)}_{\partial,1}(t) +  n^{(2k)}_{\partial,2}(t).
	\end{aligned}
\end{equation}
In fact, $n^{(2k)}_{\partial,1}$ and $n^{(2k)}_{\partial,2}$ have a different behaviour as $N \rightarrow +\infty$. The behaviour of $n_{\partial,2}^{(2k)}$ is easy to conjecture. As the domain of integration of the integral in its definition correspond to points with at least two coordinates smaller than $\frac{\delta}{N}$, we have that this integral vanishes  at least at a rate $\frac{1}{N^2}$ as~$N \to +\infty$. Then, by multiplying it by~$bN$, we obtain that for all $t\geq0$
\begin{equation}\label{eq:behaviour_second_term_cemetery}
	n_{\partial,2}(t) = bN\times O\left(\frac{1}{N^2}\right) \underset{N\rightarrow+\infty}{\approx} 0.
\end{equation}
To conjecture the behaviour of $n_{\partial,1}^{(2k)}$ as $N\rightarrow+\infty$, we need to obtain a better expression for~it. To do so, in its definition given in the first line of~\eqref{eq:decomposition_cemetery_several}, we first decompose the indicator, by using that for all $y\in\mathbb{R}_+^{2k}$, we have
$$
1_{\left\{\#\left(i\in\llbracket1,2k\rrbracket\,|\,y_i \leq \frac{\delta}{N}\right) = 1\right\}} = \sum_{i = 1}^{2k} 1_{\left\{y_i \leq \frac{\delta}{N},\,\forall j\in\llbracket1,2k\rrbracket\backslash\{i\}:\, y_j > \frac{\delta}{N}\right\}}.
$$
Then, we use the following equality to replace the term $\mu\left(\left\{v\in\mathbb{R}_+^{2k}\,|\,Ny-v \notin \mathbb{R}_+^{2k}\right\}\right)$ in the first line of~\eqref{eq:decomposition_cemetery_several}. This equality is obtained by applying the second equality in~\eqref{eq:secondstep_relation_between_mutilde_and_mu}, the fact that~$G(x) = 1$ when $x>\delta$, and~Lemma~\ref{lemm:zero_singleton_cardinal}. For all $i\in\llbracket1,2k\rrbracket$ and $y\in\mathbb{R}_+^{2k}$ verifying $y_i \leq \frac{\delta}{N}$ and $y_j > \frac{\delta}{N}$ when $j\neq i$, it holds
$$
\begin{aligned}
	\mu\left(\left\{v\in\mathbb{R}_+^{2k}\,|\,Ny-v \notin \mathbb{R}_+^{2k}\right\}\right) &= 1 - \frac{1}{2^{k}}\sum_{I \in \mathcal{I}_k} \left(\prod_{j\in I}G\left(Ny_{j}\right)\right) \\ 
	&= 1 - \frac{1}{2^{k}}\sum_{I \in \mathcal{I}_k} \left(G\left(Ny_{i}\right)1_{\{i \in I\}} + 1_{\{i\notin I\}}\right) = \frac{1}{2}\left(1-G(Ny_i)\right).
\end{aligned}
$$
Finally, for each index $i\in\llbracket1,2k\rrbracket$ of the sum, we do the change of variables $y_i' = N y_i$, and $y'_j = y_j$ for all $j\in\llbracket1,2k\rrbracket\backslash\{i\}$. We obtain that for all $t\geq0$
\begin{equation}\label{eq:final_formula_first_term_cemetery}
	\begin{aligned}
		n_{\partial,1}^{(2k)}(t) &= \frac{bN}{2}\sum_{i = 1}^{2k}\int_{y\in\mathbb{R}_+^{2k}} n^{(2k)}\left(t,y\right)\left(1-G(Ny_i)\right)1_{\left\{y_i \leq \frac{\delta}{N},\,\forall j\in\llbracket1,2k\rrbracket\backslash\{i\}:\, y_j > \frac{\delta}{N}\right\}}\mathrm{d}y \\
		&= \frac{b}{2}\sum_{i = 1}^{2k}\int_{y'\in\mathbb{R}_+^{2k}} n^{(2k)}\left(t,\sum_{j = 1,\,j\neq i}^{2k} y'_j e_j + \frac{y'_i}{N}e_i\right)\left(1-G(y'_i)\right)1_{\left\{y'_i \leq \delta,\,\forall j\in\llbracket1,2k\rrbracket\backslash\{i\}:\, y'_j > \frac{\delta}{N}\right\}}\mathrm{d}y'.
	\end{aligned}
\end{equation}
Then, we are now able to conjecture the behaviour of~$n^{(2k)}_{\partial,1}$ as $N\rightarrow +\infty$. Indeed, by letting~$N$ tend to infinity, and by using that $\int_{s\in[0,\delta]} (1-G(s)) \dd s = m_1$, we have for all $t\geq 0$
$$
n^{(2k)}_{\partial,1}(t) \underset{N\rightarrow+\infty}{\approx} \frac{bm_1}{2}\sum_{i = 1}^{2k}\int_{y'\in\mathbb{R}_+^{2k}} n^{(2k)}\left(t,y \right)\delta_0(dy_i)\left(\overset{2k}{\underset{j = 1,\,j\neq i}{\prod}} dy_j\right).
$$
Combining the above with~\eqref{eq:behaviour_second_term_cemetery} and~\eqref{eq:approx_first_line_several} finally yields that the following system seems a good approximant of~\eqref{eq:rescaled_PDE_model_telomeres_several_telos}
\begin{equation}\label{eq:approximation_transport_model_2k_telomeres}
	\begin{cases}
		\partial_t u^{(2k)}(t,x) = \frac{bm_1}{2}\nabla.\left(u^{(2k)}\mathbb{1}_{2k}\right)(t,x), & \forall t\geq0,\,x\in\mathbb{R}_+^{2k},\\
		u_{\partial}^{(2k)}(t) =  \frac{bm_1}{2} \overset{2k}{\underset{i = 1}{\sum}}\int_{y\in\mathbb{R}_+^{2k}} u^{(2k)}(t, y) \delta_0(dy_i)\left(\overset{2k}{\underset{j = 1,\,j\neq i}{\prod}} dy_j\right), &\forall t\geq0, \\ 
		u^{(2k)}(0,x) = \overset{2k}{\underset{i = 1}{\prod}}n_0(x_i), & \forall x\in\mathbb{R}_+^{2k}.
	\end{cases}
\end{equation}
\begin{rem}\label{rem:other_expression_approx_cemtery}
	By using the method of characteristics to solve the first line of~\eqref{eq:approximation_transport_model_2k_telomeres}, we have for all $(t,x)\in\mathbb{R}_+\times\mathbb{R}_+^{2k}$
	\begin{equation}\label{eq:characteristics_density_lengths}
		u^{(2k)}(t,x) = \prod_{i = 1}^{2k} n_0\left(\frac{bm_1}{2}t + x_i\right).
	\end{equation}
	In addition, by plugging~\eqref{eq:characteristics_density_lengths} in the second line of~\eqref{eq:approximation_transport_model_2k_telomeres}, and then doing the change of variables $y' = y + \frac{bm_1t}{2}\mathbb{1}_{2k}$, we have for all $t\geq0$
	\begin{equation}\label{eq:characteristics_density_cemetery}
		u_{\partial}^{(2k)}(t)  = \frac{bm_1}{2} \overset{2k}{\underset{i = 1}{\sum}}n_0\left(\frac{bm_1t}{2}\right)\prod_{\substack{j = 1 \\j\neq i}}^{2k} \left[\int_{\frac{bm_1}{2}t}^{+\infty}n_0\left(y'_j\right) \dd y'_j\right]  = kbm_1n_0\left(\frac{bm_1t}{2}\right)\left[\int_{\frac{bm_1t}{2}}^{+\infty} n_0\left(s\right)\dd s\right]^{2k-1}.
	\end{equation}
\end{rem}
\begin{rem}\label{rem:mass_conservation_transport}
	The conservation of the number of individuals, presented in Remark~\ref{rem:conservation_individual_several_telos}, still holds for~\eqref{eq:approximation_transport_model_2k_telomeres} when \hyperlink{assumption:H2}{$(H_2)-(H_3)$} are verified. To obtain it, first integrate~\eqref{eq:characteristics_density_cemetery}, in view of the fact that for all \hbox{$f\in W^{1,1}\left(\mathbb{R}_+\right)\cap W^{1,2k}\left(\mathbb{R}_+\right)\cap L^{2k-1}\left(\mathbb{R}_+\right)$} it holds~$\left(f^{2k}\right)' = 2kf'f^{2k-1}$. Then, do the change of variable $s' = s - \frac{bm_1}{2}t$. Finally, apply~\eqref{eq:characteristics_density_lengths}. It comes for all $t\geq0$ 
	\begin{equation}\label{eq:link_tail_cemetery_initial}
		\int_t^{+\infty} u_{\partial}^{(2k)}(s) \dd s = \left[\int_{\frac{bm_1t}{2}}^{+\infty} n_0\left(s\right)\dd s\right]^{2k} = \left[\int_{0}^{+\infty} n_0\left(\frac{bm_1}{2}t+s'\right)\dd s'\right]^{2k} = \int_{x\in\mathbb{R}_+^{2k}}u(t,x)\dd x. 
	\end{equation}
\end{rem}
\begin{rem}\label{rem:initial_telomere_length_correlations}
Even in the case where the assumption presented in~\eqref{rem:initial_lengths_assumptions} does not hold, the computations done to obtain the first equality in~\eqref{eq:characteristics_density_cemetery} can be done, and yield that for all $t\geq0$
$$
u_{\partial}^{(2k)}(t) = \frac{bm_1}{2}\sum_{i = 1}^{2k}\int_{\left(y'_j,\,j\in\llbracket1,2k\rrbracket\backslash\{i\}\right)\in\mathbb{R}_+^{2k-1}}n\left(0,\frac{bm_1t}{2} e_i+\sum_{j = 1, \,j\neq i}^{2k}y'_je_j\right)\left[\overset{2k}{\underset{j = 1,\,j\neq i}{\prod}} 1_{\left\{y'_j \geq \frac{bm_1t}{2}\right\}}dy'_j\right].
$$
We thus have by integrating, and then doing the change of variable $y'' = \frac{bm_1s}{2} e_i+\sum_{j = 1, \,j\neq i}^{2k}y'_je_j$, that for all $t\geq0$
$$
\int_t^{+\infty}u_{\partial}^{(2k)}(s) \dd s = \sum_{i = 1}^{2k}\int_{y''\in\mathbb{R}_+^{2k}} n\left(0,y''\right)1_{\left\{y''_i\geq \frac{bm_1t}{2}\right\}}1_{\left\{\forall j\in\llbracket1,2k\rrbracket:\, y''_j\geq y''_i\right\}} \dd y.
$$
The latter yields, in view of the fact that $y''_j \geq y''_i$ for all $j\neq i$, that for all $t\geq0$
$$
\int_t^{+\infty}u_{\partial}^{(2k)}(s) \dd s = \int_{y''\in\mathbb{R}_+^{2k}} n\left(0,y''\right)1_{\left\{\min_{1\leq i \leq 2k}\left(y''_i\right) \geq \frac{bm_1t}{2}\right\}} \dd y''.
$$
Hence, $u_{\partial}^{(2k)}$  is the distribution of the minimum initial telomere length, rescaled by $\frac{bm_1}{2}$.

\end{rem}
All we have intuited above can in fact be justified rigorously. Specifically, the proposition below provides that the error between $\left(n^{(2k)},n_{\partial}^{(2k)}\right)$  and $\left(u^{(2k)},u_{\partial}^{(2k)}\right)$ tends to $0$ when $N\rightarrow+\infty$. Its proof is separated into two parts, one part for each statement, done in Sections~\ref{subsect:proof_approximation_lengths_several} and~\ref{subsect:proof_approximation_cemetery_several}.

\begin{prop}[Pointwise approximation errors, several telomeres]\label{prop:approximation_PDE_several_telomeres}
	We recall the constant~$\lambda'_N$ defined in~\eqref{eq:approximation_eigenvalues}. The following statements hold.
	\begin{enumerate}[$(a)$]
		\item Assume \hyperlink{assumption:H1}{$(H_1)-(H_3)$}. Then, there exists $d'_0 > 0$ such that for all $t\geq 0$, $x\in\mathbb{R}_+^{2k}$, we have
		$$
			\left|n^{(2k)}(t,x) - u^{(2k)}(t,x)\right| \leq d'_0\left(\lambda D_{\lambda}\right)^{2k}\frac{k^2bt}{2N}\exp\left(-kbm_1\lambda'_Nt\right)\exp\left(-\lambda \sum_{i =1}^{2k}x_i\right),
		$$
		where $d'_0$ depends only on $g$, $\lambda$, $C_{\lambda}$, $C'_{\lambda}$ and $D_{\lambda}$.
		\item Assume \hyperlink{assumption:H1}{$(H_1)-(H_3)$}. Then, there exists $d'_1 > 0$ such that for all $t\geq 0$, we have
		\begin{equation}\label{eq:approximation_cemetery_several}
			\left|n_{\partial}^{(2k)}(t) - u_{\partial}^{(2k)}(t)\right| \leq \frac{bd'_1}{N}\left(D_{\lambda}\right)^{2k}\left(k^3 \frac{b m_1}{2}t + k^2+ k\right)\exp\left(-kbm_1 \lambda'_Nt\right), 
		\end{equation}
		where $d'_1$ depends only on $g$, $\delta$, $\lambda$, $C_{\lambda}$, $C'_{\lambda}$ and $D_{\lambda}$. 
	\end{enumerate}
\end{prop}
\begin{rem}
As for Proposition~\ref{prop:approximation_PDE_one_telomere}, a result providing bounds on the approximation errors in all Lebesgue spaces can be obtained from the above, that we do not state. 
\end{rem}
\noindent We now present the main arguments and the auxiliary statements required to prove Proposition~\ref{prop:approximation_PDE_several_telomeres} and Theorem~\ref{te:main_result}-$(b)$.
\subsection{Plan of the proofs of Proposition~\ref{prop:approximation_PDE_several_telomeres} and Theorem~\ref{te:main_result}-\texorpdfstring{$(b)$}{b)}}\label{subsect:plan_proof_prop_theo_severaltelos}

We begin by presenting the plan to prove Proposition~\ref{prop:approximation_PDE_several_telomeres}-$(a)$. To get this statement, we need to control the absolute value of~$\overline{u}^{(2k)} := n^{(2k)} - u^{(2k)}$. By proceeding as when we obtained~\eqref{eq:equation_difference_onetelomodel_approximant_lengths}, using this time Lemma~\ref{lemm:moments_measure_mu} to write the term $\frac{bm_1}{2}\nabla.\left(u^{(2k)}\mathbb{1}_{2k}\right)(t,x)$~coming from~\eqref{eq:approximation_transport_model_2k_telomeres} as an integral, we have that for all $t\geq0$, $x\in\mathbb{R}_+^{2k}$ 
\begin{equation}\label{eq:equation_difference_severaltelosmodel_approximant_lengths}
\begin{aligned}
\partial_t\overline{u}^{(2k)}(t,x)  &= bN\int_0^{\delta}\left[\overline{u}^{(2k)}\left(t,x + \frac{v}{N}\right)- \overline{u}^{(2k)}(t,x)\right]\mu\left(\mathrm{d}v\right)  \\ 
&+ bN\int_0^{\delta}\left[u^{(2k)}\left(t,x + \frac{v}{N}\right)- u^{(2k)}\left(t,x\right) - \sum_{i = 1}^{2k} \frac{v_i}{N}\partial_{x_i}u^{(2k)}(t,x)\right]\mu\left(\mathrm{d}v\right).
\end{aligned}
\end{equation}
Then, $\overline{u}^{(2k)}$ is a solution of an equation of the form presented in~\eqref{eq:PDE_to_develop_general}, with $d = 2k$, $u_{\xi} = \overline{u}^{(2k)}$, \hbox{$\xi  = \mu$} and~\hbox{$F = u^{(2k)}$}. From this result, a natural plan to prove Proposition~\ref{prop:approximation_PDE_several_telomeres} is to proceed as in the proof of Proposition~\ref{prop:approximation_PDE_one_telomere}-$(a)$, and check the assumptions of Lemma~\ref{lemm:key_lemma_lengths} to obtain \hbox{Proposition~\ref{prop:approximation_PDE_several_telomeres}-$(a)$}. Here, we do not do this directly because we must first obtain the value of the constant~$\sigma_{\xi}$, defined in Lemma~\ref{lemm:key_lemma_lengths}, for $\xi = \mu$. This value is given in the following statement, proved in Section~\ref{subsubsect:proof_value_sigma}.  
\begin{lemm}[Sum of second moments of $\mu$]\label{lemm:value_sigma}
	Assume that \hyperlink{assumption:H1}{$(H_1)$} holds. Then, we have
	$$
	\sigma_{\mu} := \sum_{1\leq \ell,\ell'\leq 2k}\int_{v\in\mathbb{R}_+^{2k}} v_\ell v_{\ell'} \mu(\dd v) = \left(m_1\right)^2k^2 + \left(m_2 - \left(m_1\right)^2\right)k.
	$$
\end{lemm}
\noindent Now that we know the value of $\sigma_{\mu}$, we prove Proposition~\ref{prop:approximation_PDE_one_telomere}-$(a)$ in Section~\ref{subsect:proof_approximation_lengths_several} by following the steps presented above. 

To prove Proposition~\ref{prop:approximation_PDE_several_telomeres}-$(b)$, we control \hbox{$\big|\overline{u}_{\partial}^{(2k)}\big| := \big|n_{\partial}^{(2k)} - u_{\partial}^{(2k)}\big|$}. To do this, a decomposition of $\overline{u}_{\partial}^{(2k)}$ is required, inspired by the one given in~\eqref{eq:decomposition_cemetery_several}. Let us introduce for all~$t\geq0$
\begin{equation}\label{eq:decomposition_function_u}
	\begin{aligned}
		u^{(2k)}_{\partial,0}(t) &:= bN\int_{y\in\mathbb{R}_+^{2k}} u^{(2k)}\left(t,y\right)\mu\left(\left\{v\in\mathbb{R}_+^{2k}\,|\,Ny-v \notin \mathbb{R}_+^{2k}\right\}\right)\mathrm{d}y, \\
		u^{(2k)}_{\partial,1}(t) &:= \frac{b}{2}\sum_{i = 1}^{2k} \int_{y'\in\mathbb{R}_+^{2k}} u^{(2k)}\left(t,\sum_{j = 1,\,j\neq i}^{2k}y'_j e_j + \frac{y'_i}{N}\right) \left(1 - G(y'_i)\right)1_{\{y'_i \leq \delta\}}\mathrm{d}y',\\
		u^{(2k)}_{\partial,2}(t) &:= u^{(2k)}_{\partial,0}(t) - u^{(2k)}_{\partial,1}(t) .
	\end{aligned}
\end{equation}
Then, as it holds $u_{\partial,0}^{(2k)} = u_{\partial,1}^{(2k)} + u_{\partial,2}^{(2k)}$, we have the following decomposition, for all~$t\geq0$,
\begin{equation}\label{eq:decomposition_difference_cemetery_several}
	\overline{u}_{\partial}^{(2k)}(t) = \left(n_{\partial}^{(2k)}(t) - u_{\partial,0}^{(2k)}(t) + u_{\partial,1}^{(2k)}(t)- u_{\partial}^{(2k)}(t)\right) + u_{\partial,2}^{(2k)}(t).
\end{equation}
This decomposition allows us to control $\big|\overline{u}_{\partial}^{(2k)}\big|$, by bounding each of the terms that compose it. The term $n_{\partial}^{(2k)} - u_{\partial,0}^{(2k)} + u_{\partial,1}^{(2k)}- u_{\partial}^{(2k)}$, on the one hand, can be bounded by using Lemma~\ref{lemm:key_lemma_cemetery}, already proved in Section~\ref{subsect:proof_key_lemma_cemetery}.  The term~$u_{\partial,2}^{(2k)}$, on the other hand, can be bounded with the following lemma, proved in Section~\ref{subsubsect:proof_integral_cemetery}. 
\begin{lemm}[Control of the probability of having several short telomeres]\label{lemm:integral_cemetery_model_several_telos}
	Assume that \hyperlink{assumption:H1}{$(H_1)$} and \hyperlink{assumption:H3}{$(H_3)$} hold. Then, there exists $\tilde{d} > 0$ such that for all $t\geq 0$ we have
	\begin{equation}\label{eq:integral_cemetery_model_several_telos}
		\begin{aligned}
			\left|u_{\partial,2}^{(2k)}(t)\right|\leq \frac{b\tilde{d}}{N}\left(D_{\lambda}\right)^{2k}k^2\exp\left(-kbm_1\lambda t\right),
		\end{aligned}
	\end{equation}
	where $\tilde{d}$ depends only on $\delta$, $\lambda$ and $g$.
\end{lemm}
\noindent As a result, we  have a control over $\big|\overline{u}_{\partial}^{(2k)}\big|$ by summing the bounds we have on the two terms presented above, in view of~\eqref{eq:decomposition_difference_cemetery_several} and the triangle inequality. We detail how we apply these lemmas and sum the bounds to prove Proposition~\ref{prop:approximation_PDE_several_telomeres}-$(b)$ in Section~\ref{subsect:proof_approximation_cemetery_several}. 

The proof of Theorem~\ref{te:main_result}-$(b)$, finally, is based on the following statement. Its proof is relatively short, and consists in first taking $t =\frac{2x}{bm_1}$ in the first equality in~\eqref{eq:link_tail_cemetery_initial}, then raising both sides of the equality to the power $\frac{1}{2k}$, and finally taking the derivative.
\begin{lemm}[Link between $n_0$ and $u_{\partial}$]\label{lemm:proof_link_initial_cemetery_several}
	Assume that \hyperlink{assumption:H1}{$(H_1)-(H_3)$} hold. Then, for all $x\geq0$, we have
	\begin{equation}\label{eq:relation_n0_cemetery}
		n_0(x) = \frac{1}{kbm_1}\frac{u_{\partial}^{(2k)}\left(\frac{2x}{bm_1}\right)}{\left(\int_{\frac{2x}{bm_1}}^{\infty} u_{\partial}^{(2k)}\left(s\right) \dd s\right)^{1 - \frac{1}{2k}}}.
	\end{equation}
\end{lemm}
\noindent From this lemma, if we replace $\widehat{n}_0^{(2k)}$ with its definition (see~\eqref{eq:definitions_estimators}), thereafter use~\eqref{eq:equality_transport_terms} to replace $\tilde{b}\tilde{m}_1$ with $bm_1$, and finally write $n_0$ with the right-hand side of~\eqref{eq:relation_n0_cemetery}, then we obtain that for all~$x\geq0$
\begin{equation}\label{eq:intermediate_inequality_error_severaltelos}
	\left|\widehat{n}_0^{(2k)}(x) - n_0(x)\right| = \frac{1}{kbm_1}\left|\frac{n_{\partial}^{(2k)}\left(\frac{2x}{bm_1}\right)}{\left((\int_{\frac{2x}{bm_1}}^{\infty} n_{\partial}^{(2k)}\left(s\right) \dd s)\right)^{1 - \frac{1}{2k}}}- \frac{u_{\partial}^{(2k)}\left(\frac{2x}{bm_1}\right)}{\left(\int_{\frac{2x}{bm_1}}^{\infty} u_{\partial}^{(2k)}\left(s\right) \dd s\right)^{1 - \frac{1}{2k}}}\right|.
\end{equation}
We thus need to control this difference, and Theorem~\ref{te:main_result}-$(b)$ will be proved. Proposition~\ref{prop:approximation_PDE_several_telomeres} allows us to control the error between $n_{\partial}^{(2k)}$ and $u_{\partial}^{(2k)}$, or their integrals. However, there is still a difficulty related to the fact that the denominator of the two terms tends to~$0$ when~$x\rightarrow+\infty$. We handle this by proving the following two lemmas, in Sections~\ref{subsubsect:proof_power_series_expansion} and~\ref{subsubsect:inequalities_tails} respectively.
\begin{lemm}[Power series expansion of the cemetery tail]\label{lemma:power_series_expansion}
	Assume that \hyperlink{assumption:H1}{$(H_1)$} and \hyperlink{assumption:H3}{$(H_3)$} hold. Let us consider the function $\tilde{n}_0\in L^{1}\left(\mathbb{R}_+^{2k}\right)$, defined for all $x\in\mathbb{R}_+^{2k}$ as
	\begin{equation}\label{eq:full_initial_distribution_n0}
		\tilde{n}_0(x) := \prod_{i = 1}^{2k} n_0(x_i).
	\end{equation}
	Then, for all $t\geq 0$, we have
	\begin{equation}\label{eq:power_series_expansion_tail_cemetery}
		\int_t^{+\infty} n_{\partial}^{(2k)}(s) \dd s = e^{-bNt} \Bigg[1 + \sum_{\ell\geq 1}  \frac{\left(bNt\right)^\ell}{\ell!}\int_{(x,v) \in \mathbb{R}_+^{2k}\times\left(\mathbb{R}_+^{2k}\right)^\ell}\tilde{n}_0\Bigg(x + \sum_{j = 1}^{\ell} \frac{v_j}{N}\Bigg) \dd x\mu(\dd v_1)\hdots\mu(\dd v_{\ell})\Bigg].
	\end{equation}
\end{lemm}
\begin{lemm}[Lower bound for cemetery tails]\label{lemma:inequalities_lower_bounds_tails}
	Assume that \hyperlink{assumption:H1}{$(H_1)-(H_4)$} hold, and recall the constant~$\omega_N'$ defined in~\eqref{eq:approximation_eigenvalues}. Then, for all~$t\geq 0$, we have
	\begin{equation}\label{eq:lowerbound_tails}
	\int_t^{+\infty} n_{\partial}^{(2k)}(s) \dd s \geq \left(D_{\omega}\right)^{2k}\exp\left(-kbm_1\omega_N' t\right) \hspace{1mm} \text{ and } \hspace{1mm} \int_t^{+\infty} u_{\partial}^{(2k)}(s) \dd s \geq \left(D_{\omega}\right)^{2k}\exp\left(-kbm_1\omega t\right).
	\end{equation}
\end{lemm}

\noindent In fact, the statement that allows us to do this control is Lemma~\ref{lemma:inequalities_lower_bounds_tails}. Lemma~\ref{lemma:power_series_expansion} corresponds to an intermediate step to obtain Lemma~\ref{lemma:inequalities_lower_bounds_tails}. Thus, the proof of Theorem~\ref{te:main_result}-$(b)$ consists in controlling the right-hand side term of~\eqref{eq:intermediate_inequality_error_severaltelos} by using Proposition~\ref{prop:approximation_PDE_several_telomeres} and Lemma~\ref{lemma:inequalities_lower_bounds_tails}. We do this in Section~\ref{subsect:proof_quality_estimator_several_telomeres}. 

We now prove all the statements given in this section, and then obtain Proposition~\ref{prop:approximation_PDE_several_telomeres} and Theorem~\ref{te:main_result}-$(b)$ from them.

\subsection{Proof of the auxiliary statements}\label{subsect:proof_auxiliary_statements_several}
This section is devoted to the proof of the auxiliary statements presented in Section~\ref{subsect:plan_proof_prop_theo_severaltelos}. These statements are proved one by one, in the same order in which they were stated.
\subsubsection{Proof of Lemma~\ref{lemm:value_sigma}}\label{subsubsect:proof_value_sigma}

We consider for all $(\ell,\ell')\in\llbracket1,2k\rrbracket^2$ the integral $I_{\ell,\ell'} := \int_{y \in\mathbb{R}_+^{2k}} y_\ell y_{\ell'} \mu(\dd v)$. One can easily see that $\sigma_{\mu} = \sum_{1\leq \ell,\ell'\leq 2k} I_{\ell,\ell'}$. Thus, our aim is to compute the values of the integrals $\left(I_{\ell,\ell'}\right)_{(\ell,\ell')\in\llbracket1,2k\rrbracket^2}$, and then conclude by summing their values. To do this, we fix $(\ell,\ell')\in\llbracket1,2k\rrbracket^2$ and do a distinction between cases.

Assume first that $\ell = \ell'$. In view of the left-hand side of~\eqref{eq:scaled_measures_shortening_and_by_set}, we have that for all $I\in\mathcal{I}_k$ 
$$
\int_{y \in\mathbb{R}_+^{2k}} y_{\ell}y_{\ell'}\mu^{(I)}\left(\dd y\right)= \begin{cases} 
	\int_{y'\in [0,\delta]} (y')^2g(y') \dd y' = m_2, &\text{if }\ell = \ell' \in I, \\
	0, & \text{otherwise.}
\end{cases}
$$
Then, by using the right-hand side of~\eqref{eq:scaled_measures_shortening_and_by_set}~and Lemma~\ref{lemm:zero_singleton_cardinal}, we obtain that 
\begin{equation}\label{eq:proof_value_sigma_case_1}
	I_{\ell,\ell} = \frac{1}{2^k}\underset{I \in \mathcal{I}_k,\,\ell = \ell'\in I}{\sum} m_2 = \frac{m_2}{2}.
\end{equation}

Assume now that $\ell \neq \ell'$. In this case, in view of the left-hand side of~\eqref{eq:scaled_measures_shortening_and_by_set}, we have that for all~$I\in\mathcal{I}_k$
\begin{equation}\label{eq:step3_approximation_lengths_intermediate_first}
	\int_{y \in\mathbb{R}_+^{2k}} y_{\ell}y_{\ell'}\mu^{(I)}\left(\dd y\right)= \begin{cases} 
		\left(\int_{y'\in [0,\delta]} y'g(y') \dd y'\right)^2 = \left(m_1\right)^2, &\text{if }\ell \in I \text{ and }\ell' \in I, \\
		0, & \text{otherwise.}
	\end{cases}
\end{equation}
In addition, by the definition of $\mathcal{I}_k$ (see~\eqref{eq:set_shortening}), we have that when $\ell = \ell' \text{ mod }k$
$$
\left\{I\in\mathcal{I}_k\,|\,\{\ell,\ell'\}\subset I\right\} = \emptyset.
$$
Then, by combining these results, we obtain that when $\ell \neq \ell'$ and $\ell = \ell' \text{ mod }k$
\begin{equation}\label{eq:proof_value_sigma_case_2}
	I_{\ell,\ell'} = \frac{1}{2^k}\sum_{I \in \mathcal{I}_k,\,\ell\in I,\,\ell'\in I} \left(m_1\right)^2 = 0.
\end{equation}
We also obtain by combining~\eqref{eq:step3_approximation_lengths_intermediate_first} with Lemma~\ref{lemm:pair_cardinal} that when $\ell \neq \ell' \text{ mod }k$
\begin{equation}\label{eq:proof_value_sigma_case_3}
	I_{\ell,\ell'} = \frac{1}{2^k}\sum_{I \in \mathcal{I}_k,\,\ell\in I,\,\ell'\in I}\left(m_1\right)^2 = \begin{cases}
		0, & \text{ if } k= 1, \\
		\frac{\left(m_1\right)^2}{4}, & \text{ if }k \geq 2.
	\end{cases}
\end{equation}

We now conclude. First, we combine~\eqref{eq:proof_value_sigma_case_1},~\eqref{eq:proof_value_sigma_case_2} and~\eqref{eq:proof_value_sigma_case_3}. Then, we use the fact that~as the set
$$
\left\{(\ell,\ell')\in\llbracket1,2k\rrbracket^2\,|\,\ell = \ell' \text{ mod }k\right\} = \left(\bigcup_{i = 1}^{2k} \left\{(i,i)\right\}\right)\bigcup\left(\bigcup_{i = 1}^{k}\left\{(i,i+k)\right\}\right)\bigcup\left(\bigcup_{i = k+1}^{2k} \left\{(i,i-k)\right\}\right)
$$
has a cardinality of $4k$, it holds
$\#\left(\left\{(\ell,\ell')\in\llbracket1,2k\rrbracket^2\,|\,\ell \neq \ell' \text{ mod }k\right\}\right) = (2k)^2 - 4k$. Finally, we use that as $k\in\mathbb{N}^*$, we have $\left((2k)^2 - 4k\right)1_{\{k\geq2\}} = (2k)^2 - 4k$. We obtain the following, which ends the proof

$$
\begin{aligned}
	\sigma_{\mu} &=   \sum_{1 \leq \ell =\ell' \leq 2k} I_{\ell,\ell}+ \sum_{\substack{1 \leq \ell \neq \ell' \leq 2k \\\text{s.t. } \ell = \ell' \text{ mod }k}} I_{\ell,\ell'}  + \sum_{\substack{1 \leq \ell \neq \ell' \leq 2k \\ \text{s.t. } \ell \neq \ell' \text{ mod }k}} I_{\ell,\ell'} \\
	&= 2k\frac{m_2}{2} + 0 + \left((2k)^2 - 4k\right)\frac{\left(m_1\right)^2}{4}1_{\{k\geq2\}} = \left(m_1\right)^2k^2 + \left(m_2 - \left(m_1\right)^2\right)k. \mathrlap{\phantom{xxxx}\hspace{1.3825mm}\qed}
\end{aligned}
$$ 
\subsubsection{Proof of Lemma~\ref{lemm:integral_cemetery_model_several_telos}}\label{subsubsect:proof_integral_cemetery}

We first need to obtain a better expression for $u_{\partial,2} = u_{\partial,0} - u_{\partial,1}$, where $u_{\partial,0}$ and $u_{\partial,1}$ are defined in~\eqref{eq:decomposition_function_u}. To do so, we develop the function $u_{\partial,0}$ in its definition.  One can notice that by the definition of~$u_{\partial,0}$ and~\eqref{eq:decomposition_cemetery_first_reason}, it holds for all $t\geq0$
$$
\begin{aligned}
	u_{\partial,0}(t) &= bN\int_{y\in\mathbb{R}_+^{2k}} u^{(2k)}\left(t,y\right)\mu\left(\left\{v\in\mathbb{R}_+^{2k}\,|\,Ny-v \notin \mathbb{R}_+^{2k}\right\}\right)1_{\left\{\#\left(i\in\llbracket1,2k\rrbracket\,|\,y_i \leq \frac{\delta}{N}\right) = 1\right\}}\mathrm{d}y \\ 
	&+bN\int_{y\in\mathbb{R}_+^{2k}} u^{(2k)}\left(t,y\right)\mu\left(\left\{v\in\mathbb{R}_+^{2k}\,|\,Ny-v \notin \mathbb{R}_+^{2k}\right\}\right)1_{\{\#\left(i\in\llbracket1,2k\rrbracket\,|\,y_i \leq \frac{\delta}{N}\right) \geq 2\}}\mathrm{d}y.
\end{aligned}
$$
The function on the first line of the above has the same definition as $n_{\partial,1}^{(2k)}$, see~\eqref{eq:decomposition_cemetery_several}, with $u^{(2k)}$ instead of $n^{(2k)}$. Then, following exactly the same steps as those to obtain~\eqref{eq:final_formula_first_term_cemetery}, replacing $n^{(2k)}$ with $u^{(2k)}$, yields that for all $t\geq0$
\begin{equation}\label{eq:proof_integral_cemetery_intermediate_first}
	\begin{aligned}
		u_{\partial,0}(t) &=\frac{b}{2}\sum_{i = 1}^{2k}\int_{y'\in\mathbb{R}_+^{2k}} u^{(2k)}\left(t,\sum_{j = 1,\,j\neq i}^{2k} y'_j e_j + \frac{y'_i}{N}e_i\right)\left(1-G(y'_i)\right)1_{\left\{y'_i \leq \delta,\,\forall j\in\llbracket1,2k\rrbracket\backslash\{i\}:\, y'_j > \frac{\delta}{N}\right\}}\mathrm{d}y'\\
		&+bN\int_{y\in\mathbb{R}_+^{2k}} u^{(2k)}\left(t,y\right)\mu\left(\left\{v\in\mathbb{R}_+^{2k}\,|\,Ny-v \notin \mathbb{R}_+^{2k}\right\}\right)1_{\left\{\#\left(\left\{i\in\llbracket1,2k\rrbracket\,|\,y_i \leq \frac{\delta}{N}\right\}\right) \geq 2\right\}}\mathrm{d}y.
	\end{aligned}
\end{equation}
Now, we use the above equation to develop $u_{\partial,2}$. Specifically, we subtract $u_{\partial,1}$ from both sides of~\eqref{eq:proof_integral_cemetery_intermediate_first}, in view of~\eqref{eq:decomposition_function_u} and the following equality
$$
1_{\left\{y'_i \leq \delta,\,\forall j\in\llbracket1,2k\rrbracket\backslash\{i\}:\, y'_j > \frac{\delta}{N}\right\}} - 1_{\left\{y'_i \leq \delta\right\}} = -1_{\left\{y'_i \leq \delta,\,\exists j\in\llbracket1,2k\rrbracket\backslash\{i\}:\, y'_j \leq \frac{\delta}{N}\right\}}.
$$
We obtain that for all $t\geq0$
$$
\begin{aligned}
	u_{\partial,2}(t) &= - \frac{b}{2}\sum_{i = 1}^{2k}\int_{y'\in\mathbb{R}_+^{2k}} u^{(2k)}\left(t,\sum_{\ell = 1,\,\ell\neq i}^{2k} y'_\ell e_\ell + \frac{y'_i}{N}e_i\right)\left(1-G(y'_i)\right)1_{\left\{y'_i \leq \delta,\,\exists j\in\llbracket1,2k\rrbracket\backslash\{i\}:\, y'_j \leq \frac{\delta}{N}\right\}}\mathrm{d}y'\\
	&+bN\int_{y\in\mathbb{R}_+^{2k}} u^{(2k)}\left(t,y\right)\mu\left(\left\{v\in\mathbb{R}_+^{2k}\,|\,Ny-v \notin \mathbb{R}_+^{2k}\right\}\right)1_{\left\{\#\left(\left\{i\in\llbracket1,2k\rrbracket\,|\,y_i \leq \frac{\delta}{N}\right\}\right) \geq 2\right\}}\mathrm{d}y \\
	&=: -u_{\partial,3}(t) + u_{\partial,4}(t).
\end{aligned}
$$
As $-u_{\partial,3}$ and $u_{\partial,4}$ have an opposite sign, the above implies that $\left|u_{\partial,2}(t)\right| \leq \max\left(u_{\partial,3}(t),u_{\partial,4}(t)\right)$ for all $t\geq0$. We thus now obtain an upper bound for both $u_{\partial,3}$ and $u_{\partial,4}$ in order to prove~\eqref{eq:integral_cemetery_model_several_telos}. 

To bound $u_{\partial,3}$, we first bound $1_{\left\{y'_i \leq \delta,\,\exists j\in\llbracket1,2k\rrbracket\backslash\{i\}:\, y'_j \leq \frac{\delta}{N}\right\}}$ by $\sum_{j \in \llbracket1,2k\rrbracket\backslash\{i\}}1_{\left\{y'_i \leq \delta,\, y'_j \leq \frac{\delta}{N}\right\}}$. Then, we use~\eqref{eq:characteristics_density_lengths} to write~$u^{(2k)}$ in terms of $n_0$, and apply~\hyperlink{assumption:H3}{$(H_3)$} to bound $n_0$. Finally, we integrate in $\mathrm{d}y'$, and use the equality $\int_0^{\delta}(1-G(y'_i)) \dd y'_i = m_1$, the inequality $\int_{0}^{\frac{\delta}{N}} \exp\left(-\lambda y'_j \right)\mathrm{d}y'_j \leq \frac{\delta}{N}$, and the equality $\int_{0}^{+\infty} \exp\left(-\lambda y'_\ell \right)\mathrm{d}y'_\ell = \frac{1}{\lambda}$ for all~$\ell\in\llbracket1,2k\rrbracket\backslash\{i,j\}$. We obtain that for all $t\geq0$ 
\begin{equation}\label{eq:auxiliary_integral_intermediate_fourth_cemetery_model_several_telos}
	\begin{aligned}
		\left|u_{\partial,3}(t)\right| \hspace{-0.0035mm}&\leq \hspace{-0.0035mm}\frac{b}{2}\sum_{i = 1}^{2k}\sum_{\substack{j = 1\\j\neq i}}^{2k}\int_{y'\in\mathbb{R}_+^{2k}} n_0\left(\frac{bm_1}{2}t+\frac{y'_i}{N}\right)\left[\prod_{\substack{\ell = 1\\\ell\neq i}}^{2k} n_0\left( \frac{bm_1}{2}t+y_\ell\right)\right]\left(1-G(y'_i)\right)1_{\left\{y'_i \leq \delta,\, y_j \leq \frac{\delta}{N}\right\}}\mathrm{d}y' \\ 
		&\leq \hspace{-0.0035mm}\frac{bm_1\delta}{2N}\sum_{i = 1}^{2k}\sum_{\substack{j = 1\\j\neq i}}^{2k} \frac{\lambda^{2k}\left(D_{\lambda}\right)^{2k}}{\lambda^{2k-2}}\exp\left(-kbm_1\lambda t\right) = 2k(2k-1)\frac{bm_1\delta}{2N} \lambda^{2}\left(D_{\lambda}\right)^{2k}\exp\left(-kbm_1\lambda t\right) .
	\end{aligned}
\end{equation}

To bound $u_{\partial,4}$, we first bound the term $\mu\left(\left\{v\in\mathbb{R}_+^{2k}\,|\,Ny-v \notin \mathbb{R}_+^{2k}\right\}\right)$ in $u_{\partial,4}$ by~$1$, as $\mu$ is a probability measure. Then, we bound the term $1_{\left\{\#\left(\left\{i\in\llbracket1,2k\rrbracket\,|\,y_i \leq \frac{\delta}{N}\right\}\right) \geq 2\right\}}$ by the sum $\sum_{(i,j) \in \llbracket1,2k\rrbracket^2,\,i\neq j}1_{\left\{y_i \leq \frac{\delta}{N},\, y_j \leq \frac{\delta}{N}\right\}}$. Finally, as done before, we successively use Eq.~\eqref{eq:characteristics_density_lengths} to write $u^{(2k)}$ in terms of $n_0$, apply~\hyperlink{assumption:H3}{$(H_3)$} to bound from above $n_0$, and integrate by using the same inequalities/equalities. We obtain that for all $t\geq0$ 
\begin{equation}\label{eq:auxiliary_integral_intermediate_fifth_cemetery_model_several_telos}
	\begin{aligned}
		\left|u_{\partial,4}(t)\right| &\leq bN\sum_{i = 1}^{2k}\sum_{\substack{j = 1\\j\neq i}}^{2k}\int_{y\in\mathbb{R}_+^{2k}} u^{(2k)}\left(t,y\right)1_{\left\{y_i \leq \frac{\delta}{N},\, y_j \leq \frac{\delta}{N}\right\}}\mathrm{d}y  \\ 
		&\leq 2k(2k-1)\frac{b\delta^2}{N} \lambda^{2}\left(D_{\lambda}\right)^{2k}\exp\left(-kbm_1\lambda t\right).
	\end{aligned}
\end{equation}
From~\eqref{eq:auxiliary_integral_intermediate_fourth_cemetery_model_several_telos},~\eqref{eq:auxiliary_integral_intermediate_fifth_cemetery_model_several_telos} and the fact that $\left|u_{\partial,2}\right| \leq \max\left(u_{\partial,3},u_{\partial,4}\right)$, the lemma is proved. \qed

\subsubsection{Proof of Lemma~\ref{lemma:power_series_expansion}}\label{subsubsect:proof_power_series_expansion}

We consider two functions $F:\mathbb{R}_+\times\mathbb{R}_+^{2k} \rightarrow\mathbb{R}$ and $\overline{F}:\mathbb{R}_+\times\mathbb{R}_+^{2k} \rightarrow\mathbb{R}$, defined for all~\hbox{$(t,x)\in\mathbb{R}_+\times\mathbb{R}_+^{2k}$} as
\begin{equation}\label{eq:inequalities_lower_bounds_tails_intermediate_first}
	\begin{aligned}
		F(t,x) &:= \tilde{n}_0(x) + \sum_{\ell\geq 1}  \frac{\left(bN\right)^\ell t^\ell}{\ell!}\int_{v_1 \in \mathbb{R}_+^{2k}}\hdots \int_{v_{\ell} \in \mathbb{R}_+^{2k}} \tilde{n}_0\left(x + \sum_{j = 1}^{\ell} \frac{v_j}{N}\right) \mu(\dd v_\ell)\hdots \mu(\dd v_{1}), \\
		\overline{F}(t,x) &:= e^{-bNt}F(t,x).
	\end{aligned}
\end{equation}
We begin by proving that $n^{(2k)} = \overline{F}$. By deriving $F$, and then taking $m = \ell-1$, we have that for all $(t,x)\in\mathbb{R}_+\times\mathbb{R}_+^{2k}$ 
$$
\begin{aligned}
	\partial_t F(t,x) &= bN\sum_{m\geq 0}  \frac{\left(bNt\right)^{m}}{m!} \int_{v_1 \in \mathbb{R}_+^{2k}}\hdots \int_{v_{m+1} \in \mathbb{R}_+^{2k}} \tilde{n}_0\left(x + \sum_{j = 1}^{m+1} \frac{v_j}{N}\right) \mu(\dd v_{m+1})\hdots \mu(\dd v_1) \\
	&= bN \int_{v_1 \in \mathbb{R}_+^{2k}} F\left(t,x+\frac{v_1}{N}\right)\mu(\dd v_1).
\end{aligned}
$$
Then, by combining the above with the fact that $\partial_t\overline{F}(t,x) = e^{-bNt}\partial_t F(t,x) - bN e^{-bNt} F(t,x)$, we obtain that $\overline{F}$ verifies the same integro-differential equation as~$n^{(2k)}$ (see Eq.~\eqref{eq:rescaled_PDE_model_telomeres_several_telos}), so that~\hbox{$n^{(2k)} = \overline{F}$ by Proposition~\ref{prop:well_definition_general_model}.} 

Now, we prove Eq.~\eqref{eq:power_series_expansion_tail_cemetery}. By applying the equality on the right-hand side of~\eqref{eq:link_density_lengths_cemetery}, and then combining the fact that \hbox{$n^{(2k)} = \overline{F}$} with the second line of~\eqref{eq:inequalities_lower_bounds_tails_intermediate_first}, we have that for all $t\geq0$
$$
\int_{t}^{+\infty}  n_{\partial}^{(2k)}(s) \dd s =  e^{-bNt}\int_{x\in\mathbb{R}_+^{2k}} F(t,x) \dd x.
$$  
Therefore, by plugging the first line of~\eqref{eq:inequalities_lower_bounds_tails_intermediate_first} in the above equation, and then using that \hbox{$\int_{x\in\mathbb{R}_+^{2k}} \tilde{n}_0(x) \dd x = 1$} to compute the term that is not in the sum, we obtain that Eq.~\eqref{eq:power_series_expansion_tail_cemetery} is~true, which ends the proof. \qed 
\subsubsection{Proof of Lemma~\ref{lemma:inequalities_lower_bounds_tails}}\label{subsubsect:inequalities_tails}

To prove this lemma, we proceed in two steps. In Step~\hyperlink{paragraph:step1_proof_lowerbounds_tails}{$1$}, we prove the left-hand side of~\eqref{eq:lowerbound_tails}, and in Step~\hyperlink{paragraph:step2_proof_lowerbounds_tails}{$2$}, we prove the right-hand side of~\eqref{eq:lowerbound_tails}.

\paragraph{Step $1$:}\hypertarget{paragraph:step1_proof_lowerbounds_tails}{} To simplify notations, we use in this step the function $\tilde{n}_0$, defined in \eqref{eq:full_initial_distribution_n0}. Our aim here is to bound from below the right-hand side term of~\eqref{eq:power_series_expansion_tail_cemetery}. To do so, we begin by bounding the coefficients in the sum. By applying~\hyperlink{assumption:H4}{$(H_4)$}, and then using that $f_{\omega}$ is non-decreasing, we have that for all $\ell\in\mathbb{N}^*$, $x\in\mathbb{R}_+^{2k}$ and $(v_1,\hdots,v_\ell)\in\left(\mathbb{R}_+^{2k}\right)^\ell$ 
$$
\begin{aligned}
	\tilde{n}_0\left(x + \sum_{j = 1}^{\ell} \frac{v_j}{N}\right) &\geq \left(D_{\omega}\right)^{2k}\frac{\prod_{i = 1}^{2k}\left(f_{\omega}\left(x_i + \sum_{j = 1}^{\ell} \frac{(v_j)_i}{N}\right)\exp\left[-\omega\left(x_i + \sum_{j = 1}^{\ell} \frac{(v_j)_i}{N}\right)\right]\right)}{\left(\int_{0}^{+\infty} f_{\omega}(y)\exp\left(-\omega y\right) \dd y\right)^{2k}} \\
	&\geq \left(D_{\omega}\right)^{2k}\frac{\prod_{i = 1}^{2k}\left(f_{\omega}\left(x_i \right)\exp\left[-\omega x_i\right]\right)}{\left(\int_{0}^{+\infty} f_{\omega}(y)\exp\left(-\omega y\right) \dd y\right)^{2k}} \exp\left(-\omega\sum_{i = 1}^{2k}\sum_{j = 1}^{\ell} \frac{(v_j)_i}{N}\right) .
\end{aligned}
$$
Then, by integrating both sides in $\dd x$ and $\left(\mu(\dd v_i)\right)_{i\in\llbracket1,\ell\rrbracket}$, and simplifying the last term with a Laplace transform, we obtain that for all~$\ell\in\mathbb{N}^*$ 
\begin{equation}\label{eq:step_1.1_intermediate_first}
	\begin{aligned}
		&\int_{x\in\mathbb{R}_+^{2k}} \int_{v_1\in\mathbb{R}_+^{2k}}\hdots \int_{v_\ell\in\mathbb{R}_+^{2k}}\tilde{n}_0\left(x + \sum_{j = 1}^{\ell} \frac{v_j}{N}\right)  \mu(\dd v_\ell) \hdots \mu(\dd v_1)\dd x \\ 
		&\geq \left(D_{\omega}\right)^{2k}\left[\prod_{j = 1}^\ell\int_{v_j\in\mathbb{R}_+^{2k}}\exp\left(-\omega\sum_{i = 1}^{2k} \frac{(v_j)_i}{N}\right)\mu(\dd v_j)\right] = \left(D_{\omega}\right)^{2k}\left(\mathcal{L}(\mu)\left(\frac{\omega}{N}\right)\right)^\ell.
	\end{aligned}
\end{equation}
We now plug~\eqref{eq:step_1.1_intermediate_first} in the right-hand side of~\eqref{eq:power_series_expansion_tail_cemetery}. Then, we use that $1 \geq \left(D_{\omega}\right)^{2k}$ (as $D_{\omega} \leq 1$) to bound from below the term that is not in the sum. We obtain that for all $t\geq0$
\begin{equation}\label{eq:step_1.1_intermediate_second}
\begin{aligned}
	\int_t^{+\infty} n_{\partial}^{(2k)}(s) \dd s &\geq e^{-bNt}\left(D_{\omega}\right)^{2k}\left[1 +\sum_{\ell\geq 1}  \frac{\left(bN\right)^\ell t^\ell}{\ell!}\left(\mathcal{L}(\mu)\left(\frac{\omega}{N}\right)\right)^\ell\right]\\ 
	& = \left(D_{\omega}\right)^{2k}\exp\left[bN\left(\mathcal{L}(\mu)\left(\frac{\omega}{N}\right)-1\right)t\right].
\end{aligned}
\end{equation}
In addition, by first using the right-hand side of~\eqref{eq:scaled_measures_shortening_and_by_set} to develop $\mathcal{L}(\mu)$, then the left-hand side of~\eqref{eq:scaled_measures_shortening_and_by_set} to obtain that $\mathcal{L}(\mu^{(I)}) = \left(\mathcal{L}(g)\right)^{k}$ for all $I\in\mathcal{I}_k$, and finally the first equality in Lemma~\ref{lemm:zero_singleton_cardinal} to simplify the sum and the fraction, we have that for all~$p\in\mathbb{C}$ such that $\text{Re}(p) >0$
\begin{equation}\label{eq:laplace_transform_mu}
	\mathcal{L}\left(\mu\right)(p) =  \frac{1}{2^k}\sum_{I \in \mathcal{I}_k} \mathcal{L}\left(\mu^{(I)}\right)(p)= \frac{1}{2^k}\sum_{I \in \mathcal{I}_k} \left(\mathcal{L}(g)\left(p\right)\right)^{k}= \left(\mathcal{L}(g)(p)\right)^{k}.
\end{equation}
By plugging the above with $p = \frac{\omega}{N}$ in~\eqref{eq:step_1.1_intermediate_second}, in view of~\eqref{eq:approximation_eigenvalues}, we obtain that~the left-hand side of~\eqref{eq:lowerbound_tails} is~proved.

\paragraph{Step $2$:}\hypertarget{paragraph:step2_proof_lowerbounds_tails}{} First, notice that by applying~\hyperlink{assumption:H4}{$(H_4)$}, then doing the change of variable $s' = s - \frac{bm_1t}{2}$, and finally using that $f_{\omega}$ is non-decreasing, we have for all $t\geq0$
$$
\begin{aligned}
	\int_{\frac{bm_1}{2}t}^{+\infty} n_0(s) \dd s &\geq  \frac{D_{\omega}\int_{\frac{bm_1}{2}t}^{+\infty} f_{\omega}(s)\exp\left(-\omega s\right) \dd s}{\int_{0}^{+\infty} f_{\omega}(z)\exp\left(-\omega z\right) \dd z}    \\
	&=  \frac{D_{\omega}\int_{0}^{+\infty} f_{\omega}\left(s'+\frac{bm_1}{2}t\right)\exp\left(-\omega s'-\frac{ bm_1\omega}{2}t\right) \dd s'}{\int_{0}^{+\infty} f_{\omega}(z)\exp\left(-\omega z\right) \dd z}\geq D_{\omega}\exp\left(-\frac{ bm_1\omega}{2}t\right).
\end{aligned}
$$
Then, by combining the above equation with the first equality in~\eqref{eq:link_tail_cemetery_initial}, we obtain that~the right-hand side of~\eqref{eq:lowerbound_tails} is true, which ends the proof. \qed
\subsection{Proof of Proposition~\ref{prop:approximation_PDE_several_telomeres}-\texorpdfstring{$(a)$}{(a)}}\label{subsect:proof_approximation_lengths_several}

To recall, in view of~\eqref{eq:equation_difference_severaltelosmodel_approximant_lengths}, our aim is to check the assumptions of Lemma~\ref{lemm:key_lemma_lengths} for $u_{\xi} = \overline{u}^{(2k)}$, $\xi = \mu$ and $F = u^{(2k)}$, and then to apply it to obtain the statement. First, notice that by~\eqref{eq:characteristics_density_lengths}, we have for all $(\ell,\ell')\in\llbracket1,2k\rrbracket^2$, $(t,x)\in \mathbb{R}_+\times\mathbb{R}_+^{2k}$,
$$
\partial_{x_{\ell}x_{\ell'}}u^{(2k)}(t,x) = \begin{cases}
	n'_0\left(\frac{bm_1}{2}t + x_\ell\right)n'_0\left(\frac{bm_1}{2}t+x_{\ell'}\right)\prod_{i = 1, i\notin\{\ell,\ell'\}}^{2k}n_0\left(\frac{bm_1}{2}t+x_i\right), & \text{if }\ell\neq \ell', \\
	n''_0\left(\frac{bm_1}{2}t + x_\ell\right)\prod_{i = 1, i\neq \ell}^{2k}n_0\left(\frac{bm_1}{2}t + x_i\right), & \text{if }\ell= \ell'.
\end{cases}
$$
Then, by applying~\hyperlink{assumption:H2}{$(H_2)$} and~\hyperlink{assumption:H3}{$(H_3)$} to the above, we obtain that for all $(\ell,\ell')\in\llbracket1,2k\rrbracket^2$ and $(t,x)\in \mathbb{R}_+\times\mathbb{R}_+^{2k}$
$$
\left|\partial_{x_\ell x_{\ell'}}u^{(2k)}(t,x)\right| \leq \begin{cases}
	\left(C'_{\lambda}\right)^2\left(\lambda D_{\lambda}\right)^{2k-2}\exp\left(- k bm_1\lambda t-\lambda\sum_{i = 1}^{2k}x_i \right), & \text{if }\ell\neq \ell', \\
	C_{\lambda}\left(\lambda D_{\lambda}\right)^{2k-1}\exp\left(- k bm_1\lambda t-\lambda\sum_{i = 1}^{2k}x_i \right), & \text{if }\ell= \ell'.
\end{cases}
$$
This implies that~\eqref{eq:condition_second_derivative_general_lemma} holds with $C = \max\left(\left(C'_{\lambda}\right)^2,C_{\lambda}\lambda D_{\lambda}\right)\left(\lambda D_{\lambda}\right)^{2k-2}$, $\alpha = kbm_1\lambda$, and $\beta = \lambda$. In addition, in view of the inequality $1 - e^{-x} \leq x$ for all $x\in\mathbb{R}$ and Lemma~\ref{lemm:moments_measure_mu}, we have that 
\begin{equation}\label{eq:step2_approximation_lengths_intermediate_second}
bN\left[1 - \mathcal{L}(\mu)\left(\frac{\lambda}{N}\right)\right] = bN\int_{u\in\mathbb{R}_+^{2k}} \left(1 - e^{-\frac{\lambda}{N}\sum_{i = 1}^{2k}u_i}\right) \mu\left(\dd u\right)  
	\leq  b\lambda\sum_{i = 1}^{2k}\int_{u\in\mathbb{R}_+^{2k}} u_i\mu\left(\dd u\right) = kbm_1\lambda,
\end{equation}
so that~\eqref{eq:inequalities_alpha_and_beta} holds with the same $\alpha$ and $\beta$ as before. From these two results, all the assumptions of Lemma~\ref{lemm:key_lemma_lengths} are verified. We thus apply this lemma, and it comes for all~\hbox{$(t,x)\in \mathbb{R}_+\times\mathbb{R}_+^{2k}$}
\begin{equation}\label{eq:step2_approximation_lengths_intermediate_third}
	\left|\overline{u}^{(2k)}(t,x)\right| \leq \max\left(\left(C'_{\lambda}\right)^2,C_{\lambda}\lambda D_{\lambda}\right)\left(\lambda D_{\lambda}\right)^{2k-2}\frac{bt\sigma_{\mu}}{2N}\exp\left[- bN\left(1-\mathcal{L}(\mu)\left(\frac{\lambda}{N}\right)\right)t-\lambda\sum_{i = 1}^{2k}x_i\right].
\end{equation}
Now, by Lemma~\ref{lemm:value_sigma}, we have that there exists $d'_0 >0$, independent of $b$ and~$k$, such~that 
\begin{equation}\label{eq:step2_approximation_lengths_intermediate_fourth}
	\max\left(\left(C'_{\lambda}\right)^2,C_{\lambda}\lambda D_{\lambda}\right)\sigma_{\mu} \leq d'_0\left(\lambda D_{\lambda}\right)^2k^2.
\end{equation}
Moreover, in view of~\eqref{eq:laplace_transform_mu} and~\eqref{eq:approximation_eigenvalues}, we have that
\begin{equation}\label{eq:link_laplace_mu_and_approximated_eigenvalue}
	bN\left[1-\mathcal{L}(\mu)\left(\frac{\lambda}{N}\right)\right] = bN\left[1-\left(\mathcal{L}(g)\left(\frac{\lambda}{N}\right)\right)^k\right]=kbm_1\lambda'_N.
\end{equation} 
Then, by plugging~\eqref{eq:step2_approximation_lengths_intermediate_fourth} and~\eqref{eq:link_laplace_mu_and_approximated_eigenvalue} in~\eqref{eq:step2_approximation_lengths_intermediate_third}, we obtain that Proposition~\ref{prop:approximation_PDE_several_telomeres}-$(a)$ is true. \qed

\subsection{Proof of Proposition~\ref{prop:approximation_PDE_several_telomeres}-\texorpdfstring{$(b)$}{(b)}}\label{subsect:proof_approximation_cemetery_several}



In view of~\eqref{eq:decomposition_difference_cemetery_several}, we consider the function $v_{\mu} := n_{\partial}^{(2k)} - u_{\partial,0}^{(2k)} + u_{\partial,1}^{(2k)}- u_{\partial}^{(2k)}$. Assume that there exists a set of functions $(h_i)_{i\in\llbracket1,2k\rrbracket}$ from $\mathbb{R}_+\times\mathbb{R}_+^{2k}$ to $\mathbb{R}$ verifying~\eqref{eq:condition_approx_cemetery_first} with $C' = \frac{1}{2}C'_{\lambda}\left(\lambda D_{\lambda}\right)^{2k-1}$, $\xi = \mu$ and~$\beta = \lambda$, such that for all $t\geq0$
\begin{equation}\label{eq:step2_prop_approx_cemetery_intermediate_first}
	u_{\partial,1}^{(2k)}(t)- u_{\partial}^{(2k)}(t) = b\sum_{i = 1}^{2k}\int_{y\in\mathbb{R}_+^{2k}} h_i(t,y)(1-G(y_i))1_{\{y_i \leq \delta\}}\mathrm{d}y.
\end{equation}
By developing $n_{\partial}^{(2k)}$ and $u_{\partial,0}^{(2k)}$ with respectively~\eqref{eq:rescaled_PDE_model_telomeres_several_telos} and~\eqref{eq:decomposition_function_u}, we have that for all $t\geq0$
\begin{equation}\label{eq:step2_prop_approx_cemetery_intermediate_second}
	n_{\partial}^{(2k)}(t) - u_{\partial,0}^{(2k)}(t) =bN\int_{y\in\mathbb{R}_+^{2k}} \left[n^{(2k)}\left(t,y\right) - u^{(2k)}\left(t,y\right)\right]\mu\left(\left\{w\in\mathbb{R}_+^{2k}\,|\,Ny-w \notin \mathbb{R}_+^{2k}\right\}\right)\mathrm{d}y .
\end{equation}
Then, by plugging \eqref{eq:step2_prop_approx_cemetery_intermediate_first} and~\eqref{eq:step2_prop_approx_cemetery_intermediate_second} in the definition of $v_{\mu}$, we have that $v_{\mu}$ verifies~\eqref{eq:condition_approx_cemetery_second} with~$u_{\xi} = n^{(2k)} - u^{(2k)}$ and $\xi = \mu$. As $u_{\xi}$ verifies the assumptions of Lemma~\ref{lemm:key_lemma_lengths} by the proof of Proposition~\ref{prop:approximation_PDE_several_telomeres}-$(a)$, see Section~\ref{subsect:proof_approximation_lengths_several}, this means that all the assumptions of Lemma~\ref{lemm:key_lemma_cemetery} are verified for $v_{\xi} = v_{\mu}$. Then, by applying this lemma, and using~\eqref{eq:step2_approximation_lengths_intermediate_fourth} and \eqref{eq:link_laplace_mu_and_approximated_eigenvalue} to simplify the bound (as done in Section~\ref{subsect:proof_approximation_lengths_several}), we obtain that for all $t\geq0$  
$$
\left|v_{\mu}(t)\right| \leq \frac{b}{N}\left(D_{\lambda}\right)^{2k-1}\bigg(d'_0\lambda D_{\lambda}\frac{k^3b m_1t}{2} + \frac{C'_{\lambda} m_2k}{2}\bigg)\exp\left(-kbm_1 \lambda'_Nt\right).
$$
Recalling Eq.~\eqref{eq:decomposition_difference_cemetery_several}, combining the above with~\eqref{eq:integral_cemetery_model_several_telos} through a triangle inequality, and then using that $-kbm_1 \lambda \leq -kbm_1 \lambda'_N$ (as a consequence of~\eqref{eq:step2_approximation_lengths_intermediate_second} and~\eqref{eq:link_laplace_mu_and_approximated_eigenvalue}), yield that~\eqref{eq:approximation_cemetery_several} is true. Then,  Proposition~\ref{prop:approximation_PDE_several_telomeres}-$(b)$ is proved, assuming that~the set of functions $(h_i)_{i\in\llbracket1,2k\rrbracket}$ presented at the beginning of the proof exists. 

It thus remains to prove that such a sequence exists. To do so, we consider for all $i\in\llbracket1,2k\rrbracket$, $(t,y)\in\mathbb{R}_+\times\mathbb{R}_+^ {2k}$,
\begin{equation}\label{eq:step2_prop_approx_cemetery_intermediate_third}
	h_i(t,y) :=  \frac{1}{2}\left[u^{(2k)}\left(t,\sum_{j = 1,\,j\neq i}^{2k}y_je_j + \frac{y_i}{N}\right) - u^{(2k)}\left(t,\sum_{j = 1,\,j\neq i}^{2k}y_je_j \right)\right].
\end{equation}
By taking the difference between the second lines of Eq.~\eqref{eq:decomposition_function_u} and Eq.~\eqref{eq:approximation_transport_model_2k_telomeres}, and using that $m_1 = \int_{y_i\in[0,\delta]} (1 - G(y_i)) \dd y_i$ for each index~$i\in\llbracket1,2k\rrbracket$ of the sum in the second line of~\eqref{eq:approximation_transport_model_2k_telomeres}, we have that \eqref{eq:step2_prop_approx_cemetery_intermediate_first} holds with the set of functions defined in~\eqref{eq:step2_prop_approx_cemetery_intermediate_third}. In addition, by using~\eqref{eq:characteristics_density_lengths}, then writing $n_0$ as an integral of $n'_0$, and finally applying~\hyperlink{assumption:H2}{$(H_2)$} and \hyperlink{assumption:H3}{$(H_3)$}, we have that for all~$i\in\llbracket1,2k\rrbracket$, $(t,y)\in\mathbb{R}_+\times\mathbb{R}_+^{2k}$, 
$$
\begin{aligned}
	h_i(t,y) &= \frac{1}{2}\left[n_0\left(\frac{bm_1}{2}t + \frac{y_i}{N}\right) - n_0\left(\frac{bm_1}{2}t\right)\right]\left[\prod_{j = 1,\,j\neq i}^{2k} n_0\left(\frac{bm_1}{2}t + y_j\right)\right] \\
	&= \frac{1}{2}\left[\int_{0}^{\frac{y_i}{N}} n'_0\left(\frac{bm_1}{2}t + u\right) \dd u\right]\left[\prod_{j = 1,\,j\neq i}^{2k} n_0\left(\frac{bm_1}{2}t + y_j\right)\right] \\
	&\leq \frac{1}{2}C'_{\lambda}\left(\lambda D_{\lambda}\right)^{2k-1}\frac{y_i}{N}\exp\left[-kbm_1\lambda t- \lambda \sum_{\substack{j = 1,\,j\neq i}}^{2k} y_j\right].
\end{aligned}
$$
This implies, by using~\eqref{eq:step2_approximation_lengths_intermediate_second} to bound the coefficient $-kbm_1\lambda$ in the exponential, that~$\left(h_{i}\right)_{i\in\llbracket1,2k\rrbracket}$ verifies~\eqref{eq:condition_approx_cemetery_first} with \hbox{$C' = \frac{1}{2}C'_{\lambda}\left(\lambda D_{\lambda}\right)^{2k-1}$}, $\xi = \mu$ and $\beta = \lambda$. Then, from these points, we have that the set of functions we need exists, which concludes the proof of Proposition~\ref{prop:approximation_PDE_several_telomeres}-$(b)$. \qed 
\subsection{Proof of Theorem~\ref{te:main_result}-\texorpdfstring{$(b)$}{(b)}}\label{subsect:proof_quality_estimator_several_telomeres}

Let $x\geq0$. To simplify notations, we denote 
$$
\widehat{N}_0(x) := \int_{\frac{2x}{bm_1}}^{+\infty} n_{\partial}^{(2k)}\left(s\right) \dd s, \hspace{3mm}\text{ and } \hspace{3mm}\widehat{U}_0(x) := \int_{\frac{2x}{bm_1}}^{+\infty} u_{\partial}^{(2k)}\left(s\right) \dd s.
$$ 
In view of~\eqref{eq:intermediate_inequality_error_severaltelos} and the triangle inequality, the following holds
\begin{equation}\label{eq:proof_main_result_intermediate_first}
	\begin{aligned}
		\left|\widehat{n}_0^{(2k)}(x) - n_0(x)\right| &\leq \frac{\left|n_{\partial}^{(2k)}\left(\frac{2x}{bm_1}\right) - u_{\partial}^{(2k)}\left(\frac{2x}{bm_1}\right)\right|}{kbm_1\widehat{N}_0(x)^{1 - \frac{1}{2k}}} + \frac{u_{\partial}^{(2k)}\left(\frac{2x}{bm_1}\right)}{kbm_1} \left|\frac{1}{\widehat{N}_0(x)^{1 - \frac{1}{2k}}} - \frac{1}{\widehat{U}_0(x)^{1 - \frac{1}{2k}}}\right| \\
		&=: \Delta_1(x) + \Delta_2(x).
	\end{aligned}
\end{equation}
Thus, our aim is to obtain an upper bound for both $\Delta_1(x)$ and $\Delta_2(x)$. Theorem~\ref{te:main_result}-$(b)$ then comes by summing these bounds.

By applying Proposition~\ref{prop:approximation_PDE_several_telomeres}-$(b)$ to bound the numerator, and Lemma~\ref{lemma:inequalities_lower_bounds_tails} to bound the denominator, we have the following
\begin{equation}\label{eq:proof_main_result_intermediate_second}
	\begin{aligned}
		\Delta_1(x) &\leq \frac{\frac{bd'_1}{N}\left(D_{\lambda}\right)^{2k}\left(k^3x + k^2+ k\right)\exp\left(-2k \lambda'_Nx\right)}{kbm_1\left[\left(D_{\omega}\right)^{2k}\exp\left(-2k\omega_N' x\right)\right]^{1 - \frac{1}{2k}}}\\ 
		&= \frac{d'_1D_{\omega}}{Nm_1}\left(\frac{D_{\lambda}}{D_{\omega}}\right)^{2k}\left(k^2x + k+ 1\right)\exp\left(-2k\lambda'_Nx+(2k-1)\omega_N'x\right).
	\end{aligned}
\end{equation}
We thus now focus on finding an upper bound for $\Delta_2(x)$, which requires more computations. For this purpose, we begin by obtaining an intermediate inequality. In the equation below, first apply the equality $\left|\frac{1}{c} -\frac{1}{d}\right| = \frac{|c-d|}{cd}$ for $c = \widehat{N}_0(x)^{1 - \frac{1}{2k}}$ and $d = \widehat{U}_0(x)^{1 - \frac{1}{2k}}$ to develop $\Delta_2(x)$. Then, use that it holds \hbox{$\frac{1}{kbm_1}u_{\partial}^{(2k)}\left(\frac{2x}{bm_1}\right) = n_0(x)\widehat{U}_0(x)^{1 - \frac{1}{2k}}$}  by Lemma~\ref{lemm:proof_link_initial_cemetery_several} to simplify the term $\widehat{U}_0(x)^{1 - \frac{1}{2k}}$ in the denominator. Finally, apply~\hyperlink{assumption:H3}{$(H_3)$} to bound from above the term~$n_0(x)$ coming from the previous computation, and use~Lemma~\ref{lemma:inequalities_lower_bounds_tails} to bound the term $\widehat{N}_0(x)^{1-\frac{1}{2k}}$ in the denominator. It comes the following inequality
\begin{equation}\label{eq:proof_main_result_intermediate_third}
	\begin{aligned}
		\Delta_2(x)\hspace{-0.15mm} &= 	 n_0(x)\frac{\left|\widehat{N}_0(x)^{1 - \frac{1}{2k}}-\widehat{U}_0(x)^{1 - \frac{1}{2k}}\right|}{\widehat{N}_0(x)^{1 - \frac{1}{2k}}} \leq  \lambda D_{\lambda}\exp\left(-\lambda x\right)\frac{\left|\widehat{N}_0(x)^{1 - \frac{1}{2k}}-\widehat{U}_0(x)^{1 - \frac{1}{2k}}\right|}{\left(D_{\omega}\right)^{2k-1}\exp\left[-(2k-1)\omega_N' x\right]}.
	\end{aligned}
\end{equation}
To continue our computations, we need to bound the numerator of~\eqref{eq:proof_main_result_intermediate_third}. To do so, we first develop it by using the inequality $|c^{\ell}-d^{\ell}|\leq \ell\frac{|c-d|}{\min(c,d)^{1-\ell}}$ for $c = \widehat{N}_0(x)$, $d = \widehat{U}_0(x)$ and $\ell = 1 - \frac{1}{2k}$. This inequality comes from the Taylor's inequality applied to the function~$y\mapsto y^\ell$, and is true when $\ell\in(0,1)$. Thereafter, we apply Lemma~\ref{lemma:inequalities_lower_bounds_tails} to bound the term~$\frac{1}{\min(c,d)^{1-\ell}}$ coming from the previous step, in view of the fact that~$\omega'_N \leq \omega$ by~\eqref{eq:step2_approximation_lengths_intermediate_second} and~\eqref{eq:link_laplace_mu_and_approximated_eigenvalue}, that also hold when~$\lambda$ and $\lambda'_N$  are replaced with $\omega$ and $\omega'_N$ respectively. Finally, to bound the term $|c-d|$ coming from the first step, we use the following equality, which comes from~\eqref{eq:link_density_lengths_cemetery}-\eqref{eq:link_tail_cemetery_initial}, Proposition~\ref{prop:approximation_PDE_several_telomeres}-$(a)$, and the fact that~$\int_{y\in\mathbb{R}_+^{2k}}\exp\left(-\lambda \sum_{i =1}^{2k}y_i\right)\dd y = \frac{1}{\lambda^{2k}}$,
$$
\left|\widehat{N}_0(x)-\widehat{U}_0(x)\right|   =  \bigg|\int_{y\in\mathbb{R}_+^{2k}} \left(n^{(2k)}-u^{(2k)}\right)\left(\frac{2x}{bm_1},y\right) \dd y\bigg| \leq  d'_0\left( D_{\lambda}\right)^{2k}\frac{k^2x}{Nm_1}\exp\left(-2k\lambda'_Nx\right).
$$
We obtain
\begin{equation}\label{eq:proof_main_result_intermediate_fourth}
	\left|\widehat{N}_0(x)^{1 - \frac{1}{2k}}-\widehat{U}_0(x)^{1 - \frac{1}{2k}}\right| \leq \left(1 - \frac{1}{2k}\right) \frac{d'_0\left( D_{\lambda}\right)^{2k}\frac{k^2x}{Nm_1}\exp\left[-2k\lambda'_Nx\right]}{D_{\omega}\exp\left(-\omega x \right)}.
\end{equation}
Then, by plugging~\eqref{eq:proof_main_result_intermediate_fourth} in~\eqref{eq:proof_main_result_intermediate_third} and bounding the term $1-\frac{1}{2k}$ by $1$, we get the following upper bound for $\Delta_2(x)$ 
\begin{equation}\label{eq:proof_main_result_intermediate_fifth}
	\begin{aligned}
		\Delta_2(x) &\leq \lambda D_{\lambda}\exp\left(-\lambda x\right)\frac{d'_0\left( D_{\lambda}\right)^{2k}\frac{k^2x}{Nm_1}\exp\left(-2k\lambda'_Nx\right)}{\left(D_{\omega}\right)^{2k}\exp\left[-(2k-1)\omega'_Nx -\omega x\right]} \\ 
		&= \frac{d'_0\lambda D_{\lambda}}{Nm_1}\left(\frac{D_{\lambda}}{D_{\omega}}\right)^{2k}k^2x\exp\left[-\left(\lambda + 2k\lambda'_N\right)x + \left(\omega + (2k-1)\omega'_N\right)x\right].
	\end{aligned}
\end{equation}


We now conclude. As it holds $\omega \geq \lambda$, we have that
\begin{equation}\label{eq:proof_main_result_intermediate_sixth}
-2k\lambda'_N+(2k-1)\omega_N' \leq -\left(\lambda + 2k\lambda'_N\right) + \left(\omega + (2k-1)\omega'_N\right).
\end{equation}
Therefore, by plugging \eqref{eq:proof_main_result_intermediate_second}~and~\eqref{eq:proof_main_result_intermediate_fifth} in~\eqref{eq:proof_main_result_intermediate_first}, and then using~\eqref{eq:proof_main_result_intermediate_sixth} to bound the term coming from~\eqref{eq:proof_main_result_intermediate_second}, we obtain Theorem~\ref{te:main_result}-$(b)$. \qed 


\section{Estimation on simulations and experimental data}\label{sect:estimation_on_simulations}
Now that we have studied the quality of our estimators from a theoretical point of view, we verify if they work in practice. First, in Section~\ref{subsect:estimation_results_onetelo}, we present estimation results in the case where we observe $n_{\partial}^{(1)}$ and~$n_{\partial}^{(2k)}$ without noise. Then, in Section~\ref{subsect:estimation_results_probabilistics_model}, we study how our inference method can be adapted in a more realistic framework where we observe noisy values of $n_{\partial}^{(1)}$ and~$n_{\partial}^{(2k)}$, focusing in particular on the noise related to sampling. Finally, in Section~\ref{subsect:estimation_real_datas}, we test our inference method on experimental data.

\subsection{Estimation results on noise-free data}\label{subsect:estimation_results_onetelo}

We study here the case where we observe the exact values of $\widehat{n}_0^{(1)}$ and~$\widehat{n}_0^{(2k)}$. We first present in Section~\ref{subsubsect:good_estimation_results} two examples in which the estimation is satisfactory in the single-telomere model. Then, in Section~\ref{subsubsect:problems_variability}, we present the issues that arise in this model when $n_0$ has a small variability. Finally, in Section~\ref{subsect:estimation_results_severaltelos},  we present estimation results for the model in~$\mathbb{R}_+^{2k}$. 

\subsubsection{Estimation results in the single-telomere model: large variability}\label{subsubsect:good_estimation_results}
Let us start by presenting the framework we use here. We introduce for all $(\ell,\beta)\in\mathbb{N}^*\times\mathbb{R}_+^*$ the following functions, for all $x\geq0$,
\begin{equation}\label{eq:density_erlang_distribution}
	h_{\ell,\beta}(x) := \frac{\beta^{\ell}}{(\ell-1)!}x^{\ell-1}\exp\left(-\beta x\right), \hspace{2.5mm} \text{ and } \hspace{2.5mm} H_{\ell,\beta}(x) := \int_0^x \frac{\beta^{\ell}}{(\ell-1)!}y^{\ell-1}\exp\left(-\beta y\right) \dd y.
\end{equation}%
These functions correspond respectively to the probability density and the cumulative density function of an Erlang distribution with parameter $(\ell,\beta)$. In the estimations presented here, we choose $n_0$ belonging to the set~$ \left\{h_{\ell,\beta}\,|\,(\ell,\beta)\in\mathbb{N}^*\times\mathbb{R}_+^*\right\}$. The first reason is that in this case, either $\widehat{n}_0^{(1)}$ or $n^{(1)}$ can be computed thanks to an explicit formula, see Propositions~\ref{prop:explicit_solutions_h1beta} and~\ref{prop:explicit_solutions_gamma}. The second reason is that our main result assumptions are verified for these functions. Finally, the last reason is that when $\ell \geq 2$, these distributions are biologically relevant since they correspond to unimodal distributions, see~\cite{Xu2013}.
 
%
%


We assume that~\hyperlink{assumption:H1}{$(H_1)$} is verified for $b = 1$, $g = 1_{[0,1]}$, and $N$ not yet fixed. Our aim is to check numerically that the curve of $\widehat{n}_0^{(1)}$ is close to the one of~$n_0$. To do so, we plot in Figures~\ref{fig:estimation_exponential_onetelo} and~\ref{fig:estimation_gamma_onetelo} the curve of~$\widehat{n}_0^{(1)}$ as a function of telomere lengths when $n_0 \in \left\{h_{1,4},h_{2,1}\right\}$, and for $N\in\{1,5,40\}$. In each of these figures, the curve of $n_0$ is also plotted in black for comparison. We observe that when~$N$ is large, the estimation performs very well.  Indeed, in Figures~\ref{fig:estimation_exponential_onetelo} and~\ref{fig:estimation_gamma_onetelo}, the blue curves, which correspond to the estimations with $N=40$ are almost superposed with the dotted curves. The estimation results are thus very satisfactory.


\begin{figure}[!ht]
	\centering
	\begin{subfigure}[t]{0.425\textwidth}
		\centering
		\includegraphics[width = \textwidth]{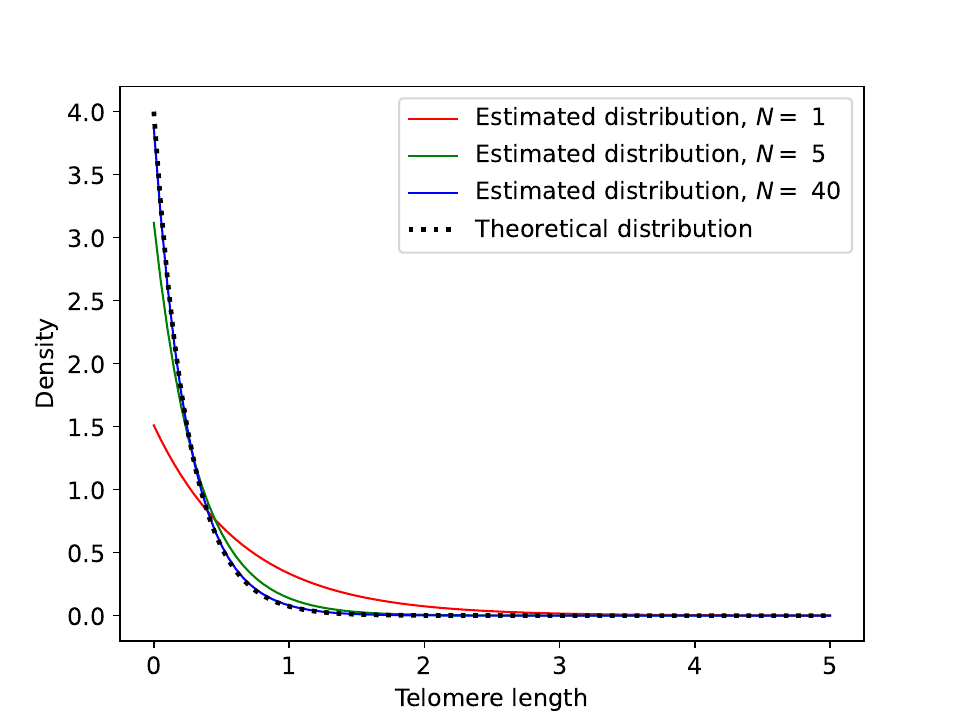}
		\caption{Estimation results when $n_0 = h_{1,4}$ and $N\in\{1,5,40\}$.}\label{fig:estimation_exponential_onetelo}
	\end{subfigure}
	\hfill
	\begin{subfigure}[t]{0.425\textwidth}
		\centering
		\includegraphics[width = \textwidth]{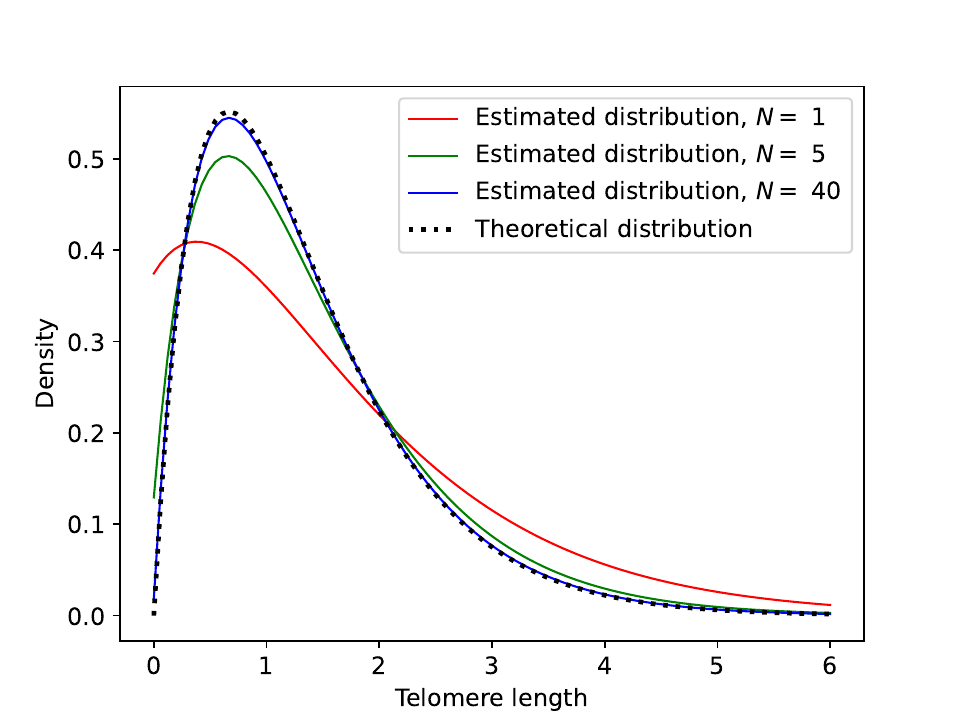} 
		\caption{Estimation results when $n_0 = h_{2,1.5}$ and $N\in\{1,5,40\}$.}\label{fig:estimation_gamma_onetelo}
	\end{subfigure}
	\caption{Estimation results in the single-telomere model for different values of $N$, when $b = 1$, $g = 1_{[0,1]}$ and $n_0\in\left\{h_{1,4},h_{2,1.5}\right\}$.}\label{fig:estimation_model_one_telomere}
\end{figure}
\begin{rem}\label{rem:choice_scaling_parameter}
We have chosen $N = 40$ for the maximum scaling parameter because this is the most realistic value for budding yeast, see the discussion about~\hyperlink{assumption:H1}{$(H_1)$} in Section~\ref{subsect:discussion_models_assumptions}.
\end{rem}

\subsubsection{Estimation results in the single-telomere model: small variability}\label{subsubsect:problems_variability}
As mentioned in the previous section, it is required to have an initial distribution with a sufficiently large coefficient of variation to ensure a good estimation. We present here the problems that occur when this is not the case and the reasons behind it. We assume that~\hyperlink{assumption:H1}{$(H_1)$} is verified with $b = 1$, $g = 1_{[0,1]}$, and~$N = 40$. We proceed to estimations of initial distributions with the same mean, but different coefficients of variation~($cv$). To do so, in view of  Eq.~\eqref{eq:density_erlang_distribution} and Proposition~\ref{prop:moments_erlang_distribution}, we first introduce for all $cv > 0$ the function $\mathcal{H}_{cv} = h_{1/cv^2,1/cv^2}$, which is the density of an Erlang distribution with mean $m = 1$ and coefficient of variation~$cv$. Then, we plot in Figure~\ref{fig:bad_estimation_small_variability} the curve of $\widehat{n}_0^{(1)}$ (blue curves), that we compare with the curve of $n_0$ (dotted curves), for~\hbox{$n_0\in\{\mathcal{H}_{1/2}, \mathcal{H}_{1/3}, \mathcal{H}_{1/5}, \mathcal{H}_{1/7}\}$}. We observe that the smaller the coefficient of variation is, the worse the estimation becomes. In particular, the spread of the initial laws is not well-captured, although the position of the mode is correctly estimated.

\begin{figure}[!ht]
	\centering
	\begin{subfigure}[t]{0.45\textwidth}
		\centering
		\includegraphics[scale = 0.425]{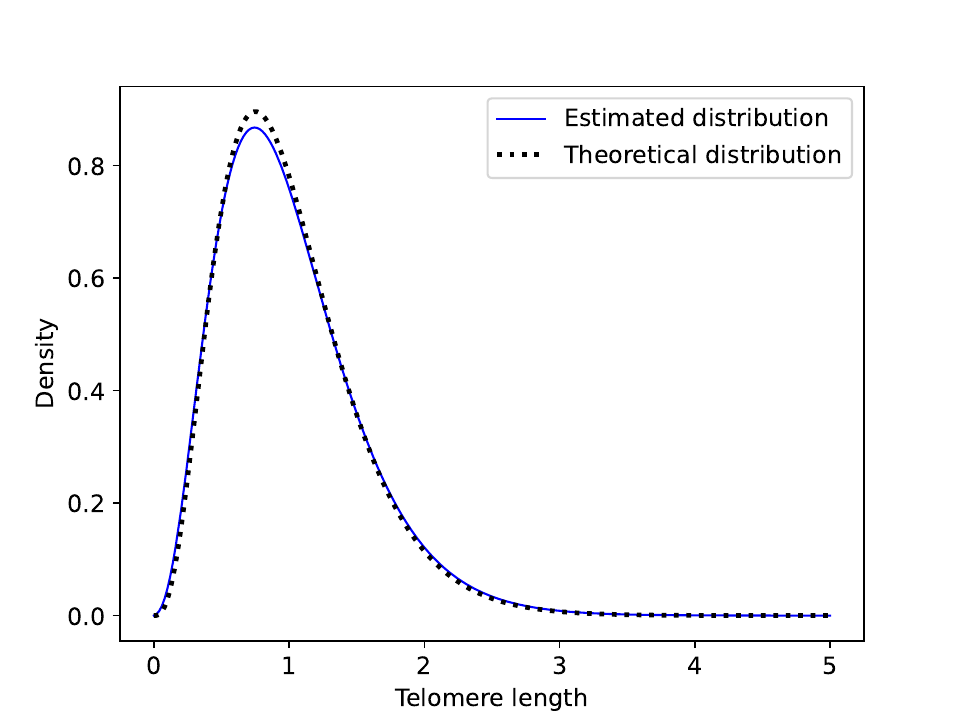}
		\caption{Estimation result when $n_0 = \mathcal{H}_{1/2}$.}
	\end{subfigure}
	\hfill
	\begin{subfigure}[t]{0.45\textwidth}
		\centering
		\includegraphics[scale = 0.425]{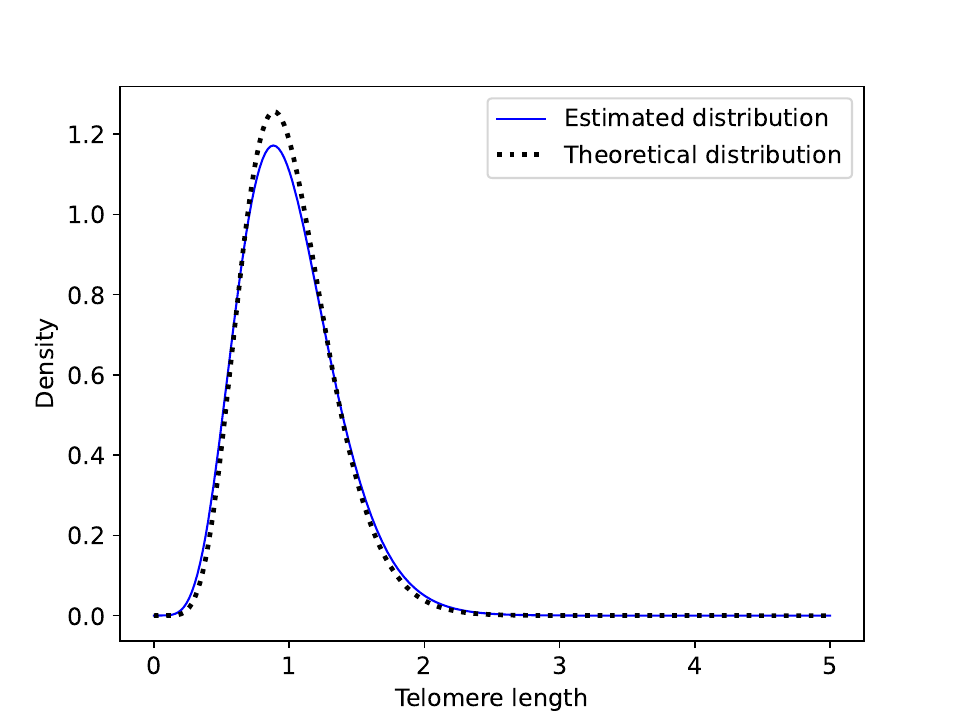}
		\caption{Estimation result when $n_0 = \mathcal{H}_{1/3}$.}
	\end{subfigure}
	\begin{subfigure}[t]{0.45\textwidth}
		\centering
		\includegraphics[scale = 0.425]{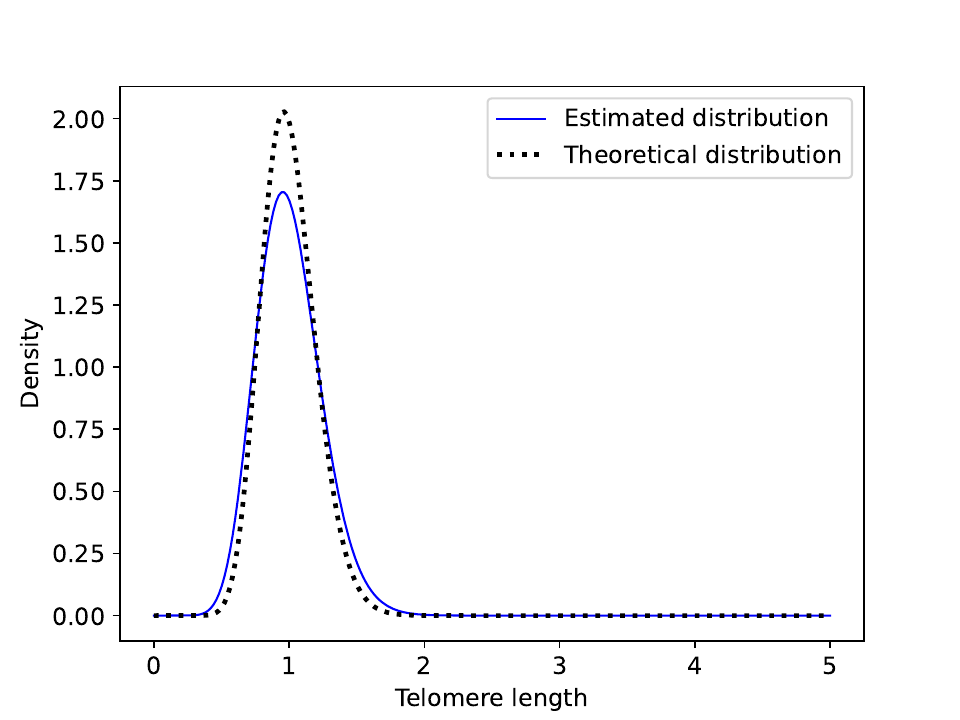}
		\caption{Estimation result when $n_0 = \mathcal{H}_{1/5}$.}
	\end{subfigure}
	\hfill
	\begin{subfigure}[t]{0.45\textwidth}
		\centering
		\includegraphics[scale = 0.425]{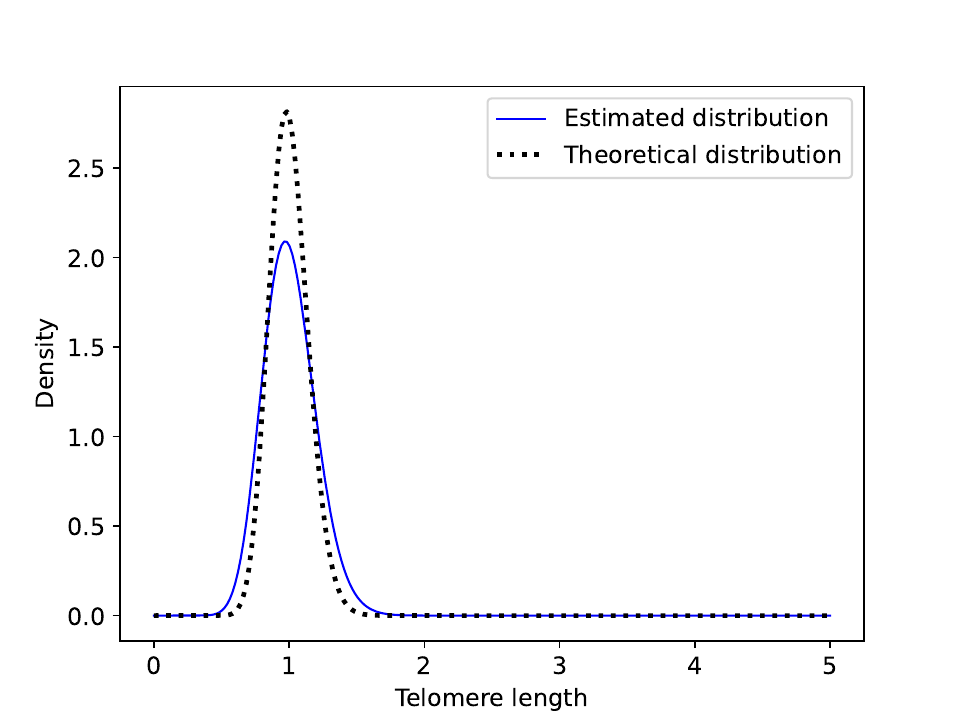}
		\caption{Estimation result when $n_0 = \mathcal{H}_{1/7}$.}
	\end{subfigure}
	\caption{Estimation results in the single-telomere model when $b = 1$, $g = 1_{[0,1]}$, $N = 40$, and $n_0\in\{\mathcal{H}_{1/2}, \mathcal{H}_{1/3}, \mathcal{H}_{1/5}, \mathcal{H}_{1/7}\}$.}\label{fig:bad_estimation_small_variability}
\end{figure}
A qualitative reason for this observation is that when the initial distribution has a small coefficient of variation, the randomness of shortening values and cell division times plays a greater role than the variability of the initial distribution. This leads to a poor estimation, as the information on $n_0$ obtained from senescence times is blurred by the information on $g$ and~$b$. A quantitative reason is that the second derivative of the initial distribution is large when its coefficient of variation is small. Thus, as the error between~$\widehat{n}_0^{(1)}$ and~$n_0$ is mainly given by the second derivative of~$n_0$ (see Sections~\ref{subsect:proof_key_lemma_lengths} and~\ref{subsect:proof_approximation_one_telomere_model}), the estimation is not reliable. To illustrate this, by using the formula in the left-hand side of~\eqref{eq:density_erlang_distribution}, we compute the $L^2$-norm of the second derivatives of~$\mathcal{H}_{1/2},\,\mathcal{H}_{1/3},\,\mathcal{H}_{1/5}, \text{ and }\mathcal{H}_{1/7}$. We obtain that \hbox{$\big|\big|\mathcal{H}_{1/2}''\big|\big|_{L^2\left(\mathbb{R}_+\right)} \approx 5.657$}, $\big|\big|\mathcal{H}_{1/3}''\big|\big|_{L^2\left(\mathbb{R}_+\right)} \approx 9.445$, $\big|\big|\mathcal{H}_{1/5}''\big|\big|_{L^2\left(\mathbb{R}_+\right)} \approx 28.17$ and \hbox{$\big|\big|\mathcal{H}_{1/7}''\big|\big|_{L^2\left(\mathbb{R}_+\right)} \approx 62.41$}, so that the values of the second derivative of $n_0$ \hbox{increase when the coefficient of variation~decreases.} 

\subsubsection{Estimation results in the model with several telomeres}\label{subsect:estimation_results_severaltelos} 

We now study how the inference method works on the model with several telomeres. To do so, we proceed in the same way as in Section~\ref{subsect:estimation_results_onetelo}. First, we assume that~\hyperlink{assumption:H1}{$(H_1)$} is verified with~\hbox{$b = 1$}, $g = 1_{[0,1]}$, and $N = 40$. Then, we plot in Figures~\ref{fig:estimation_exponential_severaltelos_small} and~\ref{fig:estimation_exponential_severaltelos_large} the curve of the estimator~$n_{0}^{(2k)}$ for different values of $k$, by using the formula given in Proposition~\ref{prop:explicit_solutions_h1beta}. In each case, this curve is compared with the one of $n_0$, plotted in dotted lines. The estimations presented in Figure~\ref{fig:estimation_exponential_severaltelos_small} correspond to estimation when $k\in\{1,3,5\}$, so to estimations in which~$k$ is small. Conversely, the estimations presented in Figure~\ref{fig:estimation_exponential_severaltelos_large} correspond to estimations when $k\in\{15,30,50\}$, so to estimations in which $k$ is large. This case is important to consider because biologists mostly study species with a large number of telomeres: budding yeast cells have $32$ telomeres ($k = 16$), and human cells have $92$ telomeres ($k =46$). We observe in Figure~\ref{fig:estimation_exponential_severaltelos_small} that  the estimated initial distributions almost overlap with the theoretical initial distribution. We also observe in Figure~\ref{fig:estimation_exponential_severaltelos_large} that the estimated curves are far from the curve of $n_0$ (dotted curve). We thus have that the estimation is satisfactory when $k$ is small, but poor when $k$ is~large.  

\begin{figure}[!ht]
	\centering
	\begin{subfigure}[t]{0.425\textwidth}
		\centering
		\includegraphics[width = \textwidth]{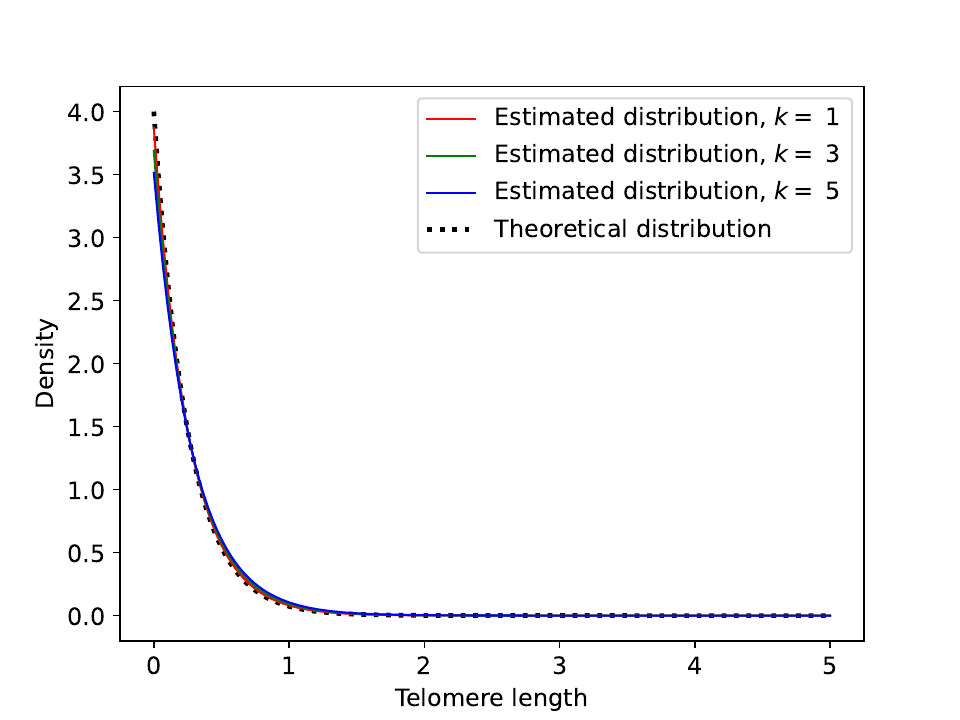}
		\caption{Estimation results when $k\in\{1,3,5\}$.}\label{fig:estimation_exponential_severaltelos_small} 
	\end{subfigure}
	\hfill
	\begin{subfigure}[t]{0.425\textwidth}
		\centering
		\includegraphics[width = \textwidth]{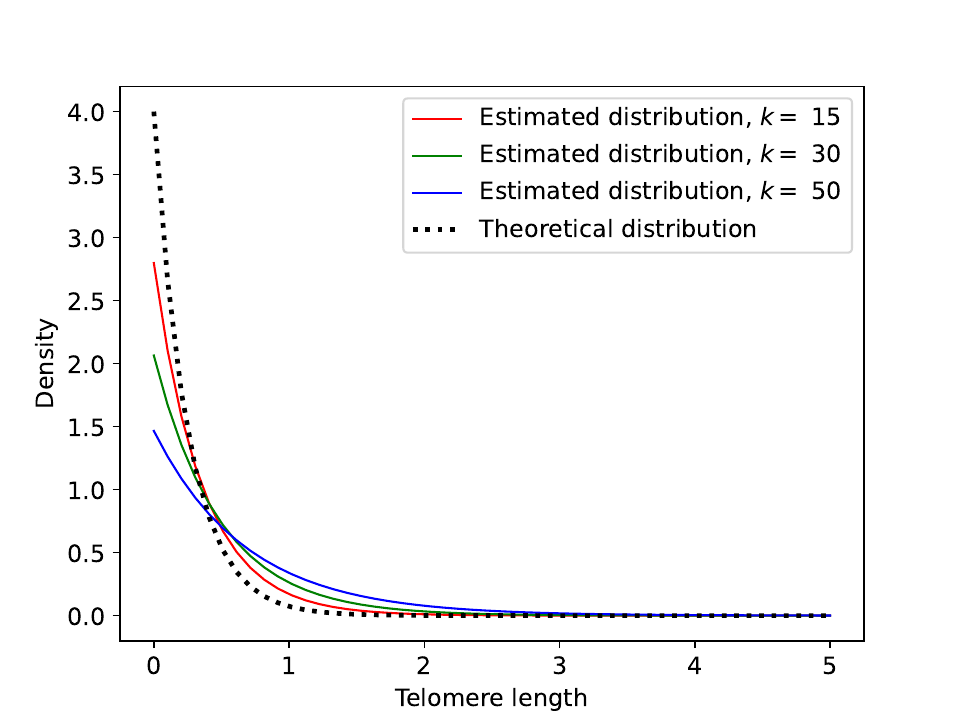}
		\caption{Estimation results when $k\in\{15,30,50\}$.}\label{fig:estimation_exponential_severaltelos_large}
	\end{subfigure}
	\caption{Estimation results in the model with several telomeres when $b = 1$, $g = 1_{[0,1]}$, $N = 40$, $n_0=h_{1,4}$, and $k\in\{1,3,5,15,30,50\}$.}\label{fig:estimation_model_several_telomeres}
\end{figure}
 

A quantitative reason for these observations is that the model approximation used to construct our estimator is no longer valid when $k$ is large. Indeed, the bounds obtained in Proposition~\ref{prop:approximation_PDE_several_telomeres} and Theorem~\ref{te:main_result} increase with the value of $k$. It follows that using $\widehat{n}_0^{(2k)}$ to estimate $n_0$ does not really make sense, because we now have no control over its estimation error. A qualitative reason is that the senescence times distribution becomes more and more determined by the distribution of cell division times when $k \rightarrow +\infty$. The estimation is thus less accurate as the information on $n_0$ is blurred by the influence of the  division times distribution. The increasing influence of the division times distribution is related to the fact that the probability of having at least one short telomere initially grows with $k$. As a result, the initial minimum length tends to concentrate near zero, and fewer shortenings are required before having a telomere with length below $0$. This yields that the randomness of the initial distribution and the shortening values lose their impact on the senescence times distribution, and that only the randomness of the division times remains to influence the latter. 
\subsection{Estimation with random variables}\label{subsect:estimation_results_probabilistics_model}
In practice, we never have the full density $n_{\partial}^{(1)}$ or $n_{\partial}^{(2k)}$. Instead, we observe senescence times $\left(T_i\right)_{1\leq i \leq m}$, where $m \in \mathbb{N}^*$, that are noisy measurements of random variables distributed according to~$n_{\partial}^{(1)}$~or~$n_{\partial}^{(2k)}$. In this subsection, we investigate whether our inference method can be adapted to this setting, by considering only the noise related to sampling, and not to measurement error. We first focus in Section~\ref{subsubsubsect:estimation_onetelo_probabilistic} on the case where we observe senescence times distributed according to~$n_{\partial}^{(1)}$. We then do the same in the case where senescence times are distributed according to~$n_{\partial}^{(2k)}$ in two different parts, as it requires more work. More precisely, we present the estimator we use in Section~\ref{subsubsect:estimator_severaltelos_probabilistic}, and check its quality on simulations in Section~\ref{subsubsect:results_severaltelos_probabilistic}. We show in this last case that issues related to the curse of the dimensionality arise. From now on, we assume that~\hyperlink{assumption:H1}{$\left(H_1\right)$} is verified with $b = 1$,~$g = 1_{[0,1]}$, and~$N = 40$. 

\subsubsection{Estimator and estimation with random variables distributed according to~\texorpdfstring{$n_{\partial}^{(1)}$}{n\_partial\^\{(1)\}}}\label{subsubsubsect:estimation_onetelo_probabilistic}

Let us consider $n_s \in\mathbb{N}^*$, and $\left(T_{1,i}\right)_{1\leq i \leq n_s}$ a sequence of random variables independent and identically distributed according to $n_{\partial}^{(1)}$. Our first objective in this section is to construct an estimator of~$n_0$ based on these random variables. In fact, this can be easily done. What we only have to do is to first construct an estimator of $n_{\partial}^{(1)}$, and then to adapt the expression of~$\widehat{n}_0^{(1)}$ given in~\eqref{eq:definitions_estimators} to this estimator. Let us detail these two steps. 

To obtain an estimator of $n_{\partial}^{(1)}$ with the times $\left(T_{1,i}\right)_{1\leq i \leq n_s}$, we do a log-transform kernel density estimation~\cite{charpentier_2015,nguyen_positive_2019}. Denoting for all $x\geq0$ the Gaussian kernel $\rho(x) = \frac{1}{\sqrt{2\pi}}e^{-\frac{x^2}{2}}$, and \hbox{$\alpha > 0$} a smoothing parameter, the log-transform kernel density estimator is defined for all~\hbox{$t\geq 0$}~as
\begin{equation}\label{eq:log_KDE_estimation}
\overline{n}_{\partial}^{(1,\alpha)}\left(t\right) := \frac{1}{n_s}\sum_{i = 1}^{n_s} \frac{1}{t\alpha}\rho\left(\frac{1}{\alpha}\log\left(\frac{t}{T_{1,i}}\right)\right).
\end{equation}
This is a classical method to estimate a density on $\mathbb{R}_+$. The main idea behind this estimator is to smooth the Dirac measures in the empirical estimator of the density, defined as~$\frac{1}{n_s}\sum_{i = 1}^{n_s} \delta_{T_{1,i}}$. This smoothing is due to  that for all $f\in \mathcal{C}_c^{\infty}\left(\mathbb{R}_+^*\right)$ and $X\in\mathbb{R}_+^*$, by the change of variable $t' = \log\left(\frac{t}{X}\right)$ and the fact that $\frac{1}{\alpha}\rho\left(\frac{.}{\alpha}\right) \underset{\alpha \rightarrow0}{\rightarrow} \delta_0$ (as it is a classical mollifier), we have 
\begin{equation}\label{eq:smoothing_exponential}
	\lim_{\alpha \rightarrow 0} \int_{t\in\mathbb{R}_+^*} \frac{1}{t\alpha}\rho\left(\frac{1}{\alpha}\log\left(\frac{t}{X}\right)\right)f(t)\dd t = \lim_{\alpha \rightarrow 0} \int_{t'\in\mathbb{R}} \frac{1}{\alpha}\rho\left(\frac{t'}{\alpha}\right)f\left(X\exp(t')\right)\dd t' =  \delta_X(f).
\end{equation}
Thus, by the above and the fact that $\rho\in\mathcal{C}^{\infty}\left(\mathbb{R}\right)$, the terms that are summed in~\eqref{eq:log_KDE_estimation} are smoothed approximations of the measures $\left(\delta_{T_{1,i}}\right)_{i\in\llbracket1,n_s\rrbracket}$. In particular, when $\alpha$ is small, $\overline{n}_{\partial}^{(1,\alpha)}$ is close to the empirical estimator $n_{\partial}$ but the smoothing is weak. When $\alpha$ is large, it is the opposite. We do not use a classical kernel density estimation~\cite{weglarczyk_kernel_2018} because this method works for densities with support on $\mathbb{R}$, and is much less satisfactory on~$\mathbb{R}_+$, see~\cite[Section $1$]{nguyen_positive_2019}. 

Adapting now the definition of $\widehat{n}_0$, see~\eqref{eq:definitions_estimators}, by replacing $n_{\partial}^{(1)}$ with $\overline{n}_{\partial}^{(1,\alpha)}$, we obtain the following estimator for $n_0$ that depends on the smoothing parameter $\alpha >0$, for all $x\geq0$,
\begin{equation}\label{eq:estimators_simulations_onetelo}
	\overline{n}_0^{(1,\alpha)}(x) := \frac{1}{\tilde{b}\tilde{m}_1}\overline{n}_{\partial}^{(1,\alpha)}\left(\frac{x}{\tilde{b}\tilde{m}_1}\right) = \frac{1}{n_s}\sum_{i = 1}^{n_s} \frac{1}{x\alpha}\rho\left(\frac{1}{\alpha}\log\left(\frac{x}{\tilde{b}\tilde{m}_1T_{1,i}}\right)\right).
\end{equation}
This estimator is what we use to estimate $n_0$ from the variables $\left(T_{1,i}\right)_{1\leq i \leq n_s}$. A theoretical result concerning its quality, whose proof is sketched in Section~\ref{sect:impact_noise}, is provided below. It states that the error of this estimator tends to $0$ with probability~$1$ when both $n_s \rightarrow +\infty$ and~$N\rightarrow+\infty$.
\begin{prop}[Confidence intervals, single-telomere model]\label{prop:quality_estimator_random_variables_onetelo}
Assume that \hyperlink{assumption:H1}{$(H_1)-(H_2)$} hold. We denote the constant
\begin{equation}\label{eq:constant_randomvar_onetelo}
C_{\widehat{n},1} := \left|\left|\text{Id}^2.\left(\widehat{n}_0^{(1)}\right)' +\text{Id}.\widehat{n}_0^{(1)} \right|\right|_{L^\infty\left(\mathbb{R}_+\right)}.
\end{equation}
We also fix $p\in(0,1]$, and the smoothing parameter $\alpha_p := \frac{1}{\left(C_{\widehat{n},1}\right)^{\frac{1}{2}}}\left(\frac{\log\left(\frac{2}{p}\right)}{2n_s}\right)^{\frac{1}{4}}$. 
Then, there exists a sequence $\left(L_{1,n}\right)_{n\in\mathbb{N}}$ of positive real numbers independent of $p$ such that \hbox{$\underset{n\rightarrow+\infty}{\limsup}\,L_{1,n} < +\infty$}, and such that with probability at least $1-p$ it holds 
$$
\left|\left|\text{Id}\left[\overline{n}_0^{(1,\alpha_p)} - n_0\right]\right|\right|_{L^\infty\left(\mathbb{R}_+\right)} \leq 
2\sqrt{\frac{2}{\pi}}\left(C_{\widehat{n},1}\right)^{\frac{1}{2}}\left(\frac{\log\left(\frac{2}{p}\right)}{2n_s}\right)^{\frac{1}{4}}  + \frac{L_{1,N}}{N}.
$$
\end{prop}
\begin{rem}
We observe that we need to weight the $L^{\infty}$-norm by the identity function to obtain a qualitative bound on the error. This weighting implies that the guarantees we have on the estimation of $n_0(x)$, where $x\geq0$, decrease when $x$ is small. This is due to the fact the logarithm explodes near $0$, which generates instability. 
\end{rem}

We now proceed to estimations on simulations. We first choose $n_s \in\{30, 300, 3000\}$, and simulate a sequence of random variables $\left(T_{1,i}\right)_{1\leq i \leq n_s}$ distributed according to $n_{\partial}^{(1)}$, when $n_0\in\left\{h_{1,4},h_{2,1.5}\right\}$. These random variables are obtained by simulating a probabilistic telomere shortening model, see~\cite{olaye_long-time_2026,benetos_stochastic_2025}. Then, in Figure~\ref{fig:estimation_results_monte_carlo_one_telo}, we compare the curve of $\overline{n}_0^{(1,\alpha)}$ with the curve of~$n_0$, for $\alpha = \alpha_{0.1}$ defined in~Proposition~\ref{prop:quality_estimator_random_variables_onetelo}. We observe that when $n_s = 3000$, the curve of $\overline{n}_0^{(1,\alpha_{0.1})}$ follows perfectly the one of~$n_0$, so the estimation is very satisfactory. We also observe that when $n_s = 300$, the quality of the estimation decreases, but remains correct. When $n_s = 30$, the curve of $\overline{n}_0^{(1,\alpha_{0.1})}$ does not follow the one of $n_0$, so the estimation is not completely satisfactory. However, this is expected because we have a very small number of data points in this case. We therefore have from these observations that the estimation stays reliable when we use~$\overline{n}_0^{(1,\alpha_{0.1})}$ to estimate instead of~$\widehat{n}_{0}^{(1)}$, provided that $n_s$ is not too small.

\begin{figure}[!ht]
	\centering
	\begin{subfigure}[t]{0.45\textwidth}
		\centering
		\includegraphics[scale = 0.375]{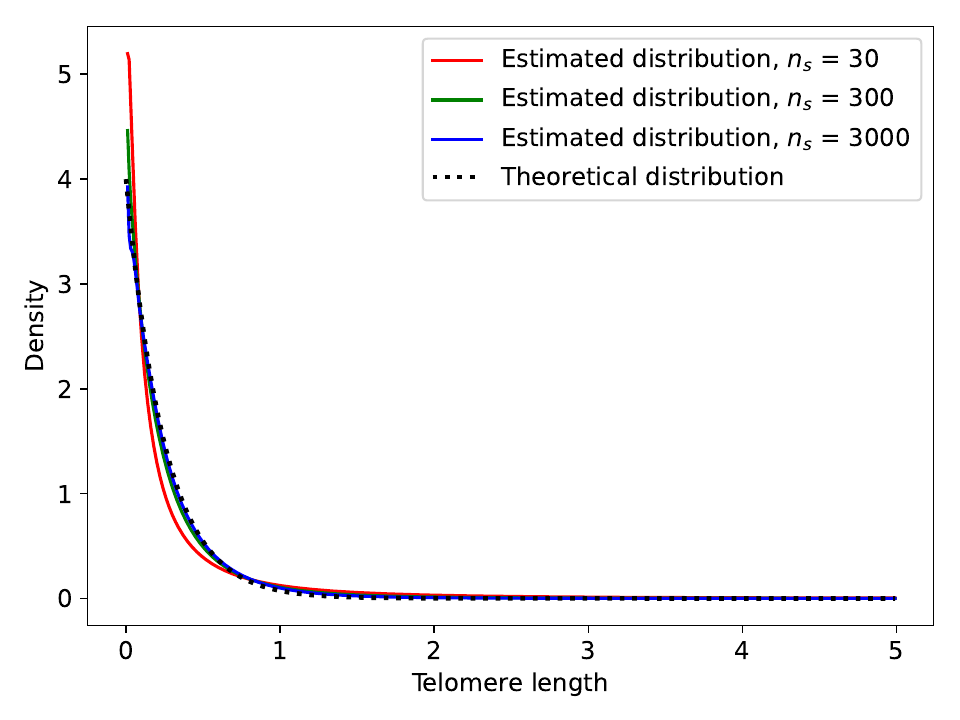}
		\caption{Comparison between the curves of $\widehat{n}_0^{(1,\alpha_{0.1})}$ and \hbox{$n_0 = h_{1,4}$}, for \hbox{$n_s \in\{30, 300, 3000\}$}.}\label{fig:estimation_montecarlo_one_telo_exponential}
	\end{subfigure}
	\hfill
	\begin{subfigure}[t]{0.45\textwidth}
		\centering
		\includegraphics[scale = 0.375]{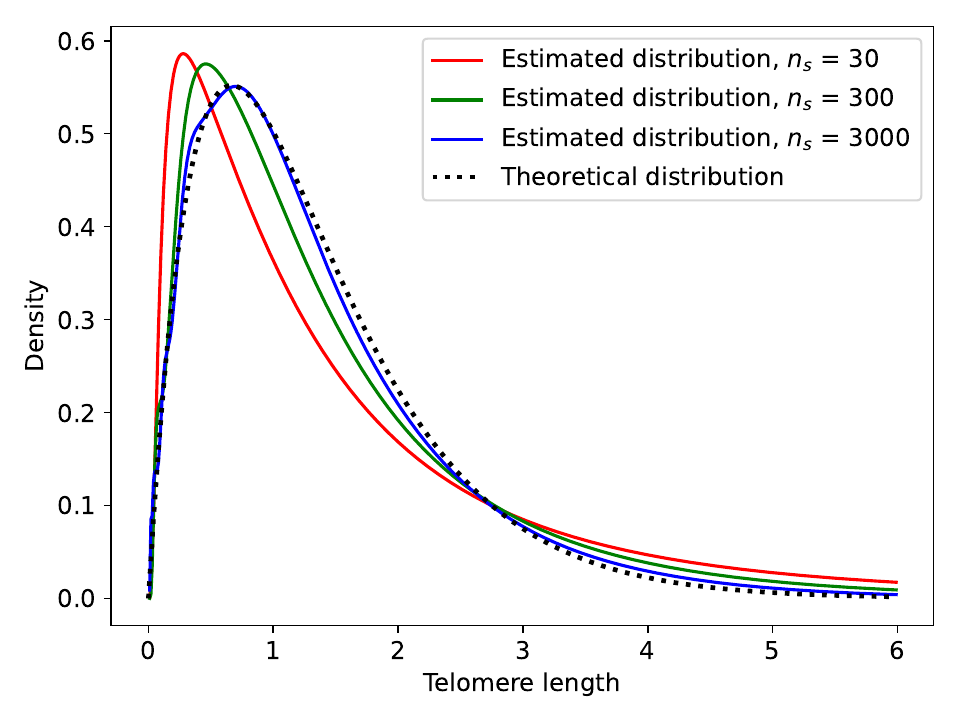}
		\caption{Comparison between the curves of $\widehat{n}_0^{(1,\alpha_{0.1})}$ and \hbox{$n_0 = h_{2,1.5}$}, for \hbox{$n_s \in\{30, 300, 3000\}$}.}\label{fig:estimation_montecarlo_one_telo_gamma}
	\end{subfigure}
	\caption{Estimation results for the estimator $\overline{n}_0^{(1,\alpha_{0.1})}$ defined in~\eqref{eq:estimators_simulations_onetelo}, when $b = 1$, $g = 1_{[0,1]}$, $N = 40$, $n_0 \in \left\{h_{1,4},h_{2,1.5}\right\}$, and for different values of $n_s$.}\label{fig:estimation_results_monte_carlo_one_telo}
\end{figure}
\begin{rem}
Choosing for smoothing parameter $\alpha_{0.1}$ allows us to have an estimation with a confidence level of $0.9$ by Proposition~\ref{prop:quality_estimator_random_variables_onetelo}. However, in practice, choosing for smoothing parameter one of the constants in $\left(\alpha_p\right)_{p\in(0,1]}$ is  impossible in most cases, as it requires to have \textit{a priori} information on the initial distribution (in particular, on $C_{\widehat{n},1}$). Developing a method for choosing a good smoothing parameter without \textit{a~priori} information on $n_0$ is quite challenging. We thus postpone this study to another work.
\end{rem}

\subsubsection{Estimator with random variables distributed according to \texorpdfstring{$n_{\partial}^{(2k)}$}{n\_partial\^\{(2k)\}}}\label{subsubsect:estimator_severaltelos_probabilistic}
We now focus on the case in which we observe a sequence of random variables $\left(T_{2k,i}\right)_{1\leq i \leq n_s}$ identically distributed according to $n_{\partial}^{(2k)}$,  where $n_s\in\mathbb{N}^*$. In particular, in this section, we construct the estimator of $n_0$ we use in this context. We assume that $T_{2k,i} \leq T_{2k,i+1}$ for all \hbox{$i\in\llbracket1,n_s-1\rrbracket$}, and that the random variables $\left(T_{2k,i}\right)_{1\leq i \leq n_s}$ are independent up to a permutation.  We consider the survival function associated to $n_{\partial}$, defined for all $t \geq 0$ as $N_{\partial}(t) := \int_t^{+\infty} n_{\partial}^{(2k)}(s)\dd s$, and its empirical estimator
\begin{equation}\label{eq:survival_function_estimator}
\forall t\geq0:\hspace{2mm} \widehat{N}_{\partial}(t) := \frac{1}{n_s}\sum_{i = 1}^{n_s} 1_{\left\{T_{2k,i} > t\right\}} = 1 - \frac{1}{n_s}\sum_{i = 1}^{n_s} 1_{\left\{T_{2k,i} \leq t\right\}}. 
\end{equation}
The strategy we follow here for constructing our estimator of $n_0$ is slightly different from the one followed in the previous section. It is based on the fact that $\widehat{n}_0^{(2k)}$ is the weak derivative of the function \hbox{$x\mapsto -\left[N_{\partial}\left(\frac{2x}{\tilde{b}\tilde{m}_1}\right)\right]^{\frac{1}{2k}}$}, see~\eqref{eq:definitions_estimators}. From this last property and the fact that $N_{\partial}$ can be estimated by~$\widehat{N}_{\partial}$, a natural estimator for $n_0$ would be to use the weak derivative of $M(x) := -\left[\widehat{N}_{\partial}\left(\frac{2x}{\tilde{b}\tilde{m}_1}\right)\right]^{\frac{1}{2k}}$. The problem is that the latter cannot be represented by a function, as we can see in the following statement. 
\begin{prop}[Derivative of $M$]\label{prop:derivative_power_empirical}
	We denote $T_{2k,0} = 0$ and $T_{2k,n_s+1} = +\infty$. Then, it holds
	$$
	M' = \sum_{j = 0}^{n_s}\left(1 - \frac{j}{n_s}\right)^{\frac{1}{2k}}\left[\delta_{\frac{\tilde{b}\tilde{m}_1}{2}T_{2k,j+1}} -\delta_{\frac{\tilde{b}\tilde{m}_1}{2}T_{2k,j}}\right].
	$$
\end{prop}
\begin{proofNoQED}
	As the times $\left(T_{2k,i}\right)_{0\leq i \leq n_s+1}$ are ordered in increasing order, we have for all~$t\geq0$ 
	\begin{equation}\label{eq:survival_function_estimator_otherform}
		\sum_{i = 1}^{n_s} 1_{\left\{T_{2k,i} \leq t\right\}} = \sum_{i = 1}^{n_s} \sum_{j = i}^{n_s} 1_{\left\{T_{2k,j} \leq t < T_{2k,j+1}\right\}} = \sum_{j = 1}^{n_s} \sum_{i = 1}^{j} 1_{\left\{T_{2k,j} \leq t < T_{2k,j+1}\right\}}  = \sum_{j = 0}^{n_s}  j.1_{\left\{T_{2k,j} \leq t < T_{2k,j+1}\right\}}.
	\end{equation}
	Therefore, by first replacing $M$ with its definition, and then combining~\eqref{eq:survival_function_estimator_otherform} with~\eqref{eq:survival_function_estimator}, in view of the fact that the intervals in the indicators in~\eqref{eq:survival_function_estimator_otherform} are disjoint, we obtain the following equality, which ends the proof, for all $f \in \mathcal{C}_c^{\infty}\left(\mathbb{R}_+^*\right)$,
	$$
	\begin{aligned}
		-\int_{x\in\mathbb{R}_+^*} M(x) f'(x) \dd x &= \sum_{j = 0}^{n_s}\left(1 - \frac{j}{n_s}\right)^{\frac{1}{2k}}\int_{x\in\mathbb{R}_+^*}  1_{\left\{T_{2k,j} \leq \frac{2x}{\tilde{b}\tilde{m}_1} < T_{2k,j+1}\right\}} f'(x)\dd x \\
		&= \sum_{j = 0}^{n_s} \left(1 - \frac{j}{n_s}\right)^{\frac{1}{2k}}\left[f\left(\frac{\tilde{b}\tilde{m}_1}{2}T_{2k,j+1}\right) -f\left(\frac{\tilde{b}\tilde{m}_1}{2}T_{2k,j}\right)\right]. \mathrlap{\phantom{xxxx}\hspace{0.05mm}\qed}
	\end{aligned}
	$$
\end{proofNoQED}
\noindent To solve the issue presented above, we use a smoothed version of $M'$ to estimate $n_0$. This version is obtained by using that in view of~\eqref{eq:smoothing_exponential}, we can approximate $\delta_{\frac{\tilde{b}\tilde{m}_1}{2}T_{2k,j}}$ for all $j\in\llbracket0,n_s+1\rrbracket$ by the measure $\frac{1}{\alpha x}\rho\left(\frac{1}{\alpha}\log\left(\frac{2x}{\tilde{b}\tilde{m}_1T_{2k,j}}\right)\right) \dd x$, where $\alpha  > 0$ is a smoothing parameter. This gives us the following estimator, depending on the parameter $\alpha$, for all $x\geq0$,
\begin{equation}\label{eq:estimators_simulations_several_telos}
	\overline{n}_0^{(2k,\alpha)}(x) := \sum_{j = 0}^{n_s} \left(1 - \frac{j}{n_s}\right)^{\frac{1}{2k}}\frac{1}{\alpha x}\left[\rho\left(\frac{1}{\alpha}\log\left(\frac{2x}{\tilde{b}\tilde{m}_1T_{2k,j+1}}\right)\right) -\rho\left(\frac{1}{\alpha}\log\left(\frac{2x}{\tilde{b}\tilde{m}_1T_{2k,j}}\right)\right)\right].
\end{equation}
The following proposition provides us a bound on the error of this estimator. It is proved in Section~\ref{sect:impact_noise}, and further discussed in Section~\ref{subsubsect:results_severaltelos_probabilistic}.
\begin{prop}[Confidence intervals, model with several telomeres]\label{prop:quality_estimator_random_variables_severaltelos}
	Assume that \hyperlink{assumption:H1}{$(H_1)-(H_4)$} hold. We denote the constant
	\begin{equation}\label{eq:constant_randomvar_severaltelos}
	C_{\widehat{n},2k} := \left|\left|\text{Id}^2.\left(\widehat{n}_0^{(2k)}\right)' +\text{Id}.\widehat{n}_0^{(2k)} \right|\right|_{L^\infty\left(\mathbb{R}_+\right)}.
	\end{equation}	
	\noindent We also fix $p\in(0,1]$, and the smoothing parameter $\tilde{\alpha}_p := \frac{1}{\left(C_{\widehat{n},2k}\right)^{\frac{1}{2}}}\left(\frac{\log\left(\frac{2}{p}\right)}{2n_s}\right)^{\frac{1}{8k}}$. 
	Then, there exists a sequence $\left(L_{2k,n}\right)_{n\in\mathbb{N}}$ of positive real numbers independent of $p$, such that \hbox{$\underset{n\rightarrow+\infty}{\limsup}\, L_{2k,n} < +\infty$}, and such that with probability at least $1-p$ it holds 
	\vspace{-1.5mm}
	\begin{equation}\label{eq:bounds_confidence_interval}
		\left|\left|\text{Id}\left[\overline{n}_0^{(2k,\tilde{\alpha}_p)} - n_0\right]\right|\right|_{L^\infty\left(\mathbb{R}_+\right)} \leq 
		2\sqrt{\frac{2}{\pi}}\left(C_{\widehat{n},2k}\right)^{\frac{1}{2}}\left(\frac{\log\left(\frac{2}{p}\right)}{2n_s}\right)^{\frac{1}{8k}}  + \frac{L_{2k,N}}{N}.
	\end{equation}
\end{prop}
\noindent At first sight, other estimators than the one given in~\eqref{eq:estimators_simulations_several_telos} seem more natural to use. In view of the definition of~$\widehat{n}_0^{(2k)}$, see~\eqref{eq:definitions_estimators}, we can for example mention estimators corresponding to the ratio between an estimator of the numerator of~$\widehat{n}_0^{(2k)}$ and an estimator of the denominator of~$\widehat{n}_0^{(2k)}$. We prefer here to use the estimator presented in~\eqref{eq:estimators_simulations_several_telos} for two reasons. The first one is that it is easier to obtain qualitative results about the quality of the estimators $\left(\overline{n}_0^{(2k,\alpha)}\right)_{\alpha >0}$, see Proposition~\ref{prop:quality_estimator_random_variables_severaltelos}. The second one is that this prevents us from having an estimator with a denominator tending to~$0$ when~$x\rightarrow+\infty$, which can generate instability. 

\subsubsection{Estimation with random variables distributed according to \texorpdfstring{$n_{\partial}^{(2k)}$}{n\_partial\^\{(2k)\}}}\label{subsubsect:results_severaltelos_probabilistic}

We consider two examples: first, the case where the model with several telomeres, defined in \eqref{eq:PDE_model_telomeres_several_telos}, has for parameters $\left(n_0,k\right) = \left(h_{1,4},5\right)$. Second, the case where this model has for parameters $\left(n_0,k\right) = \left(h_{2,1.5},16\right)$. For each of these examples, we fix $n_s = 3000$ and simulate a sequence of random variables $\left(T_{2k,i}\right)_{1\leq i \leq n_s}$ distributed according to~$n_{\partial}^{(2k)}$ by using a probabilistic telomere shortening model \cite{olaye_long-time_2026,benetos_stochastic_2025}. Then, we compute an estimation of $n_0$ from these times by using $\overline{n}_0^{(2k,\alpha)}$ for $\alpha = 0.275$. The smoothing parameter has been chosen manually, and does not need to be optimal. We do not choose one of the smoothing parameters $\left(\tilde{\alpha}_p\right)_{p\in(0,1]}$ introduced in Proposition~\ref{prop:quality_estimator_random_variables_severaltelos} for the reason given in Remark~\ref{rem:smoothing_parameter_several}. We plot in Figure~\ref{fig:estimation_results_lack_data} the estimated initial distributions (green curves), and compare them with their theoretical values (dotted curves). What we observe is surprising. The left tails of the initial distributions are first correctly estimated. However, at a certain telomere length, a peak appears on each estimated curve and the estimated density is~$0$ immediately afterwards. We thus have a very poor estimation.

\begin{figure}[!ht] 
	\centering
	\begin{subfigure}[t]{0.46\textwidth}
		\centering
		\includegraphics[scale = 0.375]{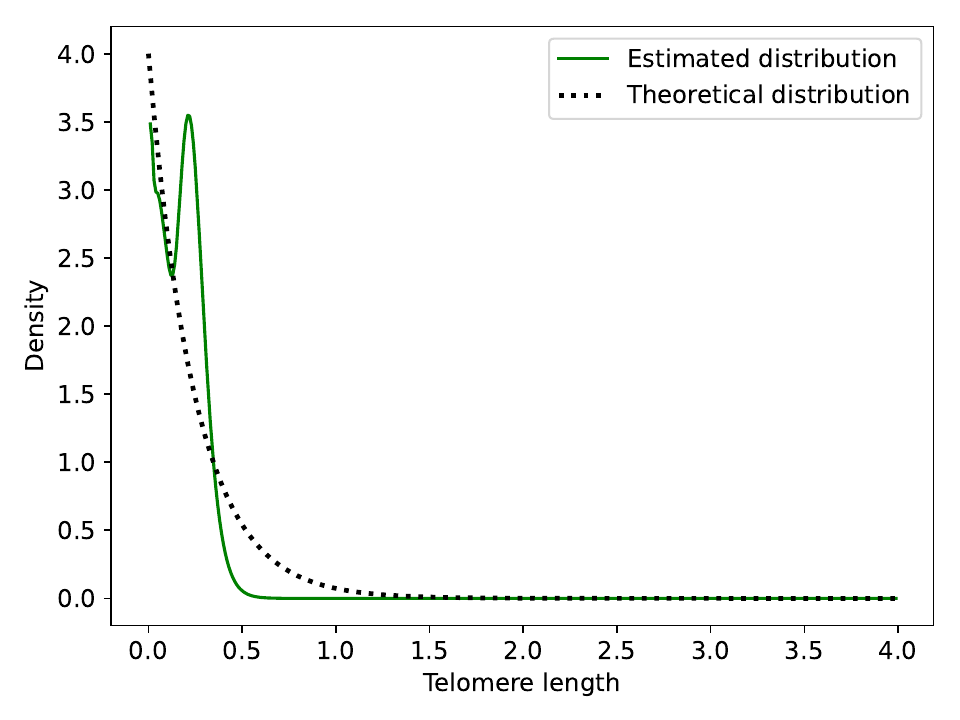}
		\caption{Comparison between the curves of $\widehat{n}_0^{(2k,\alpha)}$ and $n_0 = h_{1,4}$, for $k = 5$.}\label{fig:estimation_results_lack_data_exponential}
	\end{subfigure}
	\hfill
	\begin{subfigure}[t]{0.46\textwidth}
		\centering
		\includegraphics[scale = 0.375]{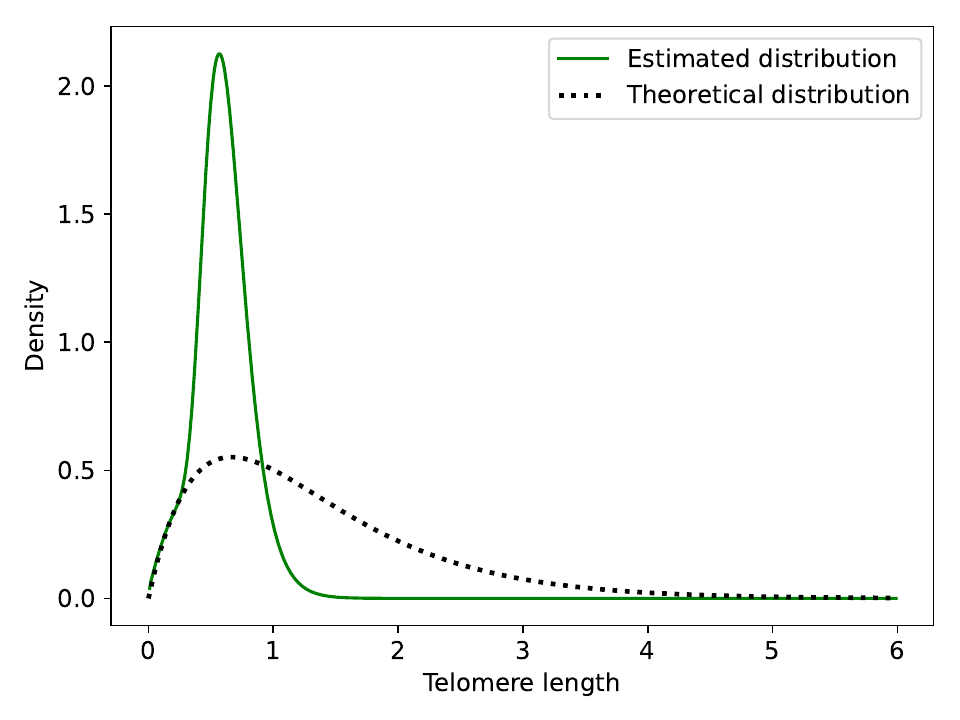}
		\caption{Comparison between the curves of $\widehat{n}_0^{(2k,\alpha)}$ and $n_0 = h_{1,4}$, for $k = 16$.}\label{fig:estimation_results_lack_data_gamma}
	\end{subfigure}
	\caption{Estimation results for the estimator $\overline{n}_0^{(2k,\alpha)}$ defined in~\eqref{eq:estimators_simulations_several_telos}, when $b = 1$, $g = 1_{[0,1]}$, $N = 40$,~\hbox{$n_s = 3000$}, $\alpha = 0.275$ and $\left(n_0,k\right) \in\left\{\left(h_{1,4},5\right),\left(h_{2,1.5},16\right)\right\}$.}\label{fig:estimation_results_lack_data}
\end{figure}

To understand this, we investigate the $n_s$ simulations of the probabilistic model used to generate our sequence of senescence times~$\left(T_{2k,i}\right)_{1\leq i \leq n_s}$. For each of these simulations, we successively collect the initial length of the telomere signalling senescence, plot in Figure~\ref{fig:histogrammes_stoppinglengths} the histogram of these lengths (green bars), and superpose the curve of $n_0$ on this histogram (black curve). In Figure~\ref{fig:histogrammes_stoppinglengths_exponential}, which corresponds to the case where $(n_0,k) = (h_{1,4},5)$, we observe that even if $n_0$ has a large density after the length $0.35$, no telomere with an initial length greater than $0.35$ has signalled senescence. In Figure~\ref{fig:histogrammes_stoppinglengths_gamma}, which corresponds to the case where $(n_0,k) = (h_{1,4},5)$, we have a similar observation. No telomere with an initial length greater than $0.65$ has signalled senescence, whereas the density of $n_0$ is still large at this length. The random variables $\left(T_{2k,i}\right)_{1\leq i \leq n_s}$ therefore do not contain a significant part of the information about $n_0$ in these two cases. This lack of information explains the poor estimation in~Figure~\ref{fig:estimation_results_lack_data}. It also explains why we observe peaks: a smoothed Dirac measure appears in each estimation due to the abrupt loss of information. 

\begin{figure}[!ht]
	\centering
	\begin{subfigure}[t]{0.45\textwidth}
		\centering
		\includegraphics[scale = 0.4]{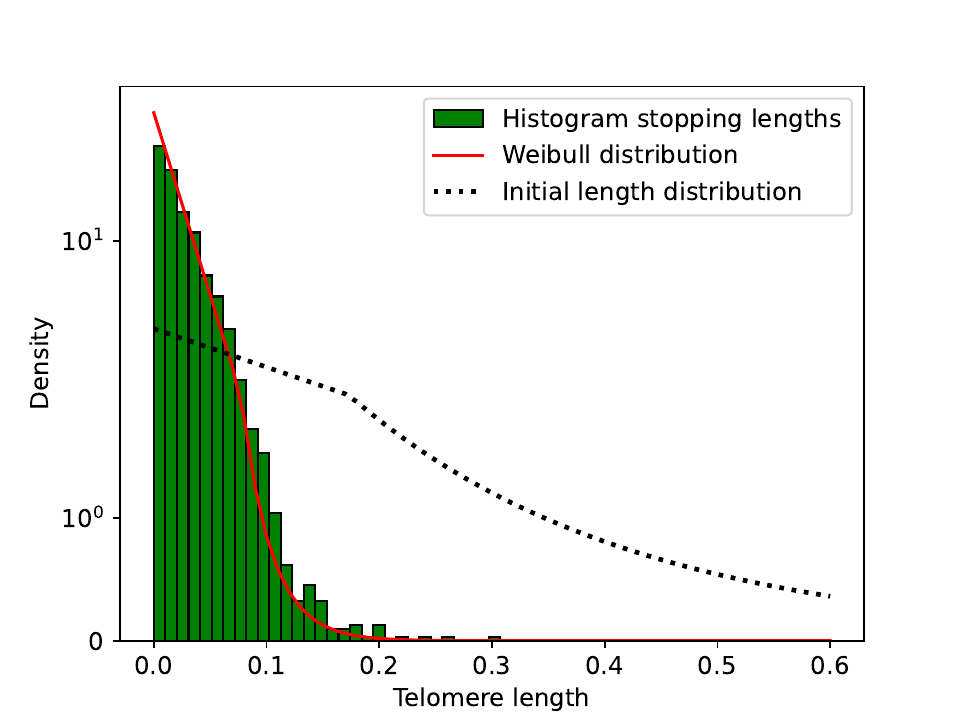}
		\caption{Histogram when $(n_0,k) = (h_{1,4},5)$, and comparison with a Weibull distribution with parameters $\left(1,\left(H_{1,4}\right)^{-1}\left(\frac{1}{10}\right)\right)$.}\label{fig:histogrammes_stoppinglengths_exponential}
	\end{subfigure}
	\hfill
	\begin{subfigure}[t]{0.45\textwidth}
		\centering
		\includegraphics[scale = 0.4]{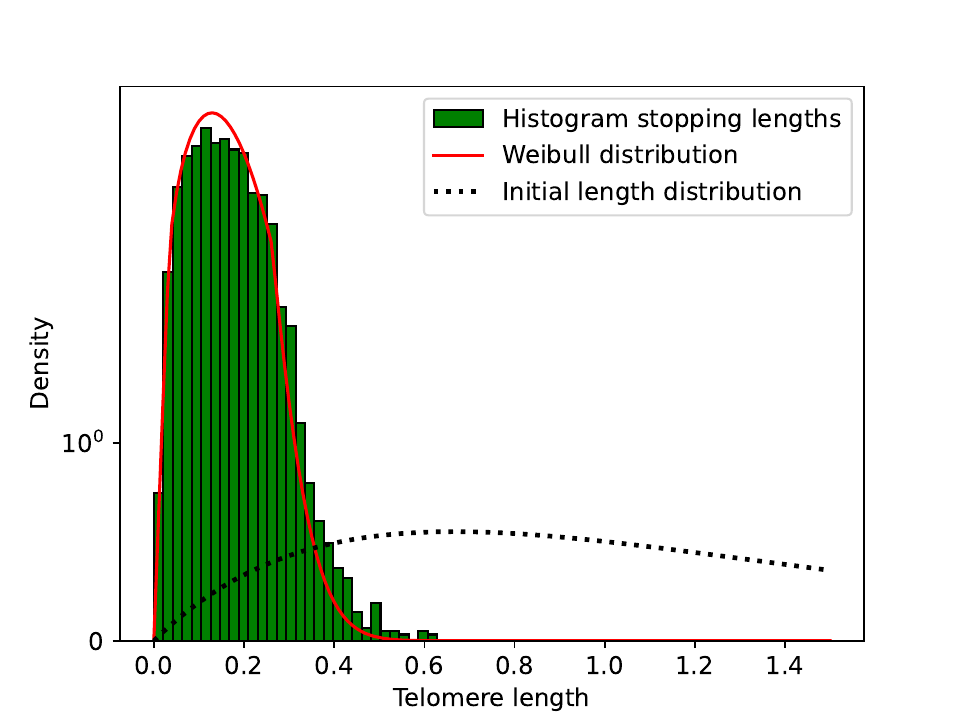}
		\caption{Histogram when $(n_0,k) = (h_{2,1.5},16)$, and comparison with a Weibull distribution with parameters $\left(2,\left(H_{2,1.5}\right)^{-1}\left(\frac{1}{32}\right)\right)$.}\label{fig:histogrammes_stoppinglengths_gamma}
	\end{subfigure}
	\caption{Histograms of the initial lengths of telomeres signalling senescence in the simulations done to generate $\left(T_{2k,i}\right)_{1\leq i \leq n_s}$, where $n_s = 3000$, at the symlog scale~\cite{webber_bi-symmetric_2012}. Comparison with the curve of $n_0$ (black curves) and with the density of a Weibull distribution (red curves). \textit{The symlog scale uses a logarithmic scale for large values and a linear scale for small ones. It allows us to make appear very small bars, while preserving the shape of the histogram close to that of the linear scale.}}\label{fig:histogrammes_stoppinglengths}
\end{figure}
The only way to manage this issue is to do the estimation with a larger number of data~points. The reason is that it allows us to have better coverage of the telomere length that signals senescence at time $t = 0$. The problem is that the number of data points required to obtain a sufficiently large coverage of $n_0$ is far too high. The latter is illustrated by the bound on the errors of the estimators $\left(\overline{n}_0^{(2k,\tilde{\alpha}_p)}\right)_{p\in(0,1]}$ we have obtained in Proposition~\ref{prop:quality_estimator_random_variables_severaltelos}. Indeed, this bound is of order~$\frac{1}{\left(n_s\right)^{8k}}$ as $n_s \rightarrow +\infty$. Hence, the number of data points required to have a qualitative bound increase exponentially fast with the dimension.
\begin{rem}\label{rem:smoothing_parameter_several}
	The fact that we have a bound of order~$\frac{1}{\left(n_s\right)^{8k}}$ as $n_s \rightarrow +\infty$ explains why we do not use one of the smoothing parameter~$\left(\tilde{\alpha}_p\right)_{p\in(0,1]}$ in the estimations presented in Figure~\ref{fig:estimation_results_lack_data}. As the bounds provided in Proposition~\ref{prop:quality_estimator_random_variables_severaltelos} are not qualitative when~$n_s = 3000$ and $k\in\{5,16\}$, there is not benefit in choosing one of these smoothing parameters.
\end{rem}
This problem of lack of data is in fact very common when working in high dimension, and is often referred to as \textit{curse of the dimensionality}. In this context, we can go further, and link this problem to extreme value theory~\cite{leadbetter_extremes_1983}, which studies the distribution of the maximum and minimum of i.i.d. random variables. This link is shown in the following statement. 
\begin{prop}[Minimum initial length when $k\rightarrow+\infty$]\label{prop:extreme_values_n0}
	Let $(\ell,\beta)\in\mathbb{N}^*\times\mathbb{R}_+^*$ and $\left(X_i\right)_{1\leq i \leq 2k}$ a sequence of i.i.d. random variables with distribution given by $h_{\ell,\beta}$. Then, we have for all~$x\geq 0$
	\begin{equation}\label{eq:weibull_convergence}
		\underset{k\rightarrow+\infty}{\lim} \mathbb{P}\left[\frac{\min\left(X_1,\hdots,X_{2k}\right)}{\left(H_{\ell,\beta}\right)^{-1}\left(\frac{1}{2k}\right)} \leq x\right] = 1 - \exp\left(-x^{\ell}\right),
	\end{equation}
	where $\left(H_{\ell,\beta}\right)^{-1}$ is the reciprocal of the function $H_{\ell,\beta}$ defined in~\eqref{eq:density_erlang_distribution}.
\end{prop}
\begin{rem}\label{rem:weibull}
	The right-hand side of~\eqref{eq:weibull_convergence} corresponds to the cumulative distribution function of a Weibull distribution with parameters $\left(\ell,1\right)$. Hence, Eq.~\eqref{eq:weibull_convergence} means that $\min\left(X_1,\hdots,X_{2k}\right)$ can be approximated by a Weibull distribution with parameters $\left(\ell,\left(H_{\ell,\beta}\right)^{-1}\left(\frac{1}{2k}\right)\right)$ when $k\rightarrow+\infty$. 
\end{rem}
\begin{proof}
	First, notice that as $X_1$ has for cumulative distribution function $H_{\ell,\beta}$, by the L'Hospital's rule and~\eqref{eq:density_erlang_distribution}, we have for all $s\geq0$
	\begin{equation}\label{eq:weibull_to_prove}
		\underset{h\rightarrow 0+}{\lim} \frac{\mathbb{P}[-X_1 > -hs]}{\mathbb{P}[-X_1 > -h]} =  \underset{h\rightarrow 0+}{\lim} \frac{H_{\ell,\beta}(hs)}{H_{\ell,\beta}(h)} = \underset{h\rightarrow 0+}{\lim} \frac{s \times h_{\ell,\beta}(hs)}{h_{\ell,\beta}(h)} =s^{\ell}.
	\end{equation}
	Therefore, by using~\eqref{eq:weibull_to_prove} to verify the assumptions of the Fisher-Tippett-Gnedenko's theorem~\cite[Corollary $1.6.3$, Type III]{leadbetter_extremes_1983},  we obtain that for all $x\geq 0$ 
	$$
	\lim_{k\rightarrow + \infty}\mathbb{P}\left[\frac{\max\left(-X_1,\hdots,-X_{2k}\right)}{\left(H_{\ell,\beta}\right)^{-1}\left(\frac{1}{2k}\right)} \leq -x\right] = \exp\left(-x^{\ell}\right).
	$$
	We thus have that the proposition is proved, by using the above and the fact that for all $y\geq0$ it holds $\mathbb{P}\left[\min\left(X_1,\hdots,X_{2k}\right) \leq y\right] = 1 - \mathbb{P}\left[\max\left(-X_1,\hdots,-X_{2k}\right) \leq -y\right]$.
\end{proof}
\noindent From this result, we have that if $n_0$ is the density of an Erlang distribution and $k$ is large, then the distribution of the shortest telomere at $t = 0$ depends only on the left tail of $n_0$. In addition, we recall that in the approximated model, the distribution of senescence times is a scaled version of the distribution of the shortest telomere at $t = 0$ (see Remark~\eqref{rem:initial_telomere_length_correlations}). Then, we only have a good estimation of the left tail in Figure~\ref{fig:estimation_results_lack_data} because the cemetery only provide information about it in the original model, when $k\rightarrow+\infty$. This explanation is all the more valid since each of the histograms presented in Figure~\ref{fig:histogrammes_stoppinglengths} has a shape similar to that of the density of a Weibull distribution with the parameters given in Remark~\ref{rem:weibull}.

From these observations, we conclude that even if our inference method is theoretically adaptable when we observe random variables instead of $n_{\partial}^{(2k)}$, it is in fact not usable in practice. In most cases, due to the lack of data, only information about the left tail can be inferred, and not about the whole curve. This suggests that the senescence times distribution is in fact mainly influenced by the left tail of $n_0$, and much less by the overall distribution. 

\subsection{Estimation on experimental data}\label{subsect:estimation_real_datas}

To conclude this study, we test our inference method on experimental data. The dataset we use here comes from~\cite{Martin2021}. It contains the senescence times of $187$ different lineages, and the cell cycle duration of $1430$ cells, both measured in hours. The histograms corresponding to these data are presented in Figure~\ref{fig:senescence_times_experimental} and~\ref{fig:division_times_experimental}. Our aim is to use the senescence times to obtain an estimation of $n_0$. Then, we will compare this estimation  with a previous estimation of the initial length distribution, obtained in~\cite{Xu2013} on the basis of an elongation-shortening model, see also~\cite{olaye_long-time_2026,benetos_stochastic_2025}. 

The budding yeast is a species with $16$ chromosomes. We can estimate its division rate as $1/1.386 = 0.7216 \text{ hours}^{-1}$ (hereafter denoted as $\text{h}^{-1}$) by computing the inverse of the average division times presented in~Figure~\ref{fig:division_times_experimental}. In addition, the shortening distribution of this species can be modelled by a uniform distribution from~$5$ to $10$ base pairs (bp), see~\cite{soudet_elucidation_2014}, and the threshold for senescence has been recently be estimated as $L_{\min} = 27$ bp in~\cite{rat_mathematical_2025}. We thus proceed to the estimation by using the estimator~$\overline{n}_0^{(2k,\alpha)}$ with the parameters \hbox{$k = 16$}, \hbox{$b = 0.7216 \text{ h}^{-1}$}, \hbox{$m_1 = 7.5$~bp}, $N = 40$ (see Remark~\ref{rem:choice_scaling_parameter} and Section~\ref{subsect:discussion_models_assumptions}), and the smoothing parameter~\hbox{$\alpha = 0.1$}. We apply this estimator to the senescence times presented in the previous paragraph, and plot in Figure~\ref{fig:estimation_experimental} the density we have inferred shifted by $L_{\min}$ (green curve). We also plot the estimation of the initial distribution coming from~\cite{Xu2013} on the same figure (black curve). As for the estimations presented in Section~\ref{subsubsect:results_severaltelos_probabilistic}, we observe in Figure~\ref{fig:estimation_experimental} that we first have an estimation of the left tail, then a peak, and finally an estimate of $0$ everywhere. Let us comment on the estimated left tail, that is the only reliable estimation, see~Section~\ref{subsubsect:results_severaltelos_probabilistic}.

\begin{figure}[!ht]
	\centering
	\begin{subfigure}[t]{0.45\textwidth}
		\centering
		\includegraphics[scale =0.385]{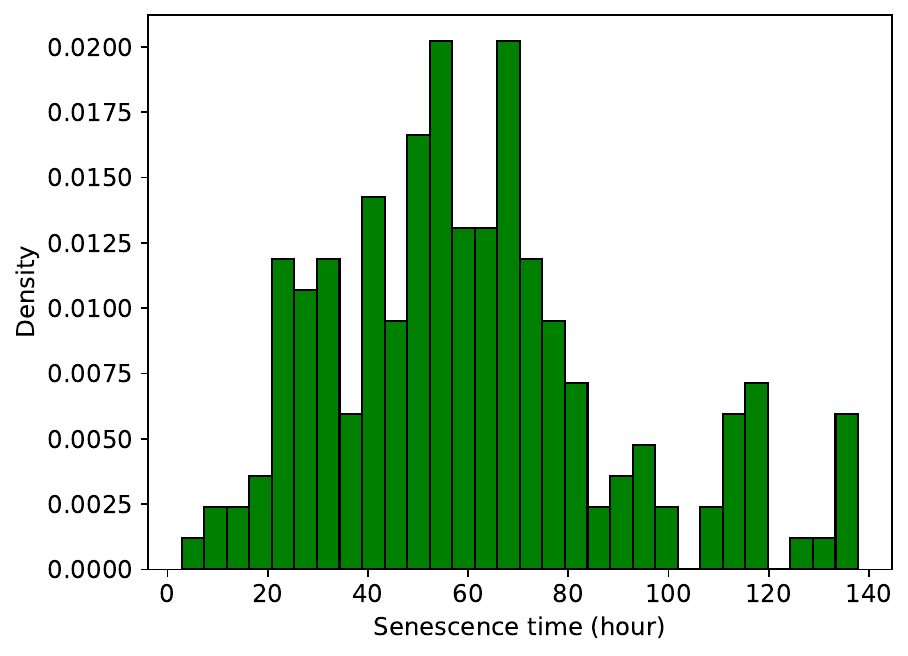}
		\caption{Histogram of senescence times coming from experiments.}\label{fig:senescence_times_experimental}
	\end{subfigure}
	\hfill 
	\begin{subfigure}[t]{0.45\textwidth}
		\centering
		\includegraphics[scale =0.385]{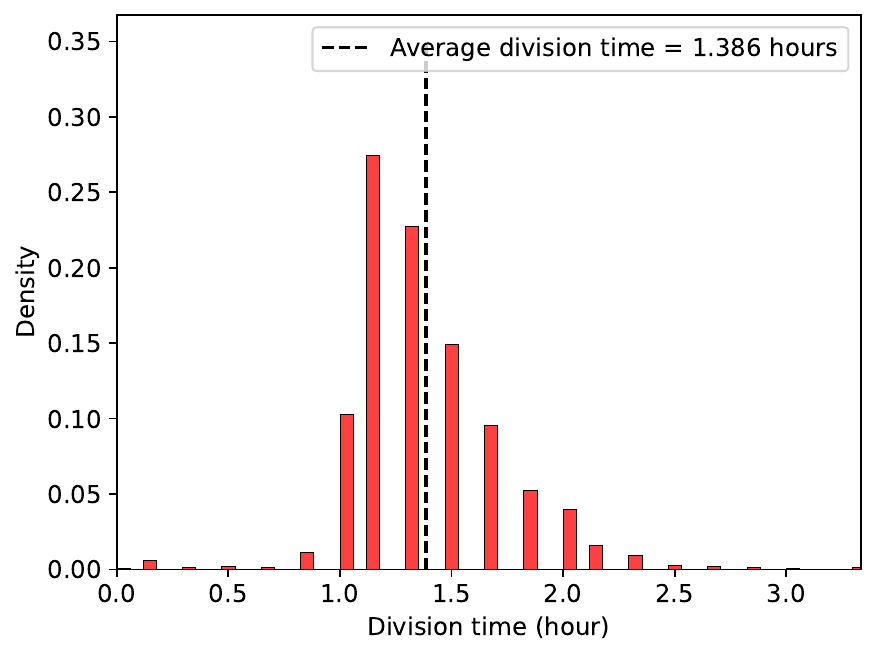}
		\caption{Histogram of division times coming from experiments.}\label{fig:division_times_experimental}
	\end{subfigure}	
	\begin{subfigure}[t]{0.45\textwidth}
		\centering
		\includegraphics[scale =0.385]{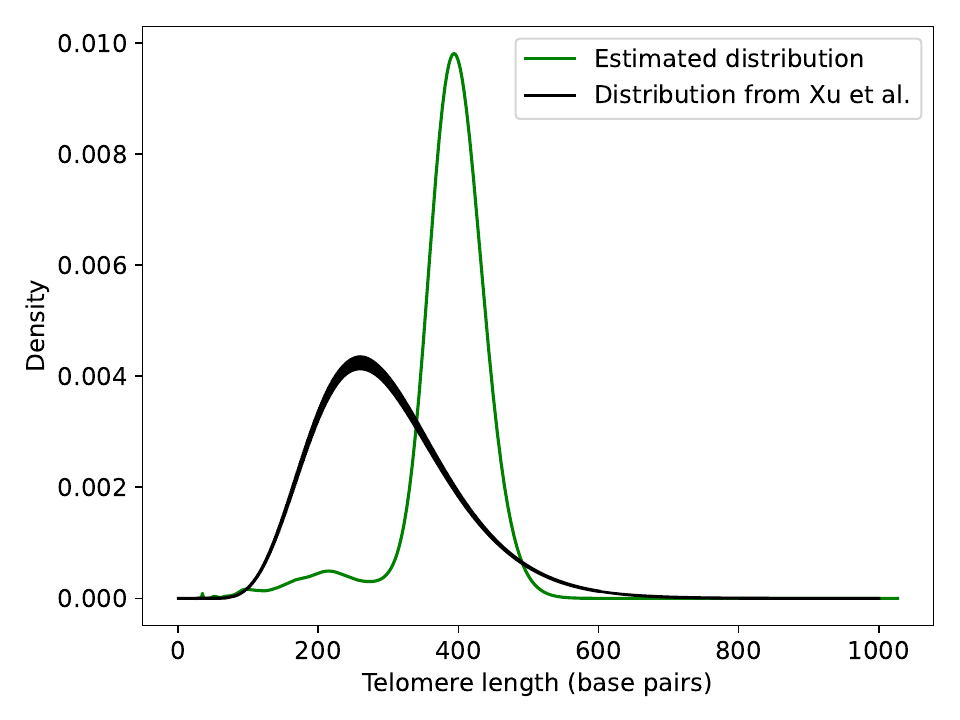}
		\caption{Estimation results on experimental data using $\overline{n}_0^{(32,\alpha)}$ with $\alpha = 0.1$. Comparison with the distribution obtained in \citeauthor{Xu2013}~\cite{Xu2013}.}\label{fig:estimation_experimental}
	\end{subfigure}
	\caption{Inference on experimental data: dataset and estimation results.}\label{fig:experimentalfigures}
\end{figure}

We observe in Figure~\ref{fig:estimation_experimental} that our estimated tail does not grow as quickly as the tail of the distribution obtained in~\cite{Xu2013}. What is the most surprising is that even at lengths close to the mode of the black curve (around $200-240$ base pairs), the tail of the estimated distribution has not grown yet. We thus have an estimation that differs from the one of~\cite{Xu2013}. The reason for this difference is surely linked to the fact that the model described by~\eqref{eq:PDE_model_telomeres_several_telos} is too simple.  For example, using an exponential distribution for division times is not realistic, as illustrated by the shape of the histogram in Figure~\ref{fig:division_times_experimental}. In addition, we do not know yet if the assumption discussed in Remark~\ref{rem:initial_lengths_assumptions} is problematic or not. Solving this inverse problem in more complex models will allow us to identify which model assumptions are too restrictive. This work therefore opens up new modelling perspectives. 

\section{Discussion}\label{sect:discussion_estimation} 
Our work began with a theoretical study, during which we study the connection between the senescence times distribution and the initial length density. We showed that there is a strong link between the two, provided that the ratio between the average telomere length and the average shortening value, denoted $N$, is large. We then have done a numerical study to understand how our method works in practice. Even if this study is encouraging in many cases, it has revealed several limitations. The first limitation is the poor estimation of initial distributions with a small coefficient of variation, see Section~\ref{subsubsect:problems_variability}. In this case, the information on the initial distribution is blurred by the other sources of randomness (shortening values, cell cycle duration). The second limitation is the poor estimation of initial distributions in high dimension. In this case, the number of data points required for a good estimation is considerable, see Section~\ref{subsubsect:results_severaltelos_probabilistic}. In addition, the error between~\eqref{eq:PDE_model_telomeres_several_telos} and~\eqref{eq:approximation_transport_model_2k_telomeres} is large, see~Section~\ref{subsect:estimation_results_severaltelos}.

This work opens up new perspectives for telomere shortening models, including applications beyond those presented in this study. We have provided a rigorous justification that these models can be approximated by transport equations, and have used these approximations to solve our inverse problem. However, this justification has limitations, since for the species that motivate our study, $N$ is not very large (for example, $N = 40$ for the yeast). The first perspective is thus to improve our approximations by adding a diffusion term to our transport equations, see~\cite{doumic_asymptotic_2025,portillo_influence_2023,wattis_mathematical_2020}. Indeed, if we manage to show that telomere shortening models can be approximated by a transport-diffusion equations with an error of order~$\frac{1}{N^2}$ instead of $\frac{1}{N}$, then it will be possible to construct estimators with estimation errors of order $\frac{1}{N^2}$ and improve the results presented in Section~\ref{subsubsect:problems_variability}. In addition, on this more complex approximation, we strongly believe that thanks to the diffusion, the distribution of senescence times distribution is not only influenced by the distribution of the initial shortest telomere, see Remark~\ref{rem:initial_telomere_length_correlations}, but also by the second shortest, third shortest etc. Hence, by obtaining information on the distribution of the shortest and second shortest initial telomeres thanks to this approximation, and then studying if these two distributions are close or not, it might be possible to obtain information on the correlations of the initial telomere lengths when the assumption presented in Remark~\ref{rem:initial_lengths_assumptions} does not hold. Obtaining such a model approximation corresponds to a work in progress.

The second perspective is to find a model approximation when $k \rightarrow +\infty$. As mentioned above, problems arise when $k$ is large, that may be related to extreme value theory or to the fact that the approximation error depends on $k$ (see Theorem~\ref{te:main_result}). It is therefore necessary to understand how the dynamics evolve when $k$ is large. The main approach to achieve this is to use extreme value theory to our advantage, by constructing a typical particle representing the telomere with the smallest length among the $2k$ present in a cell. In view of the Fisher–Tippett–Gnedenko's theorem~\cite[Theorem~$1.6.2$]{leadbetter_extremes_1983}, it will thus be possible to obtain a parametric law for the evolution of the shortest telomere over time, and thus to simplify the study of the phenomenon.

The third perspective, more directly linked to the inverse problem, is to improve our bounds on the approximation and estimation errors for the model with several telomeres. As explained in Section~\ref{subsect:main_results}, the condition to have an exponential decay of the error in Theorem~\ref{te:main_result}-$(b)$ is that $\lambda + 2k\lambda'_N > \omega + \left(2k-1\right)\omega_N'$. However, this condition is not optimal. The reason we think this is that from~\eqref{eq:cemetery_exponential}, a maximum principle in view of~\hyperlink{assumption:H3}{$\left(H_3\right)$}, and~\eqref{eq:laplace_transform_mu}, one has that for all~$t\geq0$
$$
n_{\partial}^{(2k)}(t) \leq \left(D_{\lambda}\right)^{2k}kbm_1\lambda'_N\exp\left[-kbm_1\lambda'_Nx\right].
$$
Using the latter and~Lemma~\ref{lemma:inequalities_lower_bounds_tails} to bound $n_0^{(2k)}$, in view of the right-hand side of~\eqref{eq:definitions_estimators} and~\eqref{eq:equality_transport_terms}, yields that for all $x\geq0$
$$
\widehat{n}_0^{(2k)}(x) \leq \left(D_{\lambda}\right)^{2k}\frac{\lambda'_N\exp\left[-2k\lambda'_Nx\right]}{\left(D_{\omega}\exp\left[-2k\omega'_Nx\right]\right)^{1-\frac{1}{2k}}} =  \frac{\left(D_{\lambda}\right)^{2k}}{\left(D_{\omega}\right)^{2k-1}}\lambda'_N\exp\left[-2k\lambda'_Nx + \left(2k-1\right)\omega'_N x\right],
$$
for which the right-hand side vanishes exponentially when $2k\lambda'_N > \left(2k-1\right)\omega'_N$. In addition, the function $n_0$ vanishes exponentially when $x\rightarrow+\infty$ in view of \hyperlink{assumption:H3}{$\left(H_3\right)$}. We thus have that the condition under which~$\left|n_0^{(2k)}(x) - n_0(x)\right|$ vanishes exponentially is when $2k\lambda'_N > \left(2k-1\right)\omega'_N$,  which is less restrictive than the condition $\lambda + 2k\lambda'_N > \omega + \left(2k-1\right)\omega_N'$ since $\omega \geq \lambda$. We thus would like to know if this is possible to obtain a better coefficient in the exponential in the bound of Theorem~\hbox{\ref{te:main_result}-$(b)$}, with a coefficient before the exponential at least of the same order. This can maybe be done by better controlling the term $\Delta_2(x)$ introduced in Section~\ref{subsect:proof_quality_estimator_several_telomeres}, where $x\in\mathbb{R}_+$. In addition, the constants~$\left(D_{\lambda}\right)^{2k}$ and~$\left(\frac{D_{\lambda}}{D_{\omega}}\right)^{2k}$ in respectively~Proposition~\ref{prop:approximation_PDE_several_telomeres} and~Theorem~\hbox{\ref{te:main_result}-$(b)$} are not optimal. There are related to the loss of information on~$n_{0}$ that occurs when we apply~\hyperlink{assumption:H3}{$(H_3)$} or~\hyperlink{assumption:H4}{$(H_4)$} to bound~$n_0$, and this loss increases exponentially with the dimension. We would like to obtain bounds on the estimation and approximation errors in which these two constants do not appear. This will allow us to prove that the approximation and estimation errors grow at most polynomially when the number of chromosomes increases.


The other perspectives are not related to model approximations, and correspond to questions that remain open. First, we would like to prove that the operator $\Psi_{\partial}$ defined in Remark~\ref{rem:ill_posed_inverse_problem_one_dimension}, and its equivalent on~\eqref{eq:PDE_model_telomeres_several_telos}, are invertible. Second, we would like to adapt our method for age-dependent models, see~\cite{olaye_long-time_2026}. These two perspectives are the subject of future work. 


\appendix

\section{Auxiliary statements}
%
We present here auxiliary statements which are quite standard to obtain, and used frequently during all the paper. First, we provide in Section~\ref{subsect:existence_PDE} results related to the integro-differential equations we use in this work. Then, in Section~\ref{subsect:auxiliary_statements_Ik}, we present statements dealing with the cardinalities of subsets of $\mathcal{I}_k$. Finally, in Section~\ref{subsect:auxiliary_statements_erlang_distributions}, we give results about Erlang distributions that we use in our numerical study, see~Section~\ref{sect:estimation_on_simulations}. 

\subsection{Auxiliary statements related to our integro-differential equations}\label{subsect:existence_PDE}
We begin by presenting the statement that justifies the well-posedness of the integro-differential equations used in this work, as well as their uniqueness. The proof of this statement is based on a fixed-point argument, and is inspired by the ideas presented in~\cite[Section~$3.3$]{perthame07}. We do not detail this proof, as it is relatively classical. We however mention, for the interested readers, that the detailed proof of this result is presented in~\cite[p.~$194$]{olaye_thesis_2025}.


\begin{prop}[Well-posedness of the equations]\label{prop:well_definition_general_model}
	Let $d\in\mathbb{N}^*$, $\tilde{\xi}$ a finite measure on $\mathbb{R}_+^d$, $F\in C\left(\mathbb{R}_+,L^1\left(\mathbb{R}_+^d\right)\right)$, and $w_0\in L^1\left(\mathbb{R}_+^d\right)$. Then, there exists a unique solution in~$C\left(\mathbb{R}_+;\,L^1\left(\mathbb{R}_+^d\right)\right)$ to the following integro-differential equation 
	\begin{equation}\label{eq:general_model}
		\begin{cases}
			\partial_t w(t,x) = \int_{v\in\mathbb{R}_+^d} \left[w(t,x+v) -  w(t,x) \right]\tilde{\xi}(\dd v) + F(t,x), & \forall (t,x)\in\mathbb{R}_+\times\mathbb{R}_+^{d},\\
			w(0,x) = w_0(x), & \forall x\in\mathbb{R}_+^{d}.
		\end{cases}
	\end{equation}
	In addition, if $F$ and $w_0$ are non-negative, then $w$ is non-negative.
\end{prop}
\noindent We now present the second statement of this section, which corresponds to a maximum principle for solutions of~\eqref{eq:general_model} when the source term $F$ is identically equal to $0$. Again, we do not detail its proof, and refer to~\cite[p.~$195$]{olaye_thesis_2025} for the full proof of this result.

\begin{cor}[Maximum principle]\label{cor:maximum_principle}
	We work under the setting of Theorem~\ref{te:main_result}, with $F = 0$ and $w_0$ non-negative. We consider $v_0 \in L^1\left(\mathbb{R}_+^d\right)$ verifying $\left|v_{0}\right| \leq w_{0}$, and $(v,w)\in \left(C\left(\mathbb{R}_+;\,L^1\left(\mathbb{R}_+^d\right)\right)\right)^2$ two solutions of~\eqref{eq:general_model} with respective initial conditions $v_{0}$ and $w_{0}$. Then, for all $(t,x)\in\mathbb{R}_+\times\mathbb{R}_+^d$, we have
	$$
	\left|v(t,x)\right| \leq w(t,x).
	$$
\end{cor}
\noindent The third statement we present is an intermediate statement to compute an explicit solution to~\eqref{eq:general_model} when $F = 0$. It can be proved by simply computing the derivative in time of the function given in~\eqref{eq:explicit_solution_complex_expo}.
\begin{lemm}[Complex explicit solutions to~\eqref{eq:general_model}]\label{lemm:explicit_solution_complex_expo}
Let $d\in\mathbb{N}^*$, $\tilde{\xi}$ a finite measure on $\mathbb{R}_+^d$, and two constants $C>0$ and $p\in\mathbb{C}$. Then, the function \hbox{$\zeta_{d,C,p} : \mathbb{R}_+\times\mathbb{R}_+^d \rightarrow \mathbb{C}$} defined for all $\left(t,x\right)\in\mathbb{R}_+\times\mathbb{R}_+^d$ as 
\begin{equation}\label{eq:explicit_solution_complex_expo}
\zeta_{d,C,p}(t,x) = C^d\exp\left[-\Big(\tilde{\xi}(1) - \mathcal{L}\big(\tilde{\xi}\big)\left(p\right)\!\Big)t - p\sum_{j = 1}^{d} x_j\right]
\end{equation} 
verifies the first line of~\eqref{eq:general_model} with $F = 0$.
\end{lemm}
\noindent Now, the fourth statement we give is a consequence of Lemma~\ref{lemm:explicit_solution_complex_expo}. It provides the explicit solution to~\eqref{eq:general_model} when $w_0$ is a product of an exponential and a cosine or sine function. It~can be proved by simply computing~$\frac{1}{2}\left(\zeta_{d,C,\beta + i\omega} + \zeta_{d,C,\beta - i\omega}\right)$ and $\frac{1}{2i}\left(\zeta_{d,C,\beta + i\omega} - \zeta_{d,C,\beta - i\omega}\right)$ for all~$d\in\mathbb{N}^*$, $C>0$ and $\left(\beta,\omega\right)\in\mathbb{R}_+^*\times\mathbb{R}$ by using~\eqref{eq:explicit_solution_complex_expo}, and then using the uniqueness result stated in Proposition~\ref{prop:well_definition_general_model}.
\begin{prop}[Real explicit solutions to~\eqref{eq:general_model}]\label{prop:explicit_solution_cos_sin_expo}
We work under the setting of Theorem~\ref{te:main_result} with $F = 0$. We consider $C > 0$ and $\left(\beta,\omega\right)\in\mathbb{R}_+^*\times\mathbb{R}$, and denote $\tilde{\xi}(1) := \int_{v\in \mathbb{R}_+^{d}}1\tilde{\xi}\left(\dd v\right)$. The following statements hold.
	\begin{enumerate}
	\item Assume that $w_0(x) = C^d\cos\left(\omega \sum_{j = 1}^d x_j\right)\exp\left(-\beta \sum_{j = 1}^dx_j\right)$ for all $x\in\mathbb{R}_+^d$. Then, for all $(t,x)\in\mathbb{R}_+\times\mathbb{R}_+^{d}$, we have
	\begin{equation}\label{eq:explicit_solution_cos_expo}
	\begin{aligned}
	w(t,x) &= C^d\cos\left[\text{Im}\left(\mathcal{L}\big(\tilde{\xi}\big)\left(\beta+i\omega\right)\right)t -\omega\sum_{j = 1}^d x_j\right]\\
	&\times\exp\left[-\Big(\tilde{\xi}(1) - \text{Re}\left(\mathcal{L}\big(\tilde{\xi}\big)\left(\beta+i\omega\right)\right)\!\Big)t - \beta\sum_{j = 1}^{d} x_j\right].
	\end{aligned}
	\end{equation}
	\item Assume that $w_0(x) = C^d\sin\left(\omega \sum_{j = 1}^d x_j\right)\exp\left(-\beta \sum_{j = 1}^d x_j\right)$ for all $x\in\mathbb{R}_+^d$.  Then, for all $(t,x)\in\mathbb{R}_+\times\mathbb{R}_+^{d}$, we have
	\begin{equation}\label{eq:explicit_solution_sin_expo}
	\begin{aligned}
	w(t,x) &= C^d\sin\left[\text{Im}\left(\mathcal{L}\big(\tilde{\xi}\big)\left(\beta+i\omega\right)\right)t -\omega\sum_{j = 1}^d x_j\right]\\
	&\times\exp\left[-\Big(\tilde{\xi}(1) - \text{Re}\left(\mathcal{L}\big(\tilde{\xi}\big)\left(\beta+i\omega\right)\right)\!\Big)t - \beta\sum_{j = 1}^{d} x_j\right].
	\end{aligned}
	\end{equation}
	\end{enumerate}
\end{prop}
\noindent Finally, the last statement that we present is useful to show that the cemeteries in~\eqref{eq:PDE_model_telomeres_one_telo} and~\eqref{eq:PDE_model_telomeres_several_telos} can be very close, even for two different initial conditions. This statement is essential to justify rigorously that the inverse problem we study in this work is ill-posed. 
\begin{prop}[Close cemeteries for different initial conditions]\label{prop:non_continuity_inverse_cemetery}
We work under the setting of Theorem~\ref{te:main_result} with $\tilde{\xi}\left(\mathbb{R}_+^d\backslash\{0\}\right) >0$, $F = 0$, and $w_0(x) = \exp\left(-\sum_{j=1}^d x_j\right)$ for all $x\in\mathbb{R}_+^d$. We denote~$\tilde{\xi}(1) := \int_{v\in \mathbb{R}_+^{d}}1\tilde{\xi}\left(\dd v\right)$, the following function:
\begin{equation}\label{eq:cemetery_general_model}
\forall t\geq0: \hspace{2mm}w_{\partial}(t) = \int_{y\in\mathbb{R}_+^{d}} w(t,y)\tilde{\xi}\left(\left\{v\in\mathbb{R}_+^{d}\,|\,y-v \notin \mathbb{R}_+^{d}\right\}\right)\dd y,
\end{equation}
and the following sequences of functions, for all $n\in\mathbb{N}$:
\begin{itemize}[label = \fontsize{6}{12}\selectfont\textbullet,leftmargin=0.8cm]
	\item $w_{0}^{(n)}(x) = \frac{\left(1+\sin\left(n\sum_{j=1}^d x_j\right)\right)\exp\left(-\sum_{j=1}^d x_j\right)}{\int_{y\in\mathbb{R}_+^d}\left(1+\sin\left(n\sum_{j=1}^d y_j\right)\right)\exp\left(-\sum_{j=1}^d y_j\right) \dd y}$ for all $x\in\mathbb{R}_+^d$,
	\item $w^{(n)}\in C\left(\mathbb{R}_+,L^1\left(\mathbb{R}_+^d\right)\right)$ corresponding to the solution of~\eqref{eq:general_model} with initial condition $w_{0}^{(n)}$, 
	\item $w_{\partial}^{(n)}(t) = \int_{y\in\mathbb{R}_+^{d}} w^{(n)}(t,y)\tilde{\xi}\left(\left\{v\in\mathbb{R}_+^{d}\,|\,y-v \notin \mathbb{R}_+^{d}\right\}\right)\dd y$ for all $t\geq0$.
\end{itemize}
Then, it holds:
\begin{equation}\label{eq:non_continuity_closeness_cemetery}
\lim_{n\rightarrow+\infty}\left[\left|\left|w_{\partial}^{(n)} - w_{\partial}\right|\right|_{L^1\left(\mathbb{R}_+\right)} + \left|\left|w_{\partial}^{(n)} - w_{\partial}\right|\right|_{L^\infty\left(\mathbb{R}_+\right)}\right] = 0.
\end{equation}
and
\begin{equation}\label{eq:non-continuity_different_initial_condition}
\lim_{n\rightarrow+\infty} \left|\left|w_{0}^{(n)} - w_0\right|\right|_{L^1\left(\mathbb{R}_+^d\right)} = \frac{2}{\pi}.
\end{equation}
\end{prop}
\begin{proof}
We begin by proving~\eqref{eq:non_continuity_closeness_cemetery}. For all $a>0$, $f : \mathbb{R}_+ \rightarrow \mathbb{R}$ verifying $\sup_{t\geq0}\left(e^{at}\left|f(t)\right|\right) < +\infty$, one has~that
$$
\begin{aligned}
\int_0^{\infty}\left|f(t)\right| \dd t  + \sup_{t\geq0}\left(\left|f(t)\right|\right)  &\leq \left[\int_0^{\infty}e^{-at} \dd t\right]\left[\sup_{t\geq0} \left(e^{at}\left|f(t)\right|\right)\right] + \sup_{t\geq0}\left(e^{at}\left|f(t)\right|\right) \\
&= \left(\frac{1}{a}+1\right)\sup_{t\geq0} \left(e^{at}\left|f(t)\right|\right).
\end{aligned}
$$
Therefore, in view of the above and the fact that $\mathcal{L}\left(\tilde{\xi}\right)(1) < \mathcal{L}\left(\tilde{\xi}\right)(0) = \tilde{\xi}(1)$ (because we have~$\tilde{\xi}\left(\mathbb{R}_+^d\backslash\{0\}\right) > 0$, see~\eqref{eq:definition_laplace_transform}), we are going to prove the following in order to get~\eqref{eq:non_continuity_closeness_cemetery}:
\begin{equation}\label{eq:alternative_non_continuity_closeness_cemetery}
\lim_{n\rightarrow+\infty}\left[\sup_{t\geq0}\left[\exp\left[\Big(\tilde{\xi}(1) - \mathcal{L}\left(\tilde{\xi}\right)(1)\!\Big)t\right]\left|w_{\partial}^{(n)}(t) - w_{\partial}(t)\right|\right]\right] = 0. 
\end{equation}
To do this, we first fix $n\in\mathbb{N}$, and denote the following to simplify notations:
\begin{equation}\label{eq:notations_proof_non_continuity}
	\begin{aligned}
		C_{n} &:= \int_{y\in\mathbb{R}_+^d}\left(1+\sin\left(n\sum_{j=1}^d y_j\right)\right)\exp\left(-\sum_{j=1}^d y_j\right) \dd y, \\
		\forall y\in\mathbb{R}_+^d:\hspace{2mm}A_y &:=  \left\{v\in\mathbb{R}_+^{d}\,|\,y-v \notin \mathbb{R}_+^{d}\right\}.
	\end{aligned}
\end{equation}
By using the linearity of~\eqref{eq:general_model}, and then~\eqref{eq:explicit_solution_sin_expo}, one has that for all $(t,y)\in\mathbb{R}_+\times\mathbb{R}_+^{d}$
\begin{equation}\label{eq:definition_wn}
\begin{aligned}
w^{(n)}(t,y) &= \frac{1}{C_{n}}\left(w(t,y) + \sin\left[\text{Im}\left(\mathcal{L}\big(\tilde{\xi}\big)\left(1+in\right)\right)t - n\sum_{j = 1}^d y_j\right]\right.\\
&\left.\times\exp\left[-\Big(\tilde{\xi}(1) - \text{Re}\left(\mathcal{L}\big(\tilde{\xi}\big)\left(1+in\right)\right)\!\Big)t - \sum_{j = 1}^{d} y_j\right]\right).
\end{aligned}
\end{equation}
In addition, from~\eqref{eq:definition_laplace_transform} and the inequality $\cos(x) \leq 1$ for all $x\in\mathbb{R}$, one has that
$$
\text{Re}\left(\mathcal{L}\left(\tilde{\xi}\right)(1+in)\right) = \int_{y\in\mathbb{R}_+^d} \cos\left[-n\sum_{j = 1}^d y_j\right]\exp\left[-\sum_{j = 1}^d y_j\right]\tilde{\xi}(\dd y) \leq  \mathcal{L}\left(\tilde{\xi}\right)(1),
$$
which implies that
\begin{equation}\label{eq:finiteness_integral_non_continuity}
 \exp\left[\Big(\tilde{\xi}(1) - \mathcal{L}\left(\tilde{\xi}\right)(1)\!\Big)t\right] \leq 	\exp\left[\Big(\tilde{\xi}(1) - \text{Re}\left(\mathcal{L}\big(\tilde{\xi}\big)\left(1+in\right)\right)\!\Big)t\right].
\end{equation}
Therefore, by first subtracting both sides of~\eqref{eq:definition_wn} by $w(t,y)$, then integrating with respect to the measure $\exp\left[\Big(\tilde{\xi}(1) - \mathcal{L}\left(\tilde{\xi}\right)(1)\!\Big)t\right]\tilde{\xi}\left(A_y\right) \dd y$ in view of~\eqref{eq:cemetery_general_model}, and finally applying the triangle inequality and~\eqref{eq:finiteness_integral_non_continuity} to simplify the second exponential in time, we have for all $t\geq0$
\begin{equation}\label{eq:difference_wn_w}
	\begin{aligned}
		\exp\left[\Big(\tilde{\xi}(1) - \mathcal{L}\left(\tilde{\xi}\right)(1)\!\Big)t\right]\left|w_{\partial}^{(n)}(t) -w_{\partial}(t)\right| &\leq \left|\frac{1}{C_{n}} - 1\right|\exp\left[\Big(\tilde{\xi}(1) - \mathcal{L}\left(\tilde{\xi}\right)(1)\!\Big)t\right]w_{\partial}(t) \\
		&+ \frac{1}{C_n}\left|\int_{y\in\mathbb{R}_+^d} \sin\left[\text{Im}\left(\mathcal{L}\big(\tilde{\xi}\big)\left(1+in\right)\right)t - n\sum_{j = 1}^d y_j\right]\right.\\
		&\left.\exp\left[-\sum_{j = 1}^{d} y_j\right]\tilde{\xi}\left(A_y\right)\dd y\right|.
	\end{aligned}
\end{equation}
Now, thanks to~\cite[Theorem~$2.6$]{cioranescu_introduction_1999} and the following equalities (the third equality comes from the fact that $x\in\mathbb{R} \mapsto \cos(x)$ is $2\pi$-periodic), for all $(u_1,\hdots,u_{d-1})\in\mathbb{R}_+^{d-1}$,
$$
\begin{aligned}
\int_{u_d\in(0,2\pi)} \cos\left(\sum_{j = 1}^{d}u_j\right) \dd u_d &= \int_{u_d\in(0,2\pi)} \cos\left(\sum_{j = 1}^{d-1}u_j + u_d\right) \dd u_d \\
 &= \int_{u'_d\in\left(\sum_{j = 1}^{d-1}u_j,\sum_{j = 1}^{d-1}u_j+2\pi\right)} \cos\left(u'_d\right) \dd u'_d = \int_{u_d'\in(0,2\pi)} \cos(u'_d) \dd u'_d =0,
\end{aligned}
$$
one can obtain that 
\begin{equation}\label{eq:auxiliary_limit_proof_non_continuity_zero}
\begin{aligned}
\lim_{n\rightarrow+\infty} \int_{y\in\mathbb{R}_+^d}\cos\left[n\sum_{j = 1}^d y_j\right]\exp\left[-\sum_{j = 1}^{d} y_j\right]\tilde{\xi}\left(A_y\right) \dd y &= \left[\frac{1}{(2\pi)^{d}}\int_{u \in (0,2\pi)^{d}} \cos\left[\sum_{j = 1}^{d}u_j\right]\dd u\right]\\
&\times\left[\int_{y\in\mathbb{R}_+^{d}} \exp\left[-\sum_{j = 1}^{d} y_j\right]\tilde{\xi}\left(A_y\right)\dd y\right]  = 0.
\end{aligned}
\end{equation}
By proceeding as above, we also have that
\begin{equation}\label{eq:auxiliary_limit_proof_non_continuity_zero_bis}
\begin{aligned}
\lim_{n\rightarrow+\infty}\int_{y\in\mathbb{R}_+^d}\sin\left[n\sum_{j = 1}^d y_j\right]\exp\left[-\sum_{j = 1}^{d} y_j\right]\tilde{\xi}\left(A_y\right) \dd y = 0.
\end{aligned}
\end{equation}
Therefore, by first applying the equality $\sin(c-d) = \sin(c)\cos(d) + \cos(c)\sin(d)$ for the real numbers~$c = \text{Im}\left(\mathcal{L}\big(\tilde{\xi}\big)\left(1+in\right)\right)t$ and $d = n\sum_{j = 1}^d y_j$, then using the triangle inequality and the facts that $\sup_{s\geq0}\left(|\sin(s)|\right) = 1$ and $\sup_{s\geq0}\left(|\cos(s)|\right) = 1$, and finally applying~\eqref{eq:auxiliary_limit_proof_non_continuity_zero} and~\eqref{eq:auxiliary_limit_proof_non_continuity_zero_bis} (the limit is equal to the limit superior in these equations),
we obtain that
\begin{equation}\label{eq:auxiliary_limit_proof_non_continuity_first}
\begin{aligned}
&\limsup_{n\rightarrow+\infty} \left[ \sup_{t\geq0}\left[\left|\int_{y\in\mathbb{R}_+^d}\sin\left[\text{Im}\left(\mathcal{L}\big(\tilde{\xi}\big)\left(1+in\right)\right)t - n\sum_{j = 1}^d y_j\right]\exp\left[-\sum_{j = 1}^{d} y_j\right] \tilde{\xi}\left(A_y\right) \dd y\right|\right]\right] \\
&\leq \limsup_{n\rightarrow+\infty} \left[\sup_{t\geq0}\left[\left|\sin[\text{Im}\left(\mathcal{L}\big(\tilde{\xi}\big)\left(1+in\right)\right)t]\right|\right]\left|\int_{y\in\mathbb{R}_+^d}\cos\left[n\sum_{j = 1}^d y_j\right]\exp\left[-\sum_{j = 1}^{d} y_j\right]\tilde{\xi}\left(A_y\right) \dd y\right|\right] \\
&+ \limsup_{n\rightarrow+\infty} \left[ \sup_{t\geq0}\left[\left|\cos[\text{Im}\left(\mathcal{L}\big(\tilde{\xi}\big)\left(1+in\right)\right)t]\right|\right] \left|\int_{y\in\mathbb{R}_+^d}\sin\left[n\sum_{j = 1}^d y_j\right]\exp\left[-\sum_{j = 1}^{d} y_j\right]\tilde{\xi}\left(A_y\right) \dd y\right|\right] = 0.
\end{aligned}
\end{equation}
Moreover, in view of the first line of~\eqref{eq:notations_proof_non_continuity} and~\cite[Theorem~$2.6$]{cioranescu_introduction_1999} (we proceed as in~\eqref{eq:auxiliary_limit_proof_non_continuity_zero}), and then of~\eqref{eq:cemetery_general_model} and~\eqref{eq:explicit_solution_cos_expo} for $\beta = 1$ and $\omega = 0$ (recall that $w_0(x) = \exp\left(-\sum_{j = 1}^d x_j\right)$ for all~$x\in\mathbb{R}_+^d$), one has that
\begin{equation}\label{eq:auxiliary_limit_proof_non_continuity_second}
\begin{aligned}
\lim_{n\rightarrow+\infty} C_n = 1+ &\lim_{n\rightarrow+\infty}\int_{y\in\mathbb{R}_+^d}\sin\left(n\sum_{j=1}^d y_j\right)\exp\left(-\sum_{j=1}^d y_j\right) \dd y  = 1, \\
&\sup_{t\geq 0}\left[\exp\left[\Big(\tilde{\xi}(1) - \mathcal{L}\left(\tilde{\xi}\right)(1)\!\Big)t\right] w_{\partial}(t)\right]  < +\infty.
\end{aligned}
\end{equation}
Therefore, by taking the supremum in $t\geq0$ and the limit when $n\rightarrow+\infty$ in~\eqref{eq:difference_wn_w} in view of~\eqref{eq:auxiliary_limit_proof_non_continuity_first} and~\eqref{eq:auxiliary_limit_proof_non_continuity_second}, we obtain that~\eqref{eq:alternative_non_continuity_closeness_cemetery} is true, which yields~\eqref{eq:non_continuity_closeness_cemetery}.

Now, it remains to prove~\eqref{eq:non-continuity_different_initial_condition}. Let $n\in\mathbb{N}$. In view of the definitions of $w_0^{(n)}$, $w_0$, and $C_n$, one has that 
$$
\left|\left|w_0^{(n)} - w_0\right|\right|_{L^1\left(\mathbb{R}_+^d\right)} =  \int_{y\in\mathbb{R}_+^d} \left|\left(\frac{1}{C_n}-1\right) + \frac{\sin\left(n\sum_{j = 1}^d y_j\right)}{C_n}\right| \exp\left(-\sum_{j = 1}^d y_j\right) \dd y.
$$
This yields, in view of the reverse triangle inequality, and the following classical results:
\begin{equation}\label{eq:auxiliary_results_non_continuity_differents}
\forall y\in\mathbb{R}_+^d: \hspace{2mm} \left|1 + \sin\left(n\sum_{j = 1}^d y_j\right)\right| \leq 2 \hspace{4mm} \text{ and } \hspace{4mm} \int_{y\in\mathbb{R}_+^d} \exp\left(-\sum_{j = 1}^d y_j\right) \dd y = 1, 
\end{equation}
that
$$
\begin{aligned}
&\left|\left|\left|w_0^{(n)} - w_0\right|\right|_{L^1\left(\mathbb{R}_+^d\right)} - \int_{y\in\mathbb{R}_+^d} \left|\sin\left(n\sum_{j = 1}^d y_j\right)\right|\exp\left(-\sum_{j = 1}^d y_j\right) \dd y \right| \\
&\hspace{25mm}\leq \int_{y\in\mathbb{R}_+^d} \left|\left(\frac{1}{C_n}-1\right)\left(1 + \sin\left(n\sum_{j = 1}^d y_j\right)\right)\right| \exp\left(-\sum_{j = 1}^d y_j\right) \dd y \leq 2 \left|\frac{1}{C_n}-1\right| .
\end{aligned}
$$
Combining the latter with the first line of~\eqref{eq:auxiliary_limit_proof_non_continuity_second}, and then using~\cite[Theorem~$2.6$]{cioranescu_introduction_1999} in view of the fact that $x\in\mathbb{R}^d\mapsto |\sin\left(\sum_{j = 1}^d x_j\right)|$ is $(0,\pi)^d$-periodic (see~\cite[Definition~$2.1$]{cioranescu_introduction_1999}), yields that
\begin{equation}\label{eq:auxiliary_limit_proof_non_continuity_third}
\begin{aligned}
\lim_{n\rightarrow+\infty} \left|\left|w_0^{(n)} - w_0\right|\right|_{L^1\left(\mathbb{R}_+^d\right)} &= \lim_{n\rightarrow+\infty}  \int_{y\in\mathbb{R}_+^d} \left|\sin\left(n\sum_{j = 1}^d y_j\right)\right| \exp\left(-\sum_{j = 1}^d y_j\right) \dd y  \\
&= \left[\frac{1}{\pi^d}\int_{u\in(0,\pi)^d} \left|\sin\left(\sum_{j = 1}^d u_j\right)\right| \dd u\right]\left[\int_{y\in\mathbb{R}_+^d} \exp\left(-\sum_{j = 1}^d y_j\right)\dd y\right].
\end{aligned}
\end{equation}
In addition, in view of the fact that $\int_{(0,\pi)} |\sin(x+\sigma)| \dd x = \int_{0}^{\pi} |\sin(x)| \dd x$ for all $\sigma \in \mathbb{R}$ because the function $x\in\mathbb{R} \mapsto |\sin(x)|$ is $\pi$-periodic, and then the inequality $\sin(x) \geq 0$ for $x\in(0,\pi)$, we have that
$$
\begin{aligned}
\frac{1}{\pi^d}\int_{u\in(0,\pi)^d} \left|\sin\left(\sum_{j = 1}^d u_j\right)\right| \dd &u = \frac{1}{\pi^d}\int_{\left(u_1,\hdots,u_{d-1}\right)\in(0,\pi)^{d-1}} \left[\int_{u_d\in(0,\pi)}\left|\sin\left(u_d\right)\right| \dd u_d\right] \dd u_1 \hdots \dd u_{d-1} \\ 
&= \frac{1}{\pi^d}\int_{\left(u_1,\hdots,u_{d-1}\right)\in(0,\pi)^{d-1}} \left[\int_{u_d\in(0,\pi)}\sin\left(u_d\right) \dd u_d\right] \dd u_1 \hdots \dd u_{d-1} = \frac{2}{\pi}.
\end{aligned}
$$
Then, by plugging the above and the right-hand side of~\eqref{eq:auxiliary_results_non_continuity_differents} in~\eqref{eq:auxiliary_limit_proof_non_continuity_third}, we obtain~\eqref{eq:non-continuity_different_initial_condition}.

\end{proof}
\subsection{Auxiliary statements related to \texorpdfstring{$\mathcal{I}_k$}{Ik}}\label{subsect:auxiliary_statements_Ik}

We now provide statements related to the cardinality of subsets of $\mathcal{I}_k$. The first statement we present is the following, and corresponds to~\cite[Lemma $4.9$]{olaye_long-time_2026}.  It provides information about the number of sets in $\mathcal{I}_k$, and about the number of sets that contain/not contain a chosen~index. 
\begin{lemm}[Cardinality of subsets of $\mathcal{I}_k$, less than $1$ index fixed]\label{lemm:zero_singleton_cardinal}
	It holds
	$$
	\#\left(\mathcal{I}_k\right) =  2^k,  \hspace{2.5mm}\text{ and }\hspace{2.5mm} \forall i\in\llbracket1,2k\rrbracket: \hspace{1mm}\frac{\#\left(\left\{I\in \mathcal{I}_k\,|\,i\in I\right\}\right)}{2^k} = \frac{\#\left(\left\{I\in \mathcal{I}_k\,|\,i\notin I\right\}\right)}{2^k} = \frac{1}{2}.
	$$ 
\end{lemm}
\noindent The second lemma we present is not given in~\cite{olaye_long-time_2026}. It provides the number of sets in~$\mathcal{I}_k$ that contain a chosen pair. We briefly sketch its proof but do not detail it, as it mainly consists in adapting the proof of~\cite[Lemma $4.9$]{olaye_long-time_2026}.
\begin{lemm}[Cardinality of subsets of $\mathcal{I}_k$, $2$ indexes fixed]\label{lemm:pair_cardinal} 
	For all $(\ell,\ell')\in\llbracket1,2k\rrbracket^2$ such that $\ell \neq \ell' \text{ mod }k$, we have
	\begin{equation}\label{eq:cardinal_pair_sets}
		\#\left(\left\{I\in\mathcal{I}_k\,|\,\ell\in I,\,\ell'\in I\right\}\right) = \begin{cases}
			0, &\text{if }k = 1,\\
			2^{k-2}, & \text{if }k \geq 2. 
		\end{cases}
	\end{equation}
\end{lemm}
\begin{proof}
When $k = 1$, we have by~\eqref{eq:set_shortening} that $\mathcal{I}_1 = \{\{1\},\{2\}\}$, so that Eq.~\eqref{eq:cardinal_pair_sets} is true for all $(\ell,\ell')\in \llbracket1,2\rrbracket^2$ such that $\ell\neq \ell'$. When $k = 2$, the expression of $\mathcal{I}_2$ is given in Example~\ref{ex:shortening_set}, and we see that whatever the values of \hbox{$(\ell,\ell')\in \llbracket1,4\rrbracket^2$} verifying $\ell \neq \ell' \text{ mod } 2$, Eq.~\eqref{eq:cardinal_pair_sets} holds. We thus now focus in the case where $k\geq 3$. We consider $(\ell,\ell')\in \llbracket1,2k\rrbracket^2$ such that $\ell \neq \ell' \text{ mod } k$. Following~\hbox{\cite[Section~$4.4.3$]{olaye_long-time_2026}}, one can prove that $f : \{0,1\}^{k-2}\rightarrow \left\{I\in\mathcal{I}_k\,|\,\ell\in I,\,\ell'\in I\right\}$, defined for all $x\in\{0,1\}^{k-2}$ as 
	$$
	f(x) = \left\{kx_j+j\,|\,j\in\llbracket1,k\rrbracket,\,j \neq \ell \text{ mod }k,\, j \neq \ell' \text{ mod }k\right\}\cup\{\ell,\ell'\}
	$$
	is bijective. Hence, as $\#\left(\{0,1\}^{k-2}\right) = 2^{k-2}$, we obtain that Eq.~\eqref{eq:cardinal_pair_sets} is true when $k\geq3$, which concludes the proof of the lemma. 
\end{proof}

\subsection{Auxiliary statements related to Erlang distributions}\label{subsect:auxiliary_statements_erlang_distributions}
Throughout this section, we denote for all $\beta >0$ the following constants:
\begin{equation}\label{eq:approximation_beta}
\beta_{N,1} = \frac{N}{m_1}\left[1 - \mathcal{L}(g)\left(\frac{\beta}{N}\right)\right], \hspace{6.5mm}\beta_{N,2k} =\frac{N}{km_1}\left[1 - \mathcal{L}(\mu)\left(\frac{\beta}{N}\right)\right].
\end{equation}
These constants are similar to the ones introduced in~\eqref{eq:approximation_eigenvalues} ($\beta_{N,2k}$ is similar to $\lambda'_N$ and~$\omega'_N$ by~\eqref{eq:laplace_transform_mu}), and correspond to approximations of $\beta$ when $N$ is large.

The first statement we present is useful in the specific case where $n_0$ is an exponential distribution, i.e. an Erlang distribution with parameter $\ell = 1$. It gives the explicit formula of the estimator~$\widehat{n}_0^{(d)}$, for all $d\in\{1,2k\}$.
\begin{prop}[Explicit estimators for exponential initial distributions]\label{prop:explicit_solutions_h1beta}
Assume that there exists $\beta >0$ such that $n_0 = h_{1,\beta}$, defined in~\eqref{eq:density_erlang_distribution}. Then, for all $d\in\{1,2k\}$ and $x\geq 0$, we have 
\begin{equation}\label{eq:explicit_estimator_h1beta}
\widehat{n}_0^{(d)}(x)  = \beta_{N,d}\exp\left[-\beta_{N,d}x\right].
\end{equation}
\end{prop}
\begin{proof}
The proof of this proposition when $d = 1$ follows the same steps as when $d=2k$, but is easier. Hence, we only give the proof when $d=2k$. Notice that by~\eqref{eq:PDE_model_telomeres_several_telos},~\eqref{eq:explicit_solution_cos_expo} with $\omega = 0$ and~\eqref{eq:approximation_beta}, we have that $n^{(2k)}(t,x) = \beta^{2k}\exp\left[-kbm_1 \beta_{N,2k}t - \beta \sum_{i = 1}^{2k}x_i\right]$ for all~\hbox{$t\geq0$},~\hbox{$x\in\mathbb{R}_+^{2k}$}. Our aim is to use this expression of $n^{(2k)}$ to obtain~\eqref{eq:explicit_estimator_h1beta}. To do so, we first rewrite $\mu\left(\left\{v\in\mathbb{R}_+^{2k}\,|\,Ny-v \notin \mathbb{R}_+^{2k}\right\}\right)$ in the second line of~\eqref{eq:rescaled_PDE_model_telomeres_several_telos} with an integral. Then, we switch the integrals and use the equality~$1_{\{Ny - v\notin\mathbb{R}_+^{2k}\}} = 1 - 1_{\left\{\forall i\in\llbracket1,2k\rrbracket:\,y_i \geq \frac{v_i}{N}\right\}}$. Finally, we apply the expression of $n^{(2k)}$ and~Eq.~\eqref{eq:approximation_beta}. We obtain for all~$t\geq0$  
\begin{equation}\label{eq:cemetery_exponential}
\begin{aligned}
	n_{\partial}^{(2k)}(t) &= bN\int_{y\in\mathbb{R}_+^{2k}} n^{(2k)}\left(t,y\right)\left[\int_{v\in\mathbb{R}_+^{2k}}1_{\left\{Ny-v \notin \mathbb{R}_+^{2k}\right\}}\mu(\dd v)\right]\mathrm{d}y \\
	&= bN\int_{v\in\mathbb{R}_+^{2k}}  \left[\int_{y\in\mathbb{R}_+^{2k}}    n^{(2k)}\left(t,y\right)\left(1 -1_{\left\{\forall i\in\llbracket1,2k\rrbracket:\,y_i \geq \frac{v_i}{N}\right\}}\right)\mathrm{d}y\right]\mu(\dd v) \\
	&=  bN\left[1 - \mathcal{L}(\mu)\left(\frac{\beta}{N}\right)\right]\exp\left[-kbm_1 \beta_{N,2k}t\right] = kbm_1 \beta_{N,2k}\exp\left[-kbm_1 \beta_{N,2k}t\right].
\end{aligned}
\end{equation}
Therefore, by plugging the above in~\eqref{eq:definitions_estimators}, and then using~\eqref{eq:equality_transport_terms} to replace $bm_1$ with $\tilde{b}\tilde{m}_1$, we obtain~\eqref{eq:explicit_estimator_h1beta}
\end{proof}
\noindent The second proposition we present is devoted to the more general case in which $n_0$ is an Erlang distribution, but not necessarily an exponential distribution. In this case, $n^{(1)}$ is explicit. One can then compute $n_{\partial}^{(1)}$ by using~\eqref{eq:PDE_model_telomeres_one_telo}, and then $\widehat{n}_{0}^{(1)}$ thanks to~\eqref{eq:definitions_estimators}. 
\begin{prop}[Explicit solutions for Erlang initial distributions]\label{prop:explicit_solutions_gamma}
	Assume that there exist~\hbox{$\ell\in\mathbb{N}^*$} and $\beta > 0$ such that $n_0 = h_{\ell,\beta}$, defined in~\eqref{eq:density_erlang_distribution}. We consider for all $(t,x)\in\mathbb{R}_+\times\mathbb{R}_+$ the function $\psi_{x,t}\in L^1\left(\mathbb{R}_+^*\right)$, defined for all $\alpha >0$ as $\psi_{x,t}(\alpha)  := -bN\left[1 - \mathcal{L}(g)\left(\frac{\alpha}{N}\right)\right]t - \alpha x$. Then, for all $(t,x)\in\mathbb{R}_+\times\mathbb{R}_+$, we have 
	\begin{equation}\label{eq:explicit_solutions_gamma_first}
		n^{(1)}(t,x)= (-1)^{\ell -1}\frac{\beta^\ell}{(\ell-1)!}B_{\ell-1}\Bigg[\frac{\dd}{\dd \alpha}\psi_{x,t}\left(\beta\right),\hdots, \frac{\dd^{\ell-1}}{\dd \alpha^{\ell-1}}\psi_{x,t}\left(\beta\right)\Bigg]\exp\left(\psi_{x,t}(\beta)\right), 
	\end{equation} 
	where $B_{\ell-1}$ is a complete Bell polynomial of order $\ell-1$, see~\cite[Eq. 3.c, p.~134]{comtet_advanced_1974}.
\end{prop}
\begin{proof}
We denote $f\in L^1\left(\mathbb{R}_+^*\times\mathbb{R}_+^2\right)$ the function defined for all $(\alpha,t,x)\in\mathbb{R}_+^*\times\mathbb{R}_+^2$ as $f(\alpha,t,x) := \exp\left[\psi_{x,t}\left(\alpha\right)\right]$. One can easily see that by integrating $f$ and computing its derivative, we have for all~$(\alpha,t,x)\in\mathbb{R}_+^*\times\mathbb{R}_+^2$
	\begin{equation}\label{eq:proof_prop_explicit_solutions_gamma_intermediate_first}
		\partial_t f(\alpha,t,x) = bN\int_0^{\delta}\left[f(\alpha,t,x+u)- f(\alpha,t,x)\right]g(u)\dd u .
	\end{equation}
	We also know by the Faà di Bruno's formula~\cite[Theorem C, p.~$139$]{comtet_advanced_1974} and the definition of a complete Bell polynomial~\cite[Eq. 3.c, p.~$134$]{comtet_advanced_1974}, that it holds for all~\hbox{$(\alpha,t,x)\in\mathbb{R}_+^*\times\mathbb{R}_+^2$}
	\begin{equation}\label{eq:proof_prop_explicit_solutions_gamma_intermediate_second}
		\frac{\dd^{\ell-1}}{\dd \alpha^{\ell-1}} f(\alpha,t,x)  = B_{\ell-1}\Bigg[\frac{\dd}{\dd \alpha}\psi_{x,t}\left(\alpha\right),\hdots, \frac{\dd^{\ell-1}}{\dd \alpha^{\ell-1}}\psi_{x,t}\left(\alpha\right)\Bigg]\exp\left(\psi_{x,t}(\alpha)\right).
	\end{equation}
	Then, by using Eq.~\eqref{eq:proof_prop_explicit_solutions_gamma_intermediate_second} and Eq.~\eqref{eq:proof_prop_explicit_solutions_gamma_intermediate_first}, we have that the function $\varphi_1\in L^1\left(\mathbb{R}_+^2\right)$ defined for all~\hbox{$(t,x)\in\mathbb{R}_+\times\mathbb{R}_+$} as \hbox{$\varphi_1(t,x) := (-1)^{\ell -1}\frac{\beta^\ell}{(\ell-1)!}\frac{\dd^{\ell-1}}{\dd \alpha^{\ell-1}} f(\beta,t,x)$} is equal to the right-hand side of~\eqref{eq:explicit_solutions_gamma_first}, and verifies the same equation as $n^{(1)}$, see the first line of~\eqref{eq:rescaled_PDE_model_telomeres_one_telo}. In addition, we have that $\varphi_1(0,x) = \frac{\beta^\ell }{(\ell-1)!}x^{\ell -1}e^{-\beta x} = h_{\ell,\beta}(x)$ for all $x\geq0$, in view of the fact that by Eq.~\eqref{eq:proof_prop_explicit_solutions_gamma_intermediate_second}, \hbox{\cite[Eq. $3.n'$, p.~$136$]{comtet_advanced_1974}}, and~\cite[Eq. $3.c$, p.~$134$]{comtet_advanced_1974},  it holds
	$$
	\frac{\dd^{\ell-1}}{\dd \alpha^{\ell-1}} f(\beta,0,x)  = B_{\ell-1}\left[-x,0,\hdots,0\right]\exp\left(-\beta x\right) = \left(-x\right)^{\ell - 1}\exp\left(-\beta x\right).
	$$
	As the first line of~\eqref{eq:rescaled_PDE_model_telomeres_one_telo} has a unique solution in $C\left(\mathbb{R},L^1\left(\mathbb{R}_+^2\right)\right)$ with initial condition $h_{\ell,\beta}$ by Proposition~\ref{prop:well_definition_general_model}, we obtain that~\hbox{$\varphi_1 = n^{(1)}$}, so that Eq.~\eqref{eq:explicit_solutions_gamma_first} is~true.
\end{proof}

\noindent The final proposition we present provides the value of the mean, the variance, and the coefficient of variation of an Erlang distribution. This result is quite standard, so we do not prove it. We however refer to~\cite[p.~$138$]{Ibe_2014} for a proof of the expression of the mean and the variance.
\begin{prop}[Moments of Erlang distributions]\label{prop:moments_erlang_distribution}
	Let us fix $\ell\in\mathbb{N}^*$ and $\beta > 0$. We consider
	$$
	m = \int_{0}^{+\infty} xh_{\ell,\beta}(x) \dd x\hspace{2mm} \text{ and } \hspace{2mm}\sigma^2 = \int_{0}^{+\infty} x^2h_{\ell,\beta}(x) \dd x - m^2,
	$$
	which correspond respectively to the mean and the variance of a random variable distributed according to $h_{\ell,\beta}$. We also consider $cv = \frac{\sigma}{\mu}$ the coefficient of variation of this distribution. Then, we have 
	$$
		m = \frac{\ell}{\beta}, \hspace{8mm} \sigma^2 = \frac{\ell}{\beta^2}, \hspace{4mm} \text{and}\hspace{4mm} cv = \frac{1}{\sqrt{\ell}}.
	$$
\end{prop}




\section{Proof of Propositions~\ref{prop:quality_estimator_random_variables_onetelo} and~\ref{prop:quality_estimator_random_variables_severaltelos}}\label{sect:impact_noise}

This section is devoted to the proof of Propositions~\ref{prop:quality_estimator_random_variables_onetelo} and~\ref{prop:quality_estimator_random_variables_severaltelos}. First, in Section~\ref{subsect:notations_impact_noise}, we introduce the notations and auxiliary results we need to obtain these propositions. Then, in Section~\ref{subsect:proof_practical_errors_estimators}, we prove one of these auxiliary results. Finally, in Section~\ref{subsect:proof_prop_random_variables_severaltelos}, we detail the proof of Proposition~\ref{prop:quality_estimator_random_variables_severaltelos} from this statement. The proof of Proposition~\ref{prop:quality_estimator_random_variables_onetelo} follows exactly the same steps as those given in Section~\ref{subsect:proof_prop_random_variables_severaltelos}, and is even slightly easier. Hence, we do not detail~it.

\subsection{Notations and auxiliary statements}\label{subsect:notations_impact_noise}

We begin by introducing the notations we use in this section. We first consider the following~set
$$
\mathcal{X}^{\infty}\left(\mathbb{R}\right) = \left\{f \in\mathcal{D}'(\mathbb{R})\,|\, \exists \varphi\in L^\infty(\mathbb{R}) \text{ s.t. } f = \varphi'\right\}.
$$
This set contains distributions that are not necessarily representable by a function. For example, as for all $x\in\mathbb{R}$ and $x_0\in \mathbb{R}$ it holds $\frac{\dd}{\dd x}\left(1_{x \leq x_0}\right) = \delta_{x_0}(\{x\})$, and as the shifted Heaviside functions are in $L^{\infty}\left(\mathbb{R}\right)$, we have $\delta_{x_0} \in \mathcal{X}^{\infty}\left(\mathbb{R}\right)$. We then introduce the functional~$\mathcal{K}:  \mathcal{X}^{\infty}\left(\mathbb{R}\right) \rightarrow \mathbb{R}_+$, defined for all $f\in \mathcal{X}^{\infty}\left(\mathbb{R}\right)$~as
\begin{equation}\label{eq:definition_norm_irregular_space}
\mathcal{K}(f) = \inf_{\varphi\in L^\infty(\mathbb{R})\text{ s.t. }f=\varphi'}\left( \left|\left|\varphi\right|\right|_{L^\infty(\mathbb{R})}\right).
\end{equation}
This functional allows us to control the distance between two distributions in $\mathcal{X}^{\infty}\left(\mathbb{R}\right)$ by using two functions. For example, we can control the distance between two Dirac measures by two shifted Heaviside functions. We finally define the following functions, for all \hbox{$d\in\{1,2k\}$, $y\in\mathbb{R}$}, 
\begin{equation}\label{eq:logarithm_theoretical}
	\begin{aligned}
		n_{\text{log}}^{(d)}(y) = \exp\left(y\right)\widehat{n}_0^{(d)}\left(\exp\left(y\right)\right),
	\end{aligned}
\end{equation}
and the following measures 
\begin{equation}\label{eq:logarithm_empirical}
	\begin{aligned}
		\overline{n}_{\text{log}}^{(1)}(\dd x) &= \frac{1}{n_{s}}\sum_{i = 1}^{n_{s}} \delta_{\log\left(\tilde{b}\tilde{m}_1T_{1,i}\right)}\left(\dd x\right),\\ 
		\overline{n}_{\text{log}}^{(2k)}(\dd x) &= \sum_{i = 0}^{n_s}\left(1-\frac{i}{n_s}\right)^{\frac{1}{2k}}\left[\delta_{\log\left(\frac{\tilde{b}\tilde{m}_1}{2}T_{2k,i+1}\right)}\left(\dd x\right) -\delta_{\log\left(\frac{\tilde{b}\tilde{m}_1}{2}T_{2k,i}\right)}\left(\dd x\right)\right], 
	\end{aligned}
\end{equation}
where  $\left(T_{1,i}\right)_{1 \leq i \leq n_s}$ and~$\left(T_{2k,i}\right)_{0 \leq i \leq n_s+1}$ are the sequences of random variables introduced in Sections~\ref{subsubsubsect:estimation_onetelo_probabilistic} and~\ref{subsubsect:estimator_severaltelos_probabilistic}. Since Dirac measures belong to $\mathcal{X}^{\infty}\left(\mathbb{R}\right)$, we have that $\overline{n}_{\text{log}}^{(d)} \in \mathcal{X}^{\infty}\left(\mathbb{R}\right)$ for all $d\in\{1,2k\}$. The function $n_{\text{log}}^{(d)}$, defined in~\eqref{eq:logarithm_theoretical}, is the distribution of the random variable $\log\left(Z^{(d)}\right)$, where $Z^{(d)}$ is distributed according to $\widehat{n}_0^{(d)}$. The measure $\overline{n}_{\text{log}}^{(1)}$, in the first line of~\eqref{eq:logarithm_empirical}, is the empirical estimator of~$n_{\text{log}}^{(1)}$. The measure~$\overline{n}_{\text{log}}^{(2k)}$, in the next line, is an estimator of $n_{\text{log}}^{(2k)}$. The latter has been constructed by computing the weak derivative of the empirical cumulative distribution function associated to $n_{\text{log}}^{(2k)}$, as done in the proof of Proposition~\ref{prop:derivative_power_empirical}.

We now present the two auxiliary statements that we need to prove Propositions~\ref{prop:quality_estimator_random_variables_onetelo} and~\ref{prop:quality_estimator_random_variables_severaltelos}. The first statement we provide is the following. It corresponds to an intermediate statement to obtain the second one. We briefly sketch its proof below but do not detail it, as the inequalities used in the proof are relatively classical, see~\cite{armiento_estimation_2016,bourgeron_estimating_2014,doumic_estimating_2013}.
\begin{lemm}[Bounds on the logarithm of the data]\label{lemm:bound_log_from_armiento}
Let us consider $d\in\{1,2k\}$, and~\hbox{$\alpha > 0$}. Assume that \hyperlink{assumption:H1}{$(H_1)-(H_4)$} hold, and that $\beta'_N = \left(\lambda + 2k\lambda'_N\right) - \left(\omega + (2k-1)\omega'_N\right) > 0$ when $d=2k$.  We denote for all~$x\in\mathbb{R}$: $\rho_{\alpha}(x) = \frac{1}{\alpha}\rho\left(\frac{x}{\alpha}\right)$. Then, it holds
\begin{equation}\label{eq:bound_log_from_armiento}
\left|\left|\rho_{\alpha}*\overline{n}_{\log}^{(d)} - n_{\log}^{(d)}\right|\right|_{L^{\infty}\left(\mathbb{R}\right)} \leq \alpha^{-1} \sqrt{\frac{2}{\pi}}\mathcal{K}\left(\overline{n}_{\log}^{(d)} - n_{\text{log}}^{(d)}\right) + \alpha \sqrt{\frac{2}{\pi}} \left|\left|\left(n_{\log}^{(d)}\right)'\right|\right|_{L^{\infty}\left(\mathbb{R}\right)}.
\end{equation}
\end{lemm}
\begin{proof}
First, notice that by the triangle inequality, it holds
\begin{equation}\label{eq:bound_log_from_armiento_intermediate_first}
\left|\left|\rho_{\alpha}*\overline{n}_{\log}^{(d)} - n_{\log}^{(d)}\right|\right|_{L^{\infty}\left(\mathbb{R}\right)} \leq \left|\left|\rho_{\alpha}*\left(\overline{n}_{\log}^{(d)} - n_{\log}^{(d)}\right)\right|\right|_{L^{\infty}\left(\mathbb{R}\right)} + \left|\left|\rho_{\alpha}*n_{\log}^{(d)} - n_{\log}^{(d)}\right|\right|_{L^{\infty}\left(\mathbb{R}\right)}.
\end{equation}
Then, notice that by adapting the proof of \hbox{\cite[Lemmas A.$1.(1)$ and A.$1.(4)$]{bourgeron_estimating_2014}},  we can obtain similar inequalities as in~\cite[Lemmas $3.iii)$ and $3.i)$]{armiento_estimation_2016}, for all $\varphi\in L^\infty(\mathbb{R})$ such that $\overline{n}_{\log}^{(d)} - n_{\log}^{(d)} = \varphi'$ and~$f\in W^{1,\infty}\left(\mathbb{R}\right)$, 
\begin{equation}\label{eq:bound_log_from_armiento_intermediate_second}
\begin{aligned}
\left|\left|\rho_{\alpha}*\left(\overline{n}_{\log}^{(d)} - n_{\log}^{(d)}\right)\right|\right|_{L^{\infty}\left(\mathbb{R}\right)} &= \left|\left|\rho_{\alpha}*\varphi'\right|\right|_{L^{\infty}\left(\mathbb{R}\right)} \leq \alpha^{-1}\left|\left|\rho'\right|\right|_{L^{1}(\mathbb{R})}\left|\left|\varphi\right|\right|_{L^{\infty}(\mathbb{R})},\\
\left|\left|\rho_{\alpha}*f - f\right|\right|_{L^{\infty}\left(\mathbb{R}\right)}&\leq \alpha\left|\left|\text{Id}.\rho\right|\right|_{L^{1}(\mathbb{R})}\left|\left|f'\right|\right|_{L^{\infty}(\mathbb{R})}.
\end{aligned}
\end{equation}
In view of~\eqref{eq:definition_norm_irregular_space}, by taking the infimum in the first line of~\eqref{eq:bound_log_from_armiento_intermediate_second}, we have that
$$
\left|\left|\rho_{\alpha}*\left(\overline{n}_{\log}^{(d)} - n_{\log}^{(d)}\right)\right|\right|_{L^{\infty}\left(\mathbb{R}\right)} \leq \alpha^{-1}\left|\left|\rho'\right|\right|_{L^{1}(\mathbb{R})}\mathcal{K}\left(\overline{n}_{\log}^{(d)} - n_{\text{log}}^{(d)}\right).
$$
Therefore, by first plugging the above and the second line of~\eqref{eq:bound_log_from_armiento_intermediate_second} in~\eqref{eq:bound_log_from_armiento_intermediate_first}, and then using that~$\left|\left|\rho'\right|\right|_{L^{1}(\mathbb{R})} = \left|\left|\text{Id}.\rho\right|\right|_{L^{1}(\mathbb{R})} = \sqrt{\frac{2}{\pi}}$ as $\rho$ is the density of a standard Gaussian distribution, we obtain~\eqref{eq:bound_log_from_armiento}. 
\end{proof}
\noindent The second statement we give is the following. It provides for all $d\in\{1,2k\}$ a bound on the supremum error between $\overline{n}_0^{(d,\alpha)}$ and $n_0$, when the error between $\overline{n}_{\text{log}}^{(d)}$ and $n_{\text{log}}^{(d)}$ is known. It is proved in Section~\ref{subsect:proof_practical_errors_estimators}. 
\begin{lemm}[Error bounds for simulated data]\label{lemm:practical_errors_estimators}
We recall the constants $C_{\widehat{n},1}$ and $C_{\widehat{n},2k}$ defined in~\eqref{eq:constant_randomvar_onetelo} and~\eqref{eq:constant_randomvar_severaltelos} respectively. Assume that the assumptions of Lemma~\ref{lemm:bound_log_from_armiento} hold, and there exists~$\varepsilon > 0$ such that~$\mathcal{K}\left(\overline{n}_{\text{log}}^{(d)} - n_{\text{log}}^{(d)}\right) \leq \varepsilon$. Then, there exists a sequence $\left(L_{d,n}\right)_{n\in\mathbb{N}}$ of positive real numbers such that \hbox{$\underset{n\rightarrow+\infty}{\limsup}\, L_{d,n} < +\infty$} and 
\vspace{-1.5mm}
\begin{equation}\label{eq:practical_errors_estimators}
\left|\left|\text{Id}\left[\overline{n}_0^{(d,\alpha)} - n_0\right]\right|\right|_{L^\infty\left(\mathbb{R}_+\right)} \leq \alpha^{-1} \sqrt{\frac{2}{\pi}}\varepsilon + \alpha \sqrt{\frac{2}{\pi}} C_{\widehat{n},d}  + \frac{L_{d,N}}{N}.
\end{equation}
\end{lemm}
\begin{rem}\label{rem:minimum_smoothing_parameter}
The minimum of the function $\alpha\in\mathbb{R}_+^* \mapsto \alpha^{-1}\sqrt{\frac{2}{\pi}}\varepsilon + \alpha\sqrt{\frac{2}{\pi}} C_{\widehat{n},d}$ can be computed by analysing the sign of its derivative, and is in \hbox{$\alpha^* = \left(\frac{\varepsilon}{C_{\widehat{n},d}}\right)^{\frac{1}{2}}$}. Then, in view of~\eqref{eq:practical_errors_estimators}, the smoothing parameter providing the best bound on the error is~$\alpha^*$. 
\end{rem}
\noindent We now conclude this section by proving Lemma~\ref{lemm:practical_errors_estimators} and Proposition~\ref{prop:quality_estimator_random_variables_severaltelos}.


\subsection{Proof of Lemma~\ref{lemm:practical_errors_estimators}}\label{subsect:proof_practical_errors_estimators}
First, decomposing $\left|\left|\text{Id}\left[\overline{n}_0^{(d,\alpha)} - n_0\right]\right|\right|_{L^\infty\left(\mathbb{R}_+\right)}$ into two terms thanks to the triangle inequality, and then using~Theorem~\ref{te:main_result} to bound the second term, yields that there exists a sequence~$\left(L_{d,n}\right)_{n\in\mathbb{N}}$ of positive real numbers such that $\underset{n\rightarrow+\infty}{\limsup}\, L_{d,n} < +\infty$ and
\begin{equation}\label{eq:prop_impact_noise_intermediate_zero}
		\left|\left|\text{Id}\left[\overline{n}_0^{(d,\alpha)} - n_0\right]\right|\right|_{L^\infty\left(\mathbb{R}_+\right)} \leq \left|\left|\text{Id}\left[\overline{n}_0^{(d,\alpha)} - \widehat{n}_0^{(d)}\right]\right|\right|_{L^\infty\left(\mathbb{R}_+\right)} + \frac{L_{d,N}}{N}.
	\end{equation}
	We thus only have to bound the term $\left|\left|\text{Id}\left[\overline{n}_0^{(d,\alpha)} - \widehat{n}_0^{(d)}\right]\right|\right|_{L^\infty\left(\mathbb{R}_+\right)}$, and~\eqref{eq:practical_errors_estimators} will be true. To do so, notice that by the change of variable $y = \log(x)$, the definition of~$\overline{n}_0^{(1,\alpha)}$ and~$\overline{n}_0^{(2k,\alpha)}$ given in \eqref{eq:estimators_simulations_onetelo}-\eqref{eq:estimators_simulations_several_telos},  and the ones of $n_{\log}^{(d)}$ and $\overline{n}_{\log}^{(d)}$ given in~\eqref{eq:logarithm_theoretical}-\eqref{eq:logarithm_empirical}, we have (recall that~$\rho_{\alpha} = \frac{1}{\alpha}\rho\left(\frac{.}{\alpha}\right)$) 
	
	$$
		\begin{aligned}
			\left|\left|\text{Id}\left[\overline{n}_0^{(d,\alpha)} - \widehat{n}_0^{(d)}\right]\right|\right|_{L^\infty\left(\mathbb{R}_+\right)} &= \sup_{x\in\mathbb{R}_+}\left[x\left|\overline{n}_0^{(d,\alpha)}\left(x\right)- \widehat{n}_0^{(d)}\left(x\right)\right|\right]\\ 
			&= \sup_{y\in\mathbb{R}}\left[\exp\left(y\right)\left|\overline{n}_0^{(d,\alpha)}\left(\exp\left(y\right)\right)- \widehat{n}_0^{(d)}\left(\exp\left(y\right)\right)\right|\right]\\
			&=\left|\left|\rho_{\alpha}*\overline{n}_{\log}^{(d)} - n_{\log}^{(d)}\right|\right|_{L^{\infty}\left(\mathbb{R}\right)}.
		\end{aligned}
	$$
	Therefore, by first using Lemma~\ref{lemm:bound_log_from_armiento}, and then the fact that $\mathcal{K}\left(\overline{n}_{\text{log}}^{(d)} - n_{\text{log}}^{(d)}\right)\leq \varepsilon$, we obtain
	$$
	\left|\left|\text{Id}\left[\overline{n}_0^{(d,\alpha)} - \widehat{n}_0^{(d)}\right]\right|\right|_{L^\infty\left(\mathbb{R}_+\right)} \leq \alpha^{-1} \sqrt{\frac{2}{\pi}}\varepsilon + \alpha \sqrt{\frac{2}{\pi}} \left|\left|\left(n_{\log}^{(d)}\right)'\right|\right|_{L^{\infty}\left(\mathbb{R}\right)}.
	$$
	We thus only have to plug the above in Eq.~\eqref{eq:prop_impact_noise_intermediate_zero}, and then using the following equality which comes from~\eqref{eq:logarithm_theoretical}, the change of variable $x = \log(y)$, and~\eqref{eq:constant_randomvar_onetelo}-\eqref{eq:constant_randomvar_severaltelos}, to conclude the proof of the lemma
	$$
	\begin{aligned}
		\left|\left|\left(n_{\log}^{(d)}\right)'\right|\right|_{L^\infty\left(\mathbb{R}\right)} &= \sup_{y\in\mathbb{R}}\left|\exp\left(2y\right)\left(\widehat{n}_0^{(d)}\right)'\left(\exp\left(y\right)\right) + \exp\left(y\right)\widehat{n}_0^{(d)}\left(\exp\left(y\right)\right)\right| \\
		&= \sup_{x\in\mathbb{R}_+}\left|x^2\left(\widehat{n}_0^{(d)}\right)'\left(x\right) + x\widehat{n}_0^{(d)}\left(x\right)\right| = C_{\widehat{n},d}. \mathrlap{\phantom{xxxxxxxxxxxxxxxx}\hspace{1.14mm}\qed}
	\end{aligned}
	$$ 
\subsection{Proof of Proposition~\ref{prop:quality_estimator_random_variables_severaltelos}}\label{subsect:proof_prop_random_variables_severaltelos} 

Our objective here is to apply Lemma~\ref{lemm:practical_errors_estimators}. To do so, we begin by computing the probability that $\mathcal{K}\left(\overline{n}_{\text{log}}^{(2k)} - n_{\text{log}}^{(2k)}\right) \leq \varepsilon_p$, where~\hbox{$\varepsilon_p  = \left(\frac{\log\left(\frac{2}{p}\right)}{2n_s}\right)^{\frac{1}{4k}}$}. We denote for all~$y\in\mathbb{R}$
	\begin{equation}\label{eq:definition_R_function}
		R(y) := \left[1-\frac{1}{n_s}\sum_{i = 1}^{n_s} 1_{\left\{\log\left(\frac{\tilde{b}\tilde{m}_1}{2}T_{2k,i}\right) \leq y\right\}}\right]^{\frac{1}{2k}} - \int_{y}^{+\infty} n_{\log}^{(2k)}(s)\dd s,
	\end{equation}
	and for all $t\geq0$: $N_{\partial}(t) = \int_{t}^{\infty} n_{\partial}^{(2k)}(s) \dd s$. In view of~\eqref{eq:logarithm_theoretical}, the change of variable $s' = \exp(s)$, and the fact that $\widehat{n}_0^{(2k)}$ is the derivative of $x\mapsto -\left(N_{\partial}\right)^{\frac{1}{2k}}\left(\frac{2 x}{\tilde{b}\tilde{m}_1}\right)$ (see~\eqref{eq:definitions_estimators}), we have for all~\hbox{$y\in\mathbb{R}$} 
	\begin{equation}\label{eq:proof_cor_confident_interval_intermediate_first}
		\int_{y}^{+\infty} n_{\log}^{(2k)}(s)\dd s = \int_{\exp(y)}^{+\infty} \widehat{n}_0^{(2k)}(s')\dd s' = N_{\partial}\left(\frac{2e^y}{\tilde{b}\tilde{m}_1}\right)^{\frac{1}{2k}}.
	\end{equation}
	Our aim is to use the above equality to bound the probability that~\hbox{$\sup_{y\in\mathbb{R}}\left|R(y)\right| > \varepsilon_p$}. To do so, we first use that $\left|a^{\frac{1}{2k}}-c^{\frac{1}{2k}}\right| \leq |a - c|^{\frac{1}{2k}}$ for all $a,\,c\geq0$ to bound from above~$|R|$. This inequality is a consequence of the reverse triangle inequality that holds for the quasi-distance~\hbox{$D(a,c) := |a-c|^{\frac{1}{2k}}$}, where $(a,c)\in\mathbb{R}^2$. Then, we apply~\eqref{eq:proof_cor_confident_interval_intermediate_first} to simplify the expression in the measure. Finally, we use the \hbox{Dvoretzky–Kiefer–Wolfowitz–Massart inequality}~\cite{dvoretzky_asymptotic_1956,massart_tight_1990}, in view of the fact that~\hbox{$x\mapsto 1- N_{\partial}\left(\frac{2e^x}{\tilde{b}\tilde{m}_1}\right)$} is the cumulative distribution function of the random variables $\left(\log\left(\frac{\tilde{b}\tilde{m}_1}{2}T_{2k,i}\right)\right)_{1\leq i \leq n_s}$ (as $1 - N_{\partial}$ is the one of $\left(T_{2k,i}\right)_{1\leq i \leq n_s}$). We obtain 
	\begin{equation}\label{eq:bound_massart}
		\begin{aligned}
			\mathbb{P}\left[\sup_{y\in\mathbb{R}}\left|R(y)\right| > \varepsilon_p\right] &\leq \mathbb{P}\left[\sup_{y\in\mathbb{R}}\left|1 -\frac{1}{n_s}\sum_{i = 1}^{n_s} 1_{\left\{\log\left(\frac{\tilde{b}\tilde{m}_1}{2}T_{2k,i}\right) \leq y\right\}}- \left[\int_{y}^{+\infty} n_{\log}^{(2k)}(s)\dd s\right]^{2k}\right|^{\frac{1}{2k}} >  \varepsilon_p\right]\\ 
			&= \mathbb{P}\left[\sup_{y\in\mathbb{R}}\left|\left[1-N_{\partial}\left(\frac{2e^y}{\tilde{b}\tilde{m}_1}\right)\right] -\frac{1}{n_s}\sum_{i = 1}^{n_s} 1_{\left\{\log\left(\frac{\tilde{b}\tilde{m}_1}{2}T_{2k,i}\right) \leq y\right\}}\right| >  \left(\varepsilon_p\right)^{2k}\right] \\ 
			&\leq 2e^{-2n_s(\varepsilon_p)^{4k}}.
		\end{aligned}
	\end{equation}
	In addition, in view of~\eqref{eq:definition_R_function},~\eqref{eq:logarithm_empirical}, and the reasoning that allowed us to obtain~Proposition~\ref{prop:derivative_power_empirical}, we have that $R' =  n_{\log}^{(2k)} - \overline{n}_{\text{log}}^{(2k)}$. This yields, in view of the fact that $R\in L^{\infty}\left(\mathbb{R}\right)$ and the infimum in~\eqref{eq:definition_norm_irregular_space}, that it holds $\mathcal{K}\left(n_{\text{log}}^{(2k)} - \overline{n}_{\text{log}}^{(2k)}\right)\leq ||R||_{L^{\infty}\left(\mathbb{R}\right)}$. Then, by combining this inequality with~\eqref{eq:bound_massart}, we obtain that 
	$$
	\mathbb{P}\left[\mathcal{K}\left(n_{\text{log}}^{(2k)} - \overline{n}_{\text{log}}^{(2k)}\right)\leq  \varepsilon_p\right] \geq \mathbb{P}\left[\sup_{y\in\mathbb{R}}\left|R(y)\right| \leq  \varepsilon_p\right] \geq  1 - 2e^{-2n_s(\varepsilon_p)^{4k}} = 1-p.
	$$
	We thus only have to combine the above with Lemma~\ref{lemm:practical_errors_estimators} for $\varepsilon = \varepsilon_p $ and $\alpha = \alpha_p$ to conclude that~\eqref{eq:bounds_confidence_interval} is true with probability at least $1- p$ (we notice that $\alpha_p$ is of the form given in Remark~\ref{rem:minimum_smoothing_parameter}, so is optimal). \qed 

\paragraph{Acknowledgements.} This work was partially funded by the Fondation Mathématique Jacques Hadamard and by the European Union ERC-2024-COG MUSEUM-101170884. Additional support was provided by the France 2030 program, administered by the French National Research Agency (ANR), under grant ANR-23-EXMA-0005. Views and opinions expressed are however those of the author(s) only and do not necessarily reflect those of the European Union or the European Research Council Executive Agency (ERCEA). Neither the European Union nor the granting authority can be held responsible for them. The author warmly thanks Marie Doumic for her guidance during the realisation of this work, and her careful proofreading of the article. He also thanks Milica Tomašević for her careful proofreading of the introduction. Finally, he thanks the anonymous referee for their comments and questions that allow him to improve this manuscript.

\printbibliography

\end{document}